\documentclass[12pt,reqno]{amsart}
\usepackage{amsmath,amssymb,verbatim,comment,color}
\usepackage[pagebackref,pdftex,hyperindex]{hyperref}
\usepackage[T1]{fontenc}

\usepackage[margin=2.5cm]{geometry}
\usepackage{color}

\newcommand{\C}{\mathbb{C}}

\newcommand{\N}{\mathbb{N}}

\renewcommand{\P}{\mathbb{P}}

\newcommand{\R}{\mathbb{R}}

\newcommand{\Z}{\mathbb{Z}}

\newcommand{\fg}{\mathfrak{g}}

\newcommand{\fk}{\mathfrak{k}}

\newcommand{\cA}{\mathcal{A}}

\newcommand{\cC}{\mathcal{C}}
\newcommand{\cD}{\mathcal{D}}
\newcommand{\cE}{\mathcal{E}}

\newcommand{\cG}{\mathcal{G}}

\newcommand{\cK}{\mathcal{K}}

\newcommand{\cM}{\mathcal{M}}

\newcommand{\cO}{\mathcal{O}}
\newcommand{\cP}{\mathcal{P}}
\newcommand{\cQ}{\mathcal{Q}}

\newcommand{\cS}{\mathcal{S}}
\newcommand{\cT}{\mathcal{T}}
\newcommand{\cU}{\mathcal{U}}
\newcommand{\cV}{\mathcal{V}}

\renewcommand{\a}{\alpha}
\renewcommand{\b}{\beta}
\newcommand{\g}{\gamma}
\renewcommand{\d}{\delta}
\newcommand{\e}{\varepsilon}
\newcommand{\s}{\sigma}

\renewcommand{\i}{\sqrt{-1}}
\newcommand{\p}{\partial}
\newcommand{\bp}{\bar{\partial}}
\newcommand{\ddt}{\frac{d}{dt}}

\newcommand{\ie}{{\rm i.e.\ }} 

\renewcommand{\Re}{\mathrm{Re}}
\renewcommand{\Im}{\mathrm{Im}}

\DeclareMathOperator{\Aut}{Aut}
\DeclareMathOperator{\ch}{ch}

\DeclareMathOperator{\End}{End}
\DeclareMathOperator{\FS}{FS}
\DeclareMathOperator{\GL}{GL}

\DeclareMathOperator{\Gr}{Gr}
\DeclareMathOperator{\Herm}{Herm}
\DeclareMathOperator{\Hom}{Hom}
\DeclareMathOperator{\Id}{Id}
\DeclareMathOperator{\inte}{int}
\DeclareMathOperator{\Ker}{Ker}
\DeclareMathOperator{\Lie}{Lie}
\DeclareMathOperator{\MA}{MA}
\DeclareMathOperator{\op}{op}

\DeclareMathOperator{\rank}{rank}

\DeclareMathOperator{\SU}{SU}
\DeclareMathOperator{\Supp}{Supp}
\DeclareMathOperator{\Tr}{Tr}
\DeclareMathOperator{\Uni}{U}

\renewcommand{\leq}{\leqslant}
\renewcommand{\geq}{\geqslant}

\renewcommand{\hat}{\widehat}
\renewcommand{\tilde}{\widetilde}

\numberwithin{equation}{section}       

\theoremstyle{definition}
{
\newtheorem{prop} {Proposition} [section]
\newtheorem{thm}[prop] {Theorem} 
\newtheorem{dfn}[prop] {Definition}
\newtheorem{lem}[prop] {Lemma}
\newtheorem{cor}[prop]{Corollary}

\newtheorem{exam}[prop]{Example}
}
\theoremstyle{remark}
\newtheorem*{ackn}{\bf{Acknowledgment}} 
\newtheorem{rk}[prop]{Remark}

\title{$J$-equation on holomorphic vector bundles} 
\date{\today}

\author[R. Takahashi]{Ryosuke Takahashi}
\address{Mathematical Institute\\
 Tohoku University\\
6-3\\
Aramaki Aza-Aoba\\
Aoba-ku\\
Sendai\\
980-8578\\
 JAPAN}
\email{ryosuke.takahashi.e1@tohoku.ac.jp}

\subjclass[2020]{Primary 53C55; Secondary 35A01}
\keywords{$J$-equation, holomorphic vector bundle, deformed Hermitian--Yang--Mills equation}

\begin{document}
\maketitle
\begin{abstract}
We introduce the $J$-equation on holomorphic vector bundles over compact K\"ahler manifolds and investigate some fundamental properties as well as examples of solutions. In particular, we provide an algebraic condition called (asymptotic) $J$-stability in terms of subbundles on compact K\"ahler surfaces, and a numerical criterion on vortex bundles via dimensional reduction. Also, we discuss an application for the vector bundle version of the deformed Hermitian--Yang--Mills equation in the small volume regime.
\end{abstract}
\tableofcontents

\section{Introduction}
Let $(X,\omega)$ be an $n$-dimensional compact K\"ahler manifold and $L$ an ample line bundle on $X$. We say that a Hermitian metric $h$ on $L$ with positive curvature $F_h \in c_1(L)$ solves the {\it $J$-equation} if it satisfes
\begin{equation} \label{J-equation on line bundles}
\frac{[\omega] \cdot c_1(L)^{n-1}}{c_1(L)^n} F_h^n-\omega F_h^{n-1}=0.
\end{equation}
In general, the equivalence of the stability and the solvability of an equation is an important problem in geometry. One of the pioneering work in this direction is the resolution of the Kobayashi--Hitchin correspondence by Donaldson--Uhlenbeck--Yau \cite{Don85,UY86}. Inspired by this work, Donaldson \cite{Don03} proposed many equations, including the $J$-equation as above, by using the moment map interpretation. Simultaneously, Chen \cite{Che00} discovered the equation \eqref{J-equation on line bundles} in the study of constant scalar curvature K\"ahler (cscK) metrics and described a spectacular application of \eqref{J-equation on line bundles} to the existence problem of cscK metrics, by using the Chen--Tian formula for the Mabuchi K-energy.

Another motivation for studying \eqref{J-equation on line bundles} is the relation to the {\it deformed Hermitian--Yang--Mills (dHYM) equation}
\begin{equation} \label{dHYM equation on line bundles}
\Im \big( e^{-\i \Theta} (\omega+\i F_h)^n \big)=0,
\end{equation}
where $\Theta$ is a constant. As observed in \cite{CXY17}, one can obtain \eqref{J-equation on line bundles} as the {\it small volume limit} of the dHYM equation at least in a formal level, \ie replacing $\omega \mapsto \e \omega$, one can easily observe that, up to rescaling, the dHYM equation on $L$ is of the form
\[
\frac{[\omega] \cdot c_1(L)^{n-1}}{c_1(L)^n} F_h^n-\omega F_h^{n-1}=O(\e)
\]
as $\e \to 0$. In this sense, the dHYM equation is an asymptotic version of the $J$-equation. The equation \eqref{dHYM equation on line bundles} first appeared in the physics literature \cite{MMMS00}, and \cite{LYZ01} from mathematical side as the mirror object to a special Lagrangian in the setting of Strominger--Yau--Zaslow \cite{SYZ96} mirror symmetry. It is believed that the solvability of the dHYM equation is related to some stability notions in algebraic geometry, in particular, the stability in the sense of Bridgeland \cite{Bri07} as an object in the derived category of coherent sheaves $D^b {\rm Coh}(X)$ (for instance, see \cite{CS20} for more details). Meanwhile, the existence of solutions in the supercritical phase case $\Theta>(n-2)\frac{\pi}{2}$ is well understood by the Nakai--Moishezon type criterion (see \cite{Che21,CLT21,CJY20,DP21} for the recent progress).

For a higher rank holomorphic vector bundle $E$, Collins--Yau \cite{CY18} suggested the following natural extension of \eqref{dHYM equation on line bundles}
\begin{equation} \label{dHYM equation}
\Im \big(e^{-\i \Theta}(\omega \Id_E+\i F_h)^n \big)=0,
\end{equation}
as an equation for $\End(E)$-valued $(n,n)$-forms on $X$, where $\Theta$ is a constant and $F_h$ is the curvature of a Hermitian metric $h$ on $E$. As far as the author's knowledge, there are no non-trivial examples until the recent work \cite{DMS20}, in which they studied the {\it large volume limit} $\omega \mapsto k \omega$, and relates the solvability of \eqref{dHYM equation} with a certain asymptotic Bridgeland stability condition as $k \to \infty$ when $E$ is simple and sufficiently smooth\footnote{More generally, they considered the ``$Z$-critical connections'' including the dHYM connections as a special case.}. Correa \cite{Cor23} also constructed examples of dHYM connections on higher rank holomorphic vector bundles over $\P(T_{\P^2})$. In the large volume limit $k \to \infty$, up to rescaling, the equation \eqref{dHYM equation} formally converges to the Hermitian--Einstein equation
\[
\omega^{n-1} F_h-\frac{[\omega]^{n-1} \cdot c_1(E)}{[\omega]^n \cdot \rank E} \omega^n \Id_E=O(k^{-1}).
\]
Indeed, if we consider Mumford stable bundles $E$ as a special case, then we can obtain the solution to \eqref{dHYM equation} as the small deformation of the solution to the Hermitian--Einstein equation \cite[Section 4.1]{DMS20}.

In this paper, motivated by the work \cite{DMS20}, we introduce a vector bundle version of the $J$-equation, and study some basic properties as well as examples of them. Let $E$ be a holomorphic vector bundle of rank $r$ over an $n$-dimensional compact K\"ahler manifold $(X,\omega)$ satisfying
\begin{equation} \label{positivity for Chern characters}
\ch_n(E)>0, \quad [\omega] \cdot \ch_{n-1}(E)>0.
\end{equation}
Let $\Herm(E)$ be the space of all Hermitian metrics on $E$. Then we say that a Hermitian metric $h \in \Herm(E)$ solves the $J$-equation if it satisfies
\begin{equation} \label{J-equation}
c F_h^n-\omega \Id_E \wedge F_h^{n-1}=0
\end{equation}
for some constant $c$. Taking the trace and integrating over $X$, one can easily see that the constant $c$ is uniquely determined as a cohomological invariant
\[
c=\frac{[\omega] \cdot \ch_{n-1}(E)}{n \ch_n(E)}>0.
\]
We note that the equation \eqref{J-equation} gives an extension of \eqref{J-equation on line bundles} to higher rank vector bundles, and coincides with the Hermitian--Einstein equation when $n=1$. Contrary to the case \cite{DMS20}, the equation \eqref{J-equation} is designed so that the dHYM equation \eqref{dHYM equation} reaches to \eqref{J-equation} in the small volume limit. A novelty of our approach is that we can provide examples for dHYM connections with large average angle $\Theta$ whereas all the examples studied in \cite{DMS20} have small average angle from the construction (see Theorem \ref{existence result for the dHYM equation}).

\subsection{Main results}
Now we state our main results. First, we explain there is a moment map/GIT framework associated to the $J$-equation, \ie the $J$-equation can be viewed as the zero of the moment map on a suitable infinite-dimensional symplectic manifold.

\begin{thm} \label{moment map interpretation for the J-equation}
There is a moment map interpretation for $J$-positive solutions to the $J$-equation \eqref{J-equation}.
\end{thm}
The ``$J$-positivity'' is a positivity concept for Hermitian metrics on holomorphic vector bundles which guarantees that the symplectic form is non-degenerate. Also when $n=2$, we introduce the ``$J$-Griffiths positivity'' as a weaker version of the $J$-positivity (see Section \ref{Positivity concepts of Hermitian metrics on vector bundles} for more details).

For a rank-$2$ vector bundle $E$ over a compact K\"ahler surface $(X,\omega)$, we study some algebraic obstructions for the $J$-equation as an analogue of the Mumford stability for the Hermitian--Einstein equation. We say that $E$ is {\it $J$-semistable} if the inequality
\[
\ch_2(\cS) \cdot \big( [\omega] \cdot \ch_1(E) \big) \leq \big( [\omega] \cdot \ch_1(\cS) \big) \cdot \ch_2(E)
\]
holds for all coherent subsheaves $\cS \subset E$ with $\rank \cS>0$. Also we say that $E$ is {\it $J$-stable} if the above inequality is strict for all $\cS$ with $\rank \cS=1$.

\begin{thm} \label{J-stability for rank-2 bundles over surfaces}
Let $E$ be a holomorphic vector bundle of rank-$2$ over a compact K\"ahler surface $(X,\omega)$ satisfying \eqref{positivity for Chern characters}. Assume that $E$ admits a $J$-Griffiths positive Hermitian metric $h$ solving \eqref{J-equation}. Then $E$ is $J$-semistable for all subbundles. Moreover, $E$ is $J$-stable for all subbundles if it is indecomposable.
\end{thm}

Then we study some existence and non-existence results for the $J$-equation. First, we study an example on projective spaces, which together with Lemma \ref{small deformation} gives the first non-trivial example for the $J$-equation. 

\begin{thm} \label{J equation on projective spaces}
Let $H$ be the Hermitian metric on the holomorphic tangent bundle $T' \C\P^n$ induced from the Fubini--Study metric $\omega_{\FS} \in c_1(\cO(1))$. Then $T' \C\P^n$ satisfies \eqref{positivity for Chern characters} with respect to $\omega_{\FS}$ and $H$ is a solution to the $J$-equation
\[
F_H^n-\omega_{\FS} \Id_{T' \C\P^n} F_H^{n-1}=0.
\]
Moreover, when $n=2$, $H$ is $J$-positive, and for any K\"ahler form $\omega$ on $\C\P^2$ which is sufficiently close to $\omega_{\FS}$, there is a $J$-positive solution to \eqref{J-equation} with respect to $\omega$.
\end{thm}

It seems to be hard to extend Theorem \ref{J-stability for rank-2 bundles over surfaces} directly to higher rank case, that is why we will consider the asymptotic version of $J$-stability and study the relation to the solvability of the $J$-equation: let $(X,\omega)$ be a compact K\"ahler surface and $L$ an ample line bundle over $X$. Let $\eta$ be an another K\"ahler form on $X$ chosen in Theorem \ref{existence of a solution on twisted bundles}. We recall that a coherent sheaf $\cE$ over a compact K\"ahler surface $(X,\eta)$ is said to be {\it Mumford semistable} with respect to $\eta$ if the inequality
\[
\mu_\eta(\cS) \leq \mu_\eta(\cE)
\]
holds for all coherent subsheaves $\cS \subset \cE$ with $\rank \cS>0$, where
\[
\mu_\eta(\cS):=\frac{c_1(\cS) \cdot [\eta]}{\rank \cS}
\]
is the {\it Mumford slope} of $\cS$. Moreover, $\cE$ is said to be {\it Mumford stable} if the above inequality is strict for all $\cS$ with $0<\rank \cS<\rank \cE$. Now for any coherent sheaf $\cS$ with $\rank \cS>0$ over $(X,\omega)$, we define the {\it asymptotic $J$-slope} of $\cS$ by
\[
\varphi_k(\cS):=\frac{\ch_2(\cS \otimes L^k)}{[\omega] \cdot \ch_1(\cS \otimes L^k)},
\]
where we note that $[\omega] \cdot \ch_1(\cS \otimes L^k)>0$ if $k$ is large enough by the ampleness of $L$. We say that a coherent sheaf $\cE$ is {\it asymptotically $J$-semistable} (with respect to $(\omega,L)$) if for all coherent subsheaves $\cS \subset \cE$ with $\rank \cS>0$, the inequality
\[
\varphi_k(\cS) \leq \varphi_k(\cE)
\]
holds for sufficiently large integer $k$, where the magnitude of $k$ is allowed to depend on $\cS$. Moreover, we say that $\cE$ is {\it asymptotically $J$-stable} if the above inequality is strict for all $\cS$ with $0<\rank \cS<\rank \cE$. The next theorem says that we can actually solve the equation under some stability condition including both subvarieties and subsheaves.

\begin{thm} \label{existence of a solution on twisted bundles}
Let $L$ be an ample line bundle over a compact K\"ahler surface $(X,\omega)$ satisfying
\begin{equation} \label{NM criterion}
\big( 2([\omega] \cdot c_1(L)) c_1(L)-c_1(L)^2 [\omega] \big) \cdot Y>0
\end{equation}
for any curve $Y \subset X$. In particular, we have a solution $\chi \in c_1(L)$ satisfying the $J$-equation with respect to $\omega$
\begin{equation} \label{J-equation on surfaces}
\frac{[\omega] \cdot c_1(L)}{c_1(L)^2} \chi^2-\omega \chi=0
\end{equation}
with a K\"ahler form $\eta:=2([\omega] \cdot c_1(L)) \chi-c_1(L)^2 \omega>0$\footnote{The condition $\eta>0$ is nothing but the $J$-positivity condition for $\chi$. For line bundles, any solution to \eqref{J-equation on line bundles} automatically satisfies the $J$-positivity condition (see Remark \ref{J-positivity for line bundles}).}. Let $E$ be a holomorphic vector bundle of rank $r$ over $X$ and $h_L \in \Herm(L)$ a Hermitian metric with curvature $\chi$ with the associated connection $D_L$. Let us consider the $J$-equation for $\tilde{h}_k \in \Herm(E \otimes L^k)$ with respect to $\omega$
\begin{equation} \label{J-equation on twisted bundles}
\frac{[\omega] \cdot \ch_1(E \otimes L^k)}{2 \ch_2(E \otimes L^k)} F_{\tilde{h}_k}^2-\omega \Id_E F_{\tilde{h}_k}=0.
\end{equation}
Then we have the following:
\begin{enumerate}
\item The bundle $E \otimes L^k$ satisfies the property \eqref{positivity for Chern characters} with respect to $\omega$ for sufficiently large integer $k$.
\item Assume that there exist a Hermitian metric $h_0 \in \Herm(E)$ and a sequence of $h_0$-compatible integrable connections $D_k$ on $E$ satisfying
\begin{enumerate}
\item The pair $(h_0 \otimes h_L^k,D_k \otimes D_L^k)$ solves \eqref{J-equation on twisted bundles}.
\item $D_k$ are uniformly bounded in $C^1$.
\end{enumerate}
Then for any subbundle $S \subset E$ with $0<\rank S<r$ that is holomorphic with respect to all $D_k$, we have $\varphi_k(S) \leq \varphi_k(E)$ for sufficiently large $k$, with the strict inequality holding on each level $k$ if $E$ is indecomposable with respect to $D_k$.
\item Assume that $E$ is an asymptotically $J$-stable, sufficiently smooth vector bundle and the associated graded object $\Gr(E)$ has at most $2$ Mumford stable components. Then for sufficiently large integer $k$, there exists a $J$-positive Hermitian metric $\tilde{h}_k \in \Herm(E \otimes L^k)$ satisfying \eqref{J-equation on twisted bundles}.
\end{enumerate}
\end{thm}
The equivalence between \eqref{NM criterion} and \eqref{J-equation on surfaces} is essentially due to the Nakai--Moishezon criterion \cite{DP04} and even known in the higher dimensional case \cite{Son20}. The properties (2) and (3) are analogues of those for the Gieseker stability \cite{Leu97} and the asymptotic $Z$-stability \cite{DMS20}. A key observation is that the asymptotic $J$-stability implies Mumford semistability with respect to $\eta$ (Proposition \ref{Mumford stability vs asymptotic J-stability}). Indeed, the connection that we obtained in (3) is of the form $D_k \otimes D_L^k$ where the connections $D_k$ on $E$ are constructed as a small deformation from the weak Hermitian--Einstein connection on $\Gr(E)$. Thus the boundedness of $D_k$ in $C^\infty$ even holds in this case. To prove (3) in the two components case, the standard implicit function theorem is insufficient due to the absence of the weak Hermitian--Einstein connection on $E$. As in \cite{ST20}, an approximation method together with the quantitative implicit function theorem is needed. In Theorem \ref{existence of a solution on twisted bundles}, we impose a lot of technical assumptions although we expect that all the properties actually hold in full generality. Also it would be interesting to extend Theorem \ref{existence of a solution on twisted bundles} to higher dimensional compact K\"ahler manifolds (see Remark \ref{higher dimensional case}). We plan to improve these issues in future work.

Let us consider a compact Riemann surface $\Sigma$ and the rank-$2$ vector bundle
\begin{equation} \label{vortex bundles}
E:=\pi_1^\ast((r_1+1)L) \otimes \pi_2^\ast(r_2(\cO(2))) \oplus \pi_1^\ast (r_1L) \otimes \pi_2^\ast((r_2+1) \cO(2))
\end{equation}                                             
over $\Sigma \times \C\P^1$, where $r_1, r_2 \in \N$ and $\pi_1$ (resp. $\pi_2$) is the projection to $\Sigma$ (resp. $\C\P^1$). We equip $E$ with a holomorphic structure so that the first component $\pi_1^\ast((r_1+1)L) \otimes \pi_2^\ast(r_2(\cO(2)))$ becomes a holomorphic subbundle of $E$. The bundle $E$, referred to as the {\it vortex bundle} is originally studied by Garc\'ia--Prada \cite{Pra93} in the Hermitian--Einstein equation case, and thereafter by Pingali \cite{Pin20} in the vector bundle Monge--Amp\`ere equation case. Let $L$ be a holomorphic line bundle over $\Sigma$ of degree $1$ and $h_\Sigma$ a Hermitian metric on $L$ with curvature $\omega_\Sigma>0$. Then $\SU(2)$ acts on $\Sigma \times \C\P^1$, trivially on $\Sigma$, and the standard way on $\C\P^1$ via the identification $\C\P^1 \simeq \SU(2)/\Uni (1)$. Since this action lifts to $E$ in the standard manner, it is natural to consider $\SU(2)$-invariant solutions to \eqref{J-equation}. Consider the $\SU(2)$-invariant K\"ahler metric $\omega:=s \pi_1^\ast \omega_\Sigma+\pi_2^\ast \omega_{\FS}$ for $s \in \R_{>0}$ (where without loss of generality, we may assume that the coefficient of $\pi_2^\ast \omega_{\FS}$ is $1$ since the equation \eqref{J-equation} is invariant under the scaling of $\omega$).

\begin{thm} \label{existence result on vortex bundles}
Let $E$ be the vortex bundle defined by \eqref{vortex bundles} over a product of Riemann surfaces $\Sigma \times \C\P^1$, and set $\omega:=s \pi_1^\ast \omega_\Sigma+\pi_2^\ast \omega_{\FS}$ ($s \in \R_{>0}$) as above. Then we have the following:
\begin{enumerate}
\item For any $r_1, r_2, s$, the bundle $E$ satisfies the property \eqref{positivity for Chern characters} with respect to $\omega$.
\item If $E$ admits a $J$-Griffiths positive Hermitian metric satisfying \eqref{J-equation} with a non-trivial second fundamental form of $\pi_1^\ast((r_1+1)L) \otimes \pi_2^\ast(r_2(\cO(2)))$, then we have
\[
s<\frac{r_1(r_1+1)}{2r_2(r_2+1)}.
\]
\item Assume that $r_1,r_2,s$ satisfies
\[
r_2 \geq 2, \quad \frac{(2r_1+1)\kappa_0+4r_1}{2(4r_2-\kappa_0)(2r_2+1)}<s<\frac{r_1(r_1+1)}{2r_2(r_2+1)},
\]
where
\[
\kappa_0:=\frac{1}{3} \bigg(2+\sqrt[3]{1232-528 \sqrt{3}}+2 \sqrt[3]{22(7+3 \sqrt{3})} \bigg) \approx 7.2405
\]
denotes the (unique) positive root of the polynomial $\kappa^3-2\kappa^2-28\kappa-72$. Then under a suitable choice of holomorphic structures on $E$, the bundle $E$ admits an $\SU(2)$-invariant, $J$-Griffiths positive Hermitian metric solving \eqref{J-equation}.
\end{enumerate}
\end{thm}
The upper bound for $s$ is a necessary condition to the existence, while the lower bound may not be sharp (see Remark \ref{the lower bound may not be sharp}). As in \cite{Pra93,Pin20}, our strategy is to study the following {\it $J$-vortex equation} for a Hermitian metric $h$ on $L$ obtained as a dimensional reduction of \eqref{J-equation}
\[
\begin{aligned}
F_h=2(1-|\phi|_h^2) \frac{\frac{\i c^2}{\pi} D_h' \phi D_h'' \phi^\ast+s \omega_\Sigma}{(4cr_2-2c|\phi|_h^2-1+4c)(4cr_2+2c|\phi|_h^2-1)},
\end{aligned}
\]
where $\phi$ is any given non-zero holomorphic section of $L$ with $|\phi|_h^2 \leq 1$ which determines the (equivalence class of) extension
\[
0 \to \pi_1^\ast((r_1+1)L) \otimes \pi_2^\ast(r_2(\cO(2))) \to E \to \pi_1^\ast (r_1L) \otimes \pi_2^\ast((r_2+1) \cO(2)) \to 0.
\]
Indeed, from the argument in Section \ref{Vortex bundles}, we know that we can choose any non-zero $\phi$ to solve the $J$-vortex equation. Finally, we show an existence result for the dHYM equation assuming the existence of a solution to \eqref{J-equation}.

\begin{thm} \label{existence result for the dHYM equation}
Let $E$ be a simple holomorphic vector bundle over an $n$-dimensional compact K\"ahler manifold $(X,\omega)$ satisfying the property \eqref{positivity for Chern characters}. Assume that there exists a $J$-positive Hermitian metric $h_0 \in \End(E)$ satisfying the $J$-equation
\[
\frac{[\omega] \cdot \ch_{n-1}(E)}{n \ch_n(E)} F_{h_0}^n-\omega \Id_E F_{h_0}^{n-1}=0.
\]
Then for sufficiently small $\e>0$, the bundle $E$ admits a dHYM-positive Hermitian metric $h_\e$ satisfying the deformed Hermitian--Yang--Mills equation with respect to $\e \omega$
\[
\Im \big(e^{-\i \Theta_\e}(\e \omega \Id_E+\i F_{h_\e})^n \big)=0
\]
for some constant $\Theta_\e$ (depending on $\e$).
\end{thm}
The dHYM-positivity\footnote{In \cite{DMS20}, Hermitian metrics satisfying the dHYM-positivity condition were called ``subsolutions''. However, we would like to adopt the former name to distinguish from several other notions of subsolutions.} is a positivity concept for Hermitian metrics introduced in \cite{DMS20}, which is crucial to assure the ellipticity of the equation \eqref{dHYM equation}. The main technical point is that for simple vector bundles, the kernel of the linearized $J$-operator at a $J$-positive solution consists of only real scalar multiples of the identity, so the standard implicit function theorem can be applied in this case.

\subsection{Relation to Pingali's work}
We should mention Pingali's work \cite{Pin20}, where he studied the vector bundle Monge--Amp\`ere equation
\begin{equation} \label{MA equation}
F_h^n=\eta \Id_E
\end{equation}
for a given volume form $\eta$. Our results (in particular, Theorem \ref{existence result on vortex bundles}) can be thought as an analogue of the results in \cite{Pin20} although there seems to be no relations between \eqref{J-equation} and \eqref{MA equation} in general situation. After \cite{Pin20}, Ghosh \cite{Gho20} gave a systematic study of vortex-type equations on compact Riemann surfaces. Zhang--Zhang \cite{ZZ22} introduced a family of Donaldson's functional on holomorphic vector bundles, whose Euler--Lagrange equations are a vector bundle version of the complex $k$-Hessian equations. They also discuss the uniqueness of solutions to these equations.

However, as a special case, if $n=2$ and the form $\omega/2c$ represents $c_1(L)$ for some ample line bundle $L$ over $X$, then the equation \eqref{J-equation} can be reduce to the vector bundle Monge--Amp\`ere equation on the twisted vector bundle $\tilde{E}:=E \otimes L^{-1}$
\[
F_{\tilde{h}}^2=\frac{\omega^2}{4c^2} \Id_E,
\]
where $\tilde{h}:=h \otimes h_L^{-1}$ and $h_L$ is a Hermitian metric on $L$ with curvature $\omega/2c$. This fact seems to give a supporting evidence why the same method works well for the $J$-equation.

\subsection{Organization of the paper}
This paper is organized as follows: in Section \ref{Preliminaries}, we sum up some properties on the space of connections and Hermitian metrics. The standard reference is \cite[Chapter 7]{Kob14}. Then in Section \ref{Positivity concepts of Hermitian metrics on vector bundles}, we give the definition of $J$-positivity, which is crucial to assure the ellipticity of \eqref{J-equation}. Also, in the case $n=2$, we define the notion of $J$-Griffiths positivity as a weaker version of the $J$-positivity. In Section \ref{Moment map interpretation}, we give the moment map interpretation for the $J$-equation (Theorem \ref{moment map interpretation for the J-equation}). Geometrically, the set of $J$-positive connections are characterized as a subset of the space of integrable connections where the symplectic form is non-degenerate. In Section \ref{J-stability for rank-2 vector bundles over compact Kahler surfaces}, we show that the existence of a $J$-Griffiths positive solution to \eqref{J-equation} implies the $J$-(semi)stability for rank-$2$ vector bundles over compact K\"ahler surfaces (Theorem \ref{J-stability for rank-2 bundles over surfaces}). In Section \ref{Examples}, we give three examples, namely, projective spaces (Theorem \ref{J equation on projective spaces}), sufficiently smooth vector bundles (Theorem \ref{existence of a solution on twisted bundles}) and vortex bundles (Theorem \ref{existence result on vortex bundles}) as mentioned above. Finally, in Section \ref{The deformed Hermitian--Yang--Mills equation}, we construct dHYM-positive solutions to \eqref{dHYM equation} on simple vector bundles in the small volume regime assuming the existence of a $J$-positive solution to \eqref{J-equation} (Theorem \ref{existence result for the dHYM equation}).

\begin{ackn}
The author was supported by Grant-in-Aid for Early-Career Scientists (20K14308) from JSPS. The author expresses his gratitude to Lars Martin Sektnan for giving me insightful comments related to the work \cite{DMS20}. The author thank Jacopo Stoppa and Zakarias Sj\"ostr\"om Dyrefelt for the invitation to the K\"ahler geometry seminar at SISSA/ICTP and several interesting discussions. Also the author is grateful to the referee for many insightful comments which have helped to improve the article.
\end{ackn}

\section{Preliminaries} \label{Preliminaries}
\subsection{The space of connections}
Let $E$ be the complex vector bundle and $\cA''(E)$ the set of all $\C$-linear maps $D'' \colon \Omega^0(E) \to \Omega^{0,1}(E)$ satisfying
\[
D''(fs)=(\bp f)s+f \cdot D''s, \quad s \in \Omega^0(E), \; f \in C^\infty(X;\C).
\]
For any fixed element $D''_0 \in \cA''(E)$, $\cA''(E)$ is an affine space modelled on $\Omega^{0,1}(\End(E))$
\[
\cA''(E)=\{D''_0+a''| a'' \in \Omega^{0,1}(\End(E))\}.
\]
We also consider another expression for $\cA''(E)$. For any fixed $h \in \Herm(E)$, let $\cA(E,h)$ be the set of $h$-compatible connections $D$ on $E$, and decompose $D$ as
\[
D=D'+D'',
\]
where $D' \colon \Omega^0(E) \to \Omega^{1,0}(E)$ and $D'' \colon \Omega^0(E) \to \Omega^{0,1}(E)$. If we fix an element $D_0 \in \cA(E,h)$, then we have
\[
\cA(E,h)=\{D_0+a| a \in \Omega^1(\End(E,h))\},
\]
where $\End(E,h)$ is the bundle of skew-Hermitian endomorphisms of $(E,h)$. So $\cA(E,h)$ is an infinite-dimensional affine space modelled on $\Omega^1(\End(E,h))$. Then the natural map
\begin{equation} \label{identify connection}
\cA(E,h) \to \cA''(E), \quad D \mapsto D''
\end{equation}
is bijective. If we fix a reference connection $D_0 \in \cA(E,h)$, then we can express the above correspondence as $a \mapsto a''$, where $a=a'+a'' \in \Omega^1(\End(E,h))$ ($a' \in \Omega^{1,0}(\End(E))$, $a'' \in \Omega^{0,1}(\End(E))$). The inverse is given by $a'' \mapsto -(a'')^\ast+a''$, where $\ast$ denotes the adjoint with respect to $h$.

Next, we impose the integrability condition
\[
\cA''_{\inte}(E):=\{D'' \in \cA''(E)|D'' \circ D''=0\},
\]
\[
\cA_{\inte}(E,h)=\{D \in \cA(E,h)|D'' \in \cA''_{\inte}(E)\}.
\]
The space $\cA''_{\inte}(E)$ may be considered as the set of holomorphic bundle structures on $E$ (which comes from the fact that any $D'' \in \cA''(E)$ defines the unique holomorphic bundle structure on $E$ whose $\bp$-operator coincides with $D''$ and vice versa). Moreover, the restriction of the map \eqref{identify connection} still gives the bijection $\cA_{\inte}(E,h) \to \cA_{\inte}''(E)$. The tangent spaces are given by
\[
T_{D''} \cA''_{\inte}(E)=\{a'' \in \Omega^{0,1}(\End(E))|D'' a''=0 \},
\]
\[
T_D \cA_{\inte}(E,h)=\{a \in \Omega^1(\End(E,h))|D'' a''=0 \}.
\]
On occasion we change the conventions for convenience.

Next we consider the action of the gauge group $\cG:=\Omega^0(\Uni (E,h))$ and its complexification $\cG^\C:=\Omega^0(\GL(E,\C))$, where $\Uni (E,h)$ is the bundle of unitary endomorphisms of $E$ with respect to $h$. For any $D \in \cA(E,h)$ and $g \in \cG^\C$, the group $\cG^\C$ acts on $\cA(E,h)$ by
\begin{equation} \label{action of the gauge group}
\begin{aligned}
D &\mapsto D^g:=g^\ast \circ D' \circ (g^\ast)^{-1}+g^{-1} \circ D'' \circ g \\
&=D'+g^\ast \circ D' (g^\ast)^{-1}+D''+g^{-1} \circ D'' g,
\end{aligned}
\end{equation}
where $g^\ast$ denotes the adjoint of $g$ with respect to $h$. Then the change of the curvature $F_D:=\frac{\i}{2\pi} D \circ D$ is
\begin{equation} \label{change of the curvature}
\begin{aligned}
F_D &\mapsto F_{D^g}=\frac{\i}{2\pi} \big( g^\ast \circ D' \circ D' \circ (g^\ast)^{-1}+g^{-1} \circ D'' \circ D'' \circ g \\
&+g^\ast \circ D' \circ (g^\ast)^{-1} \circ g^{-1} \circ D'' \circ g+g^{-1} \circ D'' \circ g \circ g^\ast \circ D' \circ (g^\ast)^{-1} \big).
\end{aligned}
\end{equation}
We note that the action of $\cG^\C$ preserves the integrability condition, so the group $\cG^\C$ acts on $\cA_{\inte}(E,h)$ and $\cA_{\inte}''(E)$. Set $\fg:=\Lie \cG=\Omega^0(\End(E,h))$ so that $\fg^\C=\Omega^0(\End(E))$. By using \eqref{action of the gauge group}, one can check that the infinitesimal action of $\fg^\C$ on $D \in \cA(E,h)$ is given by $v \mapsto -D'v^\ast+D''v$.

Finally, we compute the variation of the curvature $F_D$ when $D \in \cA(E,h)$ varies. For $a \in \Omega^1(\End(E,h))$, we set $D_t:=D+ta$. For any $s \in \Omega^0(E)$, we compute
\[
\begin{aligned}
F_{D_t}(s)&=\frac{\i}{2 \pi} D_t(D(s)+ta(s)) \\
&=\frac{\i}{2 \pi} D(D(s))+\frac{\i}{2 \pi} ta(D(s))+\frac{\i}{2 \pi} tD(a(s))+\frac{\i}{2 \pi} t^2 (a \wedge a)(s) \\
&=F_D(s)+\frac{\i}{2 \pi} t(Da)(s)+\frac{\i}{2 \pi} t^2 (a \wedge a)(s).
\end{aligned}
\]
Thus we have $F_{D_t}=F_D+\frac{\i}{2\pi}t Da+\frac{\i}{2\pi} t^2 a \wedge a$ and
\[
\d_a F_D=\frac{\i}{2\pi} Da.
\]
In particular, for any $v \in \fg^\C$ and $D \in \cA_{\inte}(E,h)$, we have
\[
\ddt F_{D^{\exp(vt)}}|_{t=0}=\frac{\i}{2\pi} (D'D''v-D''D'v^\ast).
\]

\subsection{Holomorphic subbundles}
Let us consider a holomorphic vector bundle $E$ over a compact complex manifold $X$ and a Hermitian metric $h$ on $E$ with induced connection $D$. For a holomorphic subbundle $S \subset E$, we have a orthogonal decomposition with respect to $h$ (as $C^\infty$ complex vector bundles)
\[
E=S \oplus S^\perp.
\]
and an identification $Q \simeq S^\perp$ as $C^\infty$ complex vector bundles, where $Q:=E/S$ denotes the quotient bundle. Via this identification, the restriction of $h$ to each component induces a connection which we will denote $D_S$ (resp. $D_Q$) with curvature $F_S$ (resp. $F_Q$). We define the second fundamental form $A \in \Omega^{1,0}(\Hom(S,Q))$ as the difference of the two connections
\[
Ds=D_Ss+As, \quad s \in \Omega^0(S).
\]
The following is well-known as the {\it Gauss--Codazzi equation} which relates the curvature $F_h$ with $F_S$, $F_Q$ and $A$ as follows
\[
F_h=
\begin{pmatrix}
F_S-\frac{\i}{2 \pi} A^\ast A & -\frac{\i}{2 \pi} D' A^\ast \\
\frac{\i}{2 \pi} D'' A & F_Q-\frac{\i}{2 \pi} A A^\ast
\end{pmatrix},
\]
where the derivatives $D' A^\ast$, $D'' A$ are taken with respect to the induced metric by $D_S$ and $D_Q$. In particular, taking the trace of both sides, we have $c_1(E,h)=c_1(S,h|_S)+c_1(Q,h|_Q)$.

\section{Positivity concepts of Hermitian metrics on vector bundles} \label{Positivity concepts of Hermitian metrics on vector bundles}
Let $E$ be a holomorphic vector bundle of rank $r$ over an $n$-dimensional compact K\"ahler manifold $(X,\omega)$ satisfying \eqref{positivity for Chern characters}. Recall that the $J$-equation on $E$ is given by
\[
cF_h^n-\omega \Id_E \wedge F_h^{n-1}=0.
\]
We note that the equation is not elliptic in general. So in order to guarantee the ellipticity, we introduce the following:
\begin{dfn}
We say that a Hermitian metric $h \in \Herm(E)$ is {\it $J$-positive} with respect to $\omega$ if for all $a'' \in \Omega^{0,1}(\End(E))$, the $(n,n)$-form
\[
-\i \Tr \bigg( c \sum_{k=0}^{n-1} F_h^k a'' F_h^{n-1-k}(a'')^\ast-\omega \Id_E \sum_{k=0}^{n-2} F_h^k a'' F_h^{n-2-k} (a'')^\ast \bigg)
\]
is positive at all points on $X$ where $a'' \neq 0$. Also we say that for a fixed Hermitian metric $h \in \Herm(E)$, a connection $D \in \cA_{\inte}(E,h)$ is $J$-positive if $h$ is $J$-positive with respect to the holomorphic structure induced by $D$.
\end{dfn}
\begin{rk} \label{J-positivity for line bundles}
When $r=1$, the $J$-positivity condition just says that the $(n-1,n-1)$-form
\[
cn F_h^{n-1}-(n-1) \omega F_h^{n-2}
\]
is positive in the sense of \cite[Chapter III, Section 1.A]{Dem12}. This condition is automatically satisfied if $h$ is a solution to \eqref{J-equation} (for instance, see \cite{SW08}). Thus the $J$-positivity condition can be seen as a natural generalization of this positivity condition for higher rank vector bundles.
\end{rk}

For any $J$-positive metric $h \in \Herm(E)$, the pairing
\begin{equation} \label{Hermitian inner product}
\langle a'', b'' \rangle:=-\i \int_X \Tr \bigg( c \sum_{k=0}^{n-1} F_h^k a'' F_h^{n-1-k} (b'')^\ast-\omega \Id_E \sum_{k=0}^{n-2} F_h^k a'' F_h^{n-2-k} (b'')^\ast \bigg)
\end{equation}
defines the Hermitian inner product. Indeed, this is clearly $\C$-linear in the first slot, anti-$\C$-linear in the second slot and positive definite by the $J$-positivity assumption. Moreover, if we regard $F_h$ as a matrix-valued $(1,1)$-form by using an orthonormal frame of $E$ with respect to $h$, then we have $(F_h)^\ast=F_h$ and $\Tr$ in \eqref{Hermitian inner product} can be interpreted as the trace of matrices. From these observations, we can easily check that $\langle a'', b'' \rangle=\overline{\langle b'', a'' \rangle}$.

\begin{lem} \label{ellipticity}
Suppose that $h \in \Herm(E)$ is $J$-positive. Then \eqref{J-equation} is elliptic.
\end{lem}
\begin{proof}
Let $D$ be the associated connection of $h$. Taking the point of view of changing the holomorphic structure on $E$ through the action of the gauge group, we define an operator $\cP \colon \i \fg \to \i \fg$ by
\[
\cP(v):=c F_{D^{\exp(v)}}^n-\omega \Id_E F_{D^{\exp(v)}}^{n-1},
\]
where $\fg=\Omega^0(\End(E,h))$\footnote{By abuse of notation, we often identify $(n,n)$-forms with functions by using a fixed volume form.}. Then the linearization of $\cP$ at $0$ in the direction $v \in \i \fg$ is given by
\[
\d|_{0,v} \cP=\frac{\i}{2 \pi} \bigg( c \sum_{k=0}^{n-1} F_D^k (D' D'' v-D'' D' v) F_D^{n-1-k}-\omega \Id_E \sum_{k=0}^{n-2} F_D^k (D' D'' v-D'' D' v) F_D^{n-2-k} \bigg).
\]
We take local coordinates $(z^i)$ as well as a real local frame $(e_\a)$ of $E$, and write $v$ locally as $v=\sum_\a v^\a e_\a$. Then we compute
\[
D' D'' v-D'' D' v=2 \sum_\a \bigg( \sum_{j,k} \frac{\p}{\p z^j} \frac{\p}{\p z^{\bar{k}}}(v^\a) dz^j \wedge dz^{\bar{k}} \bigg) e_\a+\text{(terms of lower order derivatives of $v^\a$)}.
\]
Meanwhile, let $(x^i,y^i)$ be real coordinates with $\frac{\p}{\p z^i}=\frac{1}{2} \big( \frac{\p}{\p x^i}-\i \frac{\p}{\p y^i} \big)$. If we perform the replacement $\frac{\p}{\p x^i} \mapsto \zeta_i$, $\frac{\p}{\p y^i} \mapsto \eta_i$, then $\frac{\p}{\p z^i} \mapsto \frac{1}{2} (\zeta_i-\i \eta_i)$. This shows that if we set $\xi:=\sum_i \xi_i dz^{i} \in \Omega^{1,0}$ with $\xi_i:=\frac{1}{2} (\zeta_i-\i \eta_i)$, then we may compute the principal symbol $\s$ of $\d|_{\Id_E} \cP$ just by replacing $\frac{\p}{\p z^i} \mapsto \xi_i$. Thus we have
\[
\s_\xi(v)=\frac{\i}{\pi} \bigg(c \sum_{k=0}^{n-1} F_D^k v F_D^{n-1-k}-\omega \Id_E \sum_{k=0}^{n-2} F_D^k v F_D^{n-2-k}\bigg) (\xi \Id_E) (\bar{\xi} \Id_E).
\]
On the other hand, the $J$-positivity condition for $h$ says that for all $a'' \in \Omega^{0,1}(\End(E))$ we have
\[
-\i \Tr \bigg( c \sum_{k=0}^{n-1} F_D^k a'' F_D^{n-1-k} (a'')^\ast-\omega \Id_E \sum_{k=0}^{n-2} F_D^k a'' F_D^{n-2-k} (a'')^\ast \bigg)>0
\]
at all points on $X$ where $a'' \neq 0$. Putting $a''=\bar{\xi} \otimes v$ in the above equation, we observe that
\[
\Tr(\s_\xi(v) v)>0
\]
at all points where $v \neq 0$ and $\xi \neq 0$. This shows that the equation \eqref{J-equation} is elliptic at $D$ as desired.
\end{proof}

\begin{lem} \label{small deformation}
Let $E$ be a simple holomorphic vector bundle of rank $r$ over an $n$-dimensional compact K\"ahler manifold $(X,\omega_0)$ satisfying \eqref{positivity for Chern characters} with respect to $\omega_0$, and $h_0 \in \Herm(E)$ a $J$-positive metric satisfying
\[
c_0 F_{h_0}^n-\omega_0 \Id_E F_{h_0}^{n-1}=0, \quad c_0:=\frac{[\omega_0] \cdot \ch_{n-1}(E)}{n \ch_n (E)}.
\]
Then for any K\"ahler form $\omega$ sufficiently close to $\omega_0$, there exists a smooth solution $h \in \Herm(E)$ to
\begin{equation} \label{small deformation of the J-equation}
c F_h^n-\omega \Id_E F_h^{n-1}=0, \quad c:=\frac{[\omega] \cdot \ch_{n-1}(E)}{n \ch_n (E)}.
\end{equation}
\end{lem}
\begin{proof}
As an equivalent point of view, we will consider the variation of holomorphic structures on $(E,h_0)$ through the action of the gauge group as considered in Lemma \ref{ellipticity}. For any $\ell \in \Z_{\geq 0}$ and $\b \in (0,1)$, let $\cV^{\ell,\b}$ be a Banach manifold consisting of all $C^{\ell,\b}$ Hermitian endomorphisms $v$ of $(E,h_0)$ satisfying $\int_X \Tr (v) \omega_0^n=0$, and let $\cK^{\ell,\b}$ be a Banach manifold consisting of all $C^{\ell,\b}$ K\"ahler forms $\omega$ satisfying $[\omega] \cdot \ch_{n-1}(E)>0$. Define an operator $\cP \colon \cV^{\ell+2,\b} \times \cK^{\ell,\b} \to \cV^{\ell,\b}$ by
\[
\cP(v,\omega):=cF_{D_0^{\exp(v)}}^n-\omega \Id_E F_{D_0^{\exp(v)}}^{n-1}, \quad c:=\frac{[\omega] \cdot \ch_{n-1}(E)}{n \ch_n(E)},
\]
where $D_0$ denotes the associated connection of $h_0$. Then we have $\cP(0,\omega_0)=0$ by the assumption, and the derivative of $\cP$ at $(0,\omega_0)$ in the direction $v$ is
\begin{equation} \label{linearization of the operator}
\begin{aligned}
\d|_{(0,\omega_0),v} \cP &=\frac{\i}{2 \pi} \bigg( c_0 \sum_{k=0}^{n-1} F_{D_0}^k (D_0' D_0'' v-D_0'' D_0' v) F_{D_0}^{n-1-k} \\
&-\omega_0 \Id_E \sum_{k=0}^{n-2} F_{D_0}^k (D_0' D_0'' v-D_0'' D_0' v) F_{D_0}^{n-2-k} \bigg),
\end{aligned}
\end{equation}
which is elliptic since $D_0$ is $J$-positive, and hence is Fredholm. Moreover, we know that $\d|_{(0,\omega_0)} \cP$ is self-adjoint by applying the Bianchi identity and closedness of $\omega_0$. So by the Fredholm alternative, it suffices to show that the kernel of $\d|_{(0,\omega_0)} \cP$ is trivial. Indeed, if $\d|_{(0,\omega_0)} \cP(v)=0$, then multiplying \eqref{linearization of the operator} by $v$, taking the trace and integrating by parts, we have
\[
\langle D''_0v, D''_0v \rangle=0,
\]
where $\langle \cdot, \cdot \rangle$ denotes the Hermitian inner product defined by \eqref{Hermitian inner product}. This shows that $D''_0v=0$ and hence $v=\lambda \Id_E$ for some $\lambda \in \R$ since $E$ is simple. However, this yields that $v=0$ from the normalization $\int_X \Tr(v) \omega_0^n=0$. Consequently, we find that $\d|_{(0,\omega_0)} \cP$ is an isomorphism, and apply the implicit function theorem on Banach manifolds to get the desired statement.
\end{proof}

Also we introduce a weaker concept of $J$-positivity when $n=2$ as follows.

\begin{dfn}
Let $E$ be a holomorphic vector bundle of rank $r$ over a compact K\"ahler surface $(X,\omega)$ satisfying \eqref{positivity for Chern characters}. We say that $h \in \Herm(E)$ is {\it $J$-Griffiths positive} (with respect to $\omega$) if it satisfies
\begin{equation} \label{J-Griffiths positivity}
-\i \big( 2c \cdot h(F_h(v)(\xi,\bar{\xi}),v)-|v|_h^2 \omega(\xi,\bar{\xi}) \big)>0
\end{equation}
for all $x \in X$, non-zero $v \in E_x$ and non-zero $\xi \in T_x' X$.
\end{dfn}

In particular, if $h$ is $J$-Griffiths positive, then taking the trace of \eqref{J-Griffiths positivity}, we know that the $(1,1)$-form
\[
2c \cdot c_1(E,h)-r \omega
\]
is K\"ahler.
\begin{prop} \label{J-positivity and J-Griffiths positivity}
Let $E$ be a holomorphic vector bundle of rank $r$ over a compact K\"ahler surface $(X,\omega)$. If $h \in \Herm(E)$ is $J$-positive with respect to $\omega$, then it is $J$-Griffiths positive with respect to $\omega$.
\end{prop}
\begin{proof}
For any $x \in X$, non-zero $v \in E_x$ and non-zero $\xi \in T_x' X$, we choose local coordinates $(z^1,z^2)$ so that $\xi=\frac{\p}{\p z^1}$. Moreover, we choose a local orthonormal frame of $E$ with respect to $h$ on which the curvature $F_h$ can be expressed as
\[
F=A \i dz^1 dz^{\bar{1}}+B \i dz^1 dz^{\bar{2}}+B^\ast \i dz^2 dz^{\bar{1}}+C \i dz^2 dz^{\bar{2}},
\]
where $A, B, C$ are $r \times r$ matrix-valued functions satisfying $A=A^\ast$, $C=C^\ast$. Similarly, we write the $\i dz^1 dz^{\bar{1}}$-component of $\omega$ as $D>0$. Then what we have to show is that
\[
2cA-D \Id_r>0.
\]
The $J$-positivity assumption means that
\[
-\i \Tr \bigg(c \big( a'' F_h (a'')^\ast+F_h a'' (a'')^\ast \big)-\omega \Id_E a'' (a'')^\ast \bigg)>0
\]
holds for all non-zero $a'' \in \Omega^{0,1}(\End(E))$. In particular, if we take $a''=dz^{\bar{2}} \otimes \eta$ ($\eta \in \Omega^0(\End(E))$), then
\[
\frac{-\i c \Tr \big(a'' F_h (a'')^\ast+F_h a'' (a'')^\ast \big)}{(\i)^2 dz^1 dz^{\bar{1}} dz^2 dz^{\bar{2}}}=c \big( \Tr(\eta A \eta^\ast)+\Tr(\eta^\ast A \eta) \big),
\]
\[
\frac{\i \Tr \big( \omega \Id_E a'' (a'')^\ast \big)}{(\i)^2 dz^1 dz^{\bar{1}} dz^2 dz^{\bar{2}}}=-D \Tr(\eta \eta^\ast).
\]

If we set $\eta:=v v^\ast$, then we have $\eta=\eta^\ast$ and
\[
\begin{aligned}
0&<c \big( \Tr(\eta A \eta^\ast)+\Tr(\eta^\ast A \eta) \big)-D \Tr(\eta \eta^\ast) \\
&=2c \Tr(\eta A \eta^\ast)-D \Tr(\eta \eta^\ast) \\
&=\Tr(\eta (2cA-D \Id_r) \eta^\ast) \\
&=|v|_h^2 v^\ast (2cA-D \Id_r) v.
\end{aligned}
\]
This shows that $2cA-D \Id_r>0$ as desired.
\end{proof}

\begin{rk}
Pingali \cite{Pin20} introduced another concept of positivity called the $\MA$-positivity. We say that $h \in \Herm(E)$ is $\MA$-positive if
if for all $a'' \in \Omega^{0,1}(\End(E))$, the $(n,n)$-form
\[
-\i \Tr \bigg( \sum_{k=0}^{n-1} F_h^k a'' F_h^{n-1-k}(a'')^\ast \bigg)
\]
is positive at all points on $X$ where $a'' \neq 0$. He showed that the $\MA$-positivity implies the Griffiths positivity \cite[Lemma 2.4]{Pin20}. So Proposition \ref{J-positivity and J-Griffiths positivity} can be viewed as the $J$-analogue of this result. Also, from the definition, it is clear that the $J$-positivity (resp. $J$-Griffiths positivity) implies the $\MA$-positivity (resp. Griffiths positivity).
\end{rk}

\section{Moment map interpretation} \label{Moment map interpretation}
Let $(E,h)$ be a Hermitian vector bundle or rank $r$ over a compact K\"ahler manifold $(X,\omega)$. Let us consider the open subset $\cU \subset \cA_{\inte}''(E)$ consisting of all $J$-positive connections $D'' \in \cA_{\inte}''(E) (\simeq \cA_{\inte}(E,h))$. Then for any $D'' \in \cU$, the $J$-positivity condition assures that a real $(1,1)$-form
\[
\Omega_{D''}(a'',b''):=-\frac{1}{\pi} \Im \langle a'',b'' \rangle, \quad a'', b'' \in T_{D''} \cA_{\inte}''(E)
\]
is positive. If we set $a:=-(a'')^\ast+a''$, $b:=-(b'')^\ast+b''$, then
\[
\begin{aligned}
\Omega_{D''}(a'',b'') &=-\frac{1}{\pi} \Im \langle a'', b'' \rangle \\
&=\frac{1}{2 \pi} \int_X \Tr \bigg[ c \sum_{k=0}^{n-1} F_{D''}^k a'' F_{D''}^{n-1-k} (b'')^\ast+c \sum_{k=0}^{n-1} F_{D''}^k (a'')^\ast F_{D''}^{n-1-k} b''\\
&-\omega \Id_E \bigg( \sum_{k=0}^{n-2} F_{D''}^k a'' F_{D''}^{n-2-k} (b'')^\ast+\sum_{k=0}^{n-2} F_{D''}^k (a'')^\ast F_{D''}^{n-2-k} b'' \bigg) \bigg] \\
&=-\frac{1}{2 \pi} \int_X \Tr \bigg( c \sum_{k=0}^{n-1} F_{D''}^k a F_{D''}^{n-1-k} b-\omega \Id_E \sum_{k=0}^{n-2} F_{D''}^k a F_{D''}^{n-2-k} b \bigg).
\end{aligned}
\]
Thus the form $\Omega$ is transferred to a real $(1,1)$-form
\[
\Omega_D(a,b):=-\frac{1}{2 \pi} \int_X \Tr \bigg( c \sum_{k=0}^{n-1} F_D^k a F_D^{n-1-k} b-\omega \Id_E \sum_{k=0}^{n-2} F_D^k a F_D^{n-2-k} b \bigg), \quad a,b \in T_D \cA_{\inte}(E,h)
\]
via the identification $\cA_{\inte}(E,h) \simeq \cA_{\inte}''(E)$.

\begin{thm}
The $(1,1)$-form $\Omega$ defines a K\"ahler structure on $\cU$. Moreover, the moment map associated to the action of the gauge group $\cG$ is given by
\[
\mu_D(v):=\i \int_X \Tr \bigg( v \big(c F_D^n-\omega \Id_E F_D^{n-1} \big) \bigg), \quad v \in \fg=\Omega^0(\End(E,h)).
\]
\end{thm}
\begin{proof}
First, we will show that $\Omega$ is a K\"ahler form on $\cU$. Indeed, the $(1,1)$-form $\Omega$ is positive from the definition of $\cU$. As for the closedness, it is enough to show that $\Omega$ is closed on a larger affine space $\cA(E,h)$ including $\cU$. We just consider the closedness of the first term of $\Omega$ (up to scaling)
\[
\tilde{\Omega}_D(a,b):=-2 \pi \i \int_X \Tr \bigg( \sum_{k=0}^{n-1} F_D^k a F_D^{n-1-k} b \bigg), \quad D \in \cA(E,h)
\]
since the second term is similar. Take constant vector fields $a,b,c$ on $\cA(E,h)$ so that
\[
[a,b]=[b,c]=[c,a]=0.
\]
Then we have
\[
d\tilde{\Omega}_D(a,b,c)=d(\tilde{\Omega}(b,c))_D(a)-d(\tilde{\Omega}(a,c))_D(b)+d(\tilde{\Omega}(a,b))_D(c).
\]
To simplify the notations, we define
\[
I(A,B,C):=\sum_{\substack{k,\ell,m \geq 0 \\ k+\ell+m=n-2}} F_D^k A F_D^\ell B F_D^m C, \quad J(A,B,C):=\int_X \Tr I(A,B,C),
\]
where $A, B, C$ are elements in $\Omega^1(\End(E,h))$ or $\Omega^2(\End(E,h))$. Then
\[
d(\tilde{\Omega}(b,c))_D(a)=J(Da,b,c)+J(b,Da,c)=J(Da,b,c)-J(Da,c,b).
\]
On the other hand, by applying the Bianchi identity, we observe that
\[
\begin{aligned}
0&=\int_X D \Tr \big( I(a,b,c)-I(a,c,b) \big) \\
&=\int_X \Tr \bigg( D \big( I(a,b,c)-I(a,c,b) \big) \bigg) \\
&=J(Da,b,c)-J(Da,c,b)-J(a,Db,c)+J(a,Dc,b)+J(a,b,Dc)-J(a,c,Db) \\
&=J(Da,b,c)-J(Da,c,b)+J(Db,c,a)-J(Db,a,c)+J(Dc,a,b)-J(Dc,b,a) \\
&=d(\tilde{\Omega}(b,c))_D(a)-d(\tilde{\Omega}(a,c))_D(b)+d(\tilde{\Omega}(a,b))_D(c)
\end{aligned}
\]
as desired.

Finally, we will check that $\mu$ gives the moment map for the action of $\cG$ on $(\cA_{\inte}(E,h), \Omega)$ by showing that for any $a \in \Omega^1(\End(E,h))$ and $v \in \cG$, we have
\[
\begin{aligned}
(d \mu_D(v))(a)&=\i \int_X \Tr \bigg[ v \bigg( c \sum_{k=0}^{n-1} F_D^k \bigg( \frac{\i}{2 \pi} D a \bigg) F_D^{n-1-k}-\omega \Id_E \sum_{k=0}^{n-2} F_D^k \bigg( \frac{\i}{2 \pi} D a \bigg) F_D^{n-2-k} \bigg) \bigg]\\
&=\frac{1}{2 \pi} \int_X \Tr \bigg[ Dv \bigg( c \sum_{k=0}^{n-1} F_D^k a F_D^{n-1-k}-\omega \Id_E \sum_{k=0}^{n-2} F_D^k a F_D^{n-2-k} \bigg) \bigg]\\
&=\frac{1}{2 \pi} \int_X \Tr \bigg[ c \sum_{k=0}^{n-1}F_D^{n-1-k} (Dv) F_D^k a-\omega \Id_E \sum_{k=0}^{n-2} F_D^{n-2-k} (Dv) F_D^k a \bigg]\\
&=-\Omega_D(Dv,a).
\end{aligned}
\]
\end{proof}

\section{$J$-stability for rank-$2$ vector bundles over compact K\"ahler surfaces} \label{J-stability for rank-2 vector bundles over compact Kahler surfaces}
In this section, we will show the following:
\begin{thm}[Theorem \ref{J-stability for rank-2 bundles over surfaces}]
Let $E$ be a holomorphic vector bundle of rank-$2$ over a compact K\"ahler surface $(X,\omega)$ satisfying \eqref{positivity for Chern characters}. Assume that $E$ admits a $J$-Griffiths positive Hermitian metric $h$ solving \eqref{J-equation}. Then $E$ is $J$-semistable for all subbundles. Moreover, $E$ is $J$-stable for all subbundles if it is indecomposable.
\end{thm}
\begin{proof}
For a holomorphic subbundle $S \subset E$ of $\rank S>0$, we recall the Gauss--Codazzi equation
\[
F_h=
\begin{pmatrix}
F_S-\frac{\i}{2 \pi} A^\ast A & -\frac{\i}{2 \pi} D' A^\ast \\
\frac{\i}{2 \pi} D'' A & F_Q-\frac{\i}{2 \pi} A A^\ast
\end{pmatrix},
\]
where $Q:=E/S$ is the quotient bundle and $A \in \Omega^{1,0}(\Hom(S,Q))$ denotes the second fundamental form. If $h$ satisfies \eqref{J-equation}, then we have
\begin{equation} \label{equation on S}
c \Tr_S \bigg[ \bigg(F_S-\frac{\i}{2 \pi}A^\ast A \bigg)^2 \bigg]-c \bigg( \frac{\i}{2 \pi} \bigg)^2 \Tr_S (D'A^\ast D'' A)-\omega \bigg( \Tr_S (F_S)-\frac{\i}{2 \pi} \Tr_S (A^\ast A) \bigg)=0,
\end{equation}
\begin{equation} \label{equation on Q}
c \Tr_Q \bigg[ \bigg(F_Q-\frac{\i}{2 \pi}A A^\ast \bigg)^2 \bigg]-c \bigg( \frac{\i}{2 \pi} \bigg)^2 \Tr_Q (D''A D' A^\ast)-\omega \bigg( \Tr_Q (F_Q)-\frac{\i}{2 \pi} \Tr_Q (A A^\ast) \bigg)=0.
\end{equation}
The equation \eqref{equation on S} shows that
\begin{equation} \label{expansion of the equation on S}
\begin{aligned}
& 2c \ch_2(S,h|_S)-\frac{\i}{\pi} c \Tr_S (F_S A^\ast A)-c \bigg( \frac{\i}{2 \pi} \bigg)^2 \Tr_S (D'A^\ast D'' A)-\omega \ch_1(S,h|_S) \\
&+\frac{\i}{2 \pi} \omega \Tr_S (A^\ast A)=0,
\end{aligned}
\end{equation}
where we note that $\Tr_S(A^\ast A A^\ast A)=0$ since $E$ has rank-$2$. Likewise the equation \eqref{equation on Q} yields that
\begin{equation} \label{expansion of the equation on Q}
\begin{aligned}
&2c \ch_2(Q,h|_Q)-\frac{\i}{\pi} c \Tr_Q (F_Q A A^\ast)-c \bigg( \frac{\i}{2 \pi} \bigg)^2 \Tr_Q (D'' A D'A^\ast)-\omega \ch_1(Q,h|_Q)\\
&+\frac{\i}{2 \pi} \omega \Tr_Q (A A^\ast)=0.
\end{aligned}
\end{equation}
Subtracting \eqref{expansion of the equation on Q} from \eqref{expansion of the equation on S} and integrating over $X$, we obtain
\[
4c \ch_2(S)-2 [\omega] \cdot \ch_1(S)+\frac{\i}{\pi} \int_X \big(c \cdot c_1(E,h)-\omega \big) \Tr (A A^\ast)=2c \ch_2(E)-[\omega] \cdot \ch_1(E).
\]
Since $h$ is $J$-Griffiths positive, we know that $c \cdot c_1(E,h)-\omega>0$. Thus we have
\[
4c \ch_2(S)-2 [\omega] \cdot \ch_1(S) \leq 2c \ch_2(E)-[\omega] \cdot \ch_1(E).
\]
Substituting $c=([\omega] \cdot \ch_1(E))/(2 \ch_2(E))$, we conclude that
\begin{equation} \label{inequality in general case}
\ch_2(S) \cdot \big([\omega] \cdot \ch_1(E) \big) \leq \big([\omega] \cdot \ch_1(S) \big) \cdot \ch_2(E).
\end{equation}
Moreover, the integral
\[
\i \int_X \big(c \cdot c_1(E,h)-\omega \big) \Tr (A A^\ast)
\]
vanishes if and only if $A=0$. Thus if $E$ is indecomposable and $\rank S=1$, the inequality \eqref{inequality in general case} is strict.
\end{proof}

\section{Examples} \label{Examples}
\subsection{Projective spaces} \label{Projective spaces}
Let $\omega_{\FS} \in c_1(\cO(1))$ be the Fubini--Study metric on a projective space $\C\P^n$. The metric $\omega_{\FS}$ produces a Hermitian metric $H$ on the holomorphic tangent bundle $T' \C\P^n$ with curvature form $F_H$. In this subsection, we will show that the Hermitian metric $H$ satisfies the $J$-equation
\begin{equation} \label{J-equation on projective spaces}
F_H^n-\omega_{\FS} \Id_{T' \C\P^n} F_H^{n-1}=0.
\end{equation}
Since the isometry group $\SU(n+1)$ acts on transitively on $\C\P^n$, it suffices to show \eqref{J-equation on projective spaces} at a point $[1,0,\ldots,0] \in \C\P^n$. We take a Euclidean coordinates $z=(z^1,\ldots,z^n)$ centered at $[1,0,\ldots,0]$, on which the Fubini--Study metric $\omega_{\FS}$ can be written as
\[
\omega_{\FS}=\frac{\i}{2 \pi} \p \bp \log (1+|z|^2)=\frac{\i}{2\pi} \sum_{i,j} H_{ij} dz^i dz^{\bar{j}},
\]
where the Hermitian matrix $(H_{ij})$ is given by
\[
H_{ij}:=\frac{1}{1+|z|^2} \tilde{H}_{ij}, \quad \tilde{H}_{ij}:=\d_{ij}-\frac{z^i z^{\bar{j}}}{1+|z|^2}.
\]
Since
\[
\p H_{ij}=-\frac{1}{(1+|z|^2)^2} \sum_k z^{\bar{k}} \p z^k \tilde{H}_{ij}+\frac{1}{1+|z|^2} \p \tilde{H}_{ij},
\]
\[
\begin{aligned}
H^{-1} \p H&=(1+|z|^2) \tilde{H}^{-1} \bigg[ -\frac{1}{(1+|z|^2)^2} \sum_k z^{\bar{k}} \p z^k \tilde{H}+\frac{1}{1+|z|^2} \p \tilde{H} \bigg] \\
&=-\frac{1}{1+|z|^2} \sum_k z^{\bar{k}} \p z^k \Id_{T' \C\P^n}+\tilde{H}^{-1} \p \tilde{H} \\
&=-\p \log (1+|z|^2) \Id_{T' \C\P^n}+\tilde{H}^{-1} \p \tilde{H},
\end{aligned}
\]
the curvature $F_H$ is expressed as
\[
F_H=\frac{\i}{2\pi} \bp (H^{-1} \p H)=\omega_{\FS} \Id_{T' \C\P^n}+\frac{\i}{2\pi} \bp (\tilde{H}^{-1} \p \tilde{H}).
\]
On the other hand, we see that
\[
\begin{aligned}
\sum_k (\d_{ik}+z^i z^{\bar{k}}) \bigg( \d_{kj}-\frac{z^k z^{\bar{j}}}{1+|z|^2} \bigg)&=\sum_k \bigg( \d_{ik} \d_{kj}-\d_{ik} \frac{z^k z^{\bar{j}}}{1+|z|^2}+z^i z^{\bar{k}} \d_{kj}-\frac{z^k z^{\bar{k}} z^i z^{\bar{j}}}{1+|z|^2} \bigg) \\
&=\d_{ij}-\frac{z^i z^{\bar{j}}}{1+|z|^2}+z^i z^{\bar{j}}-\frac{|z|^2 z^i z^{\bar{j}}}{1+|z|^2} \\
&=\d_{ij}.
\end{aligned}
\]
So the inverse matrix $(\tilde{H}^{ik})$ of $(\tilde{H}_{ik})$ satisfies
\[
\tilde{H}^{ik}=\d_{ik}+z^i z^{\bar{k}}.
\]
By using this, we compute the term $\bp (\tilde{H}^{-1} \p \tilde{H})$ as
\[
\p \tilde{H}_{kj}=\frac{1}{(1+|z|^2)^2} \bigg( \sum_{\ell} z^{\bar{\ell}} \p z^\ell \bigg) z^k z^{\bar{j}}-\frac{1}{1+|z|^2} z^{\bar{j}} \p z^k,
\]
\[
\begin{aligned}
(\tilde{H}^{-1} \p \tilde{H})_{ij}&=\sum_k \tilde{H}^{ik} \p \tilde{H}_{kj} \\
&=\sum_k (\d_{ik}+z^i z^{\bar{k}}) \bigg[ \frac{1}{(1+|z|^2)^2} \bigg( \sum_{\ell} z^{\bar{\ell}} \p z^\ell \bigg) z^k z^{\bar{j}}-\frac{1}{1+|z|^2} z^{\bar{j}} \p z^k \bigg] \\
&=\frac{1}{(1+|z|^2)^2} \bigg( \sum_{\ell} z^{\bar{\ell}} \p z^\ell \bigg) z^i z^{\bar{j}}-\frac{1}{1+|z|^2} z^{\bar{j}} \p z^i+\frac{|z|^2}{(1+|z|^2)^2} \bigg( \sum_\ell z^{\bar{\ell}} \p z^\ell \bigg) z^i z^{\bar{j}}\\
&-\frac{1}{1+|z|^2} \bigg( \sum_k z^{\bar{k}} \p z^k \bigg) z^i z^{\bar{j}} \\
&=-\frac{1}{1+|z|^2} z^{\bar{j}} \p z^i,
\end{aligned}
\]
\[
\bp (\tilde{H}^{-1} \p \tilde{H})_{ij}=\frac{1}{1+|z|^2} \p z^i \bp z^{\bar{j}}+\frac{1}{(1+|z|^2)^2} \bigg( \sum_k z^k \bp z^{\bar{k}} \bigg) z^{\bar{j}} \p z^i.
\]
So at $z=0$, we obtain the following expression of the curvature $F_H$
\[
(F_H)_{\a \b}=\frac{\i}{2 \pi} \sum_{i,j} \d_{ij} dz^i dz^{\bar{j}} \d_{\a \b}+\frac{\i}{2 \pi} dz^\a dz^{\bar{\b}}.
\]
Thus we can compute $(F_H)^n$ as
\[
\begin{aligned}
& \big( (F_H)^n \big)_{\a \b} \\
&=\omega_{\FS}^n \d_{\a \b}+\sum_{r=0}^{n-1} \dbinom{n}{r} \omega_{\FS}^r \sum_{k_1,\ldots,k_{n-r-1}} \bigg( \frac{\i}{2 \pi} dz^\a dz^{\overline{k_1}} \bigg) \bigg( \frac{\i}{2 \pi} dz^{k_1} dz^{\overline{k_2}} \bigg) \cdots \bigg( \frac{\i}{2 \pi} dz^{k_{n-r-1}} dz^{\bar{\b}} \bigg) \\
&=\omega_{\FS}^n \d_{\a \b}+\sum_{r=0}^{n-1} \dbinom{n}{r} \omega_{\FS}^r (-\omega_{\FS})^{n-r-1} \bigg( \frac{\i}{2 \pi} dz^\a dz^{\bar{\b}} \bigg) \\
&=\omega_{\FS}^n \d_{\a \b}+(n-1)! \bigg( \frac{\i}{2 \pi} \bigg)^{n-1} \sum_{j=1}^n (dz^1 dz^{\bar{1}} \cdots \hat{dz^j dz^{\bar{j}}} \cdots dz^n dz^{\bar{n}})\\
& \wedge \sum_{r=0}^{n-1} \dbinom{n}{r} (-1)^{n-r-1} \bigg( \frac{\i}{2 \pi} dz^\a dz^{\bar{\b}} \bigg) \\
&=\omega_{\FS}^n \d_{\a \b} \bigg(1+\frac{1}{n} \sum_{r=0}^{n-1} \dbinom{n}{r} (-1)^{n-r-1} \bigg) \\
&=\bigg(1+\frac{1}{n} \bigg) \omega_{\FS}^n \d_{\a \b},
\end{aligned}
\]
where the notation $\widehat{f}$ means eliminating the term $f$. Similarly, we have
\[
\begin{aligned}
& \big( (F_H)^{n-1} \big)_{\g \b} \\
&=\omega_{\FS}^{n-1} \d_{\g \b}+\sum_{r=0}^{n-2} \dbinom{n-1}{r} \omega_{\FS}^r \sum_{k_1,\ldots,k_{n-r-2}} \bigg( \frac{\i}{2 \pi} dz^\g dz^{\overline{k_1}} \bigg) \bigg( \frac{\i}{2 \pi} dz^{k_1} dz^{\overline{k_2}} \bigg) \cdots \bigg( \frac{\i}{2 \pi} dz^{k_{n-r-2}} dz^{\bar{\b}} \bigg) \\
&=\omega_{\FS}^{n-1} \d_{\g \b}+\sum_{r=0}^{n-2} \dbinom{n-1}{r} \omega_{\FS}^r (-\omega_{\FS})^{n-r-2} \bigg( \frac{\i}{2 \pi} dz^\g dz^{\bar{\b}} \bigg).
\end{aligned}
\]
So multiplying $\omega_{\FS} \Id_{T' \C\P^n}$ to the above, we obtain
\[
\begin{aligned}
\big( (\omega_{\FS} \Id_{T' \C\P^n} \big) (F_H)^{n-1} \big)_{\a \b}&=\sum_{\g=1}^n \omega_{\FS} \d_{\a \g} \bigg[ \omega_{\FS}^{n-1} \d_{\g \b}+\sum_{r=0}^{n-2} \dbinom{n-1}{r} \omega_{\FS}^r (-\omega_{\FS})^{n-r-2} \bigg( \frac{\i}{2 \pi} dz^\g dz^{\bar{\b}} \bigg) \bigg] \\
&=\omega_{\FS}^n \d_{\a \b}+\omega_{\FS}^{n-1} \wedge \sum_{r=0}^{n-2} \dbinom{n-1}{r}(-1)^{n-r-2} \bigg( \frac{\i}{2 \pi} dz^\a dz^{\bar{\b}} \bigg) \\
&=\omega_{\FS}^n \d_{\a \b}+(n-1)! \bigg( \frac{\i}{2 \pi} \bigg)^{n-1} \sum_{j=1}^n(dz^1 dz^{\bar{1}} \cdots \hat{dz^j dz^{\bar{j}}} \cdots dz^n dz^{\bar{n}}) \\
& \wedge \sum_{r=0}^{n-2} \dbinom{n-1}{r}(-1)^{n-r-2} \bigg( \frac{\i}{2 \pi} dz^\a dz^{\bar{\b}} \bigg) \\
&=\omega_{\FS}^n \d_{\a \b} \bigg( 1+\frac{1}{n} \sum_{r=0}^{n-2} \dbinom{n-1}{r}(-1)^{n-r-2} \bigg) \\
&=\bigg(1+\frac{1}{n} \bigg) \omega_{\FS}^n \d_{\a \b}.
\end{aligned}
\]
From the above computations, we know that
\begin{equation} \label{cohomological invariants for projective spaces}
n \ch_n(T' \C\P^n)=[\omega_{FS}] \cdot \ch_{n-1}(T' \C\P^n)=\frac{n+1}{(n-1)!}
\end{equation}
and the Hermitian metric $H$ solves \eqref{J-equation on projective spaces}.

Now we assume $n=2$. We will check the $J$-positivity of $H$, \ie for all $a'' \in \Omega^{0,1}(\End(T' \C\P^n))$, the inequality
\[
-\i \Tr \bigg(a''F_H(a'')^\ast+F_Ha'' (a'')^\ast-\omega_{\FS} \Id_{T' \C\P^2} a'' (a'')^\ast \bigg)>0
\]
holds at all points on $\C\P^2$ where $a'' \neq 0$. By symmetry, it suffices once again to prove this at $z=0$. Suppose $a''=(\overline{a_{\b\a}})$ is a $2 \times 2$ matrix valued $(0,1)$-form, where $\overline{a_{\b\a}}$ denotes the $(\a,\b)$-component of $a''$, \ie $(a'')^\ast=(a_{\a\b})$. We set $a_{\a\b}=\sum_\mu a_{\a\b,\mu} dz^\mu$. Then we compute each term as
\[
\begin{aligned}
-\i \Tr \big(a''F_H (a'')^\ast \big)&=-\i \sum_{\a,\b,\g} \overline{a_{\b\a}} \bigg( \omega_{\FS} \d_{\b \g}+\frac{\i}{2\pi} dz^\b dz^{\bar{\g}} \bigg) a_{\g \a}\\
&=\i \sum_{\a,\b} \omega_{\FS} a_{\b\a} \overline{a_{\b\a}}-\frac{1}{2 \pi} \sum_{\a,\b,\g} dz^\b dz^{\bar{\g}} a_{\g\a} \overline{a_{\b\a}},
\end{aligned}
\]
\[
\begin{aligned}
-\i \Tr \big( F_H a'' (a'')^\ast \big)&=\i \Tr \big( (a'')^\ast F_H a'' \big)\\
&=\i \sum_{\a,\b,\g} a_{\a\b} \bigg( \omega_{\FS} \d_{\b \g}+\frac{\i}{2\pi} dz^\b dz^{\bar{\g}} \bigg) \overline{a_{\a\g}}\\
&=\i \sum_{\a,\b} \omega_{\FS} a_{\a\b} \overline{a_{\a\b}}-\frac{1}{2 \pi} \sum_{\a,\b,\g} dz^\b dz^{\bar{\g}} a_{\a\b} \overline{a_{\a\g}},
\end{aligned}
\]
\[
\i \Tr \big( \omega_{\FS} \Id_{T' \C\P^2} a'' (a'')^\ast \big)
=\i \sum_{\a,\b} \omega_{\FS} a_{\a\b} \overline{a_{\a\b}}
=\frac{(\i)^2}{2 \pi} dz^1 dz^{\bar{1}} dz^2 dz^{\bar{2}} \sum_{\a,\b,\mu}|a_{\a\b,\mu}|^2.
\]
Thus we have
\[
\begin{aligned}
&-\i \Tr \bigg(a''F_H (a'')^\ast+F_H a'' (a'')^\ast-\omega_{\FS} \Id_{T' \C\P^2} a'' (a'')^\ast \bigg) \\
&=\frac{(\i)^2}{2 \pi} dz^1 dz^{\bar{1}} dz^2 dz^{\bar{2}} \bigg( \sum_{\a,\b,\mu} |a_{\a\b,\mu}|^2+\sum_\a |a_{1\a,2}|^2+\sum_\a |a_{2\a,1}|^2-2\Re \bigg( \sum_\a a_{2\a,2} \overline{a_{1\a,1}} \bigg) \\
&+\sum_\a |a_{\a1,2}|^2+\sum_\a |a_{\a2,1}|^2-2 \Re \bigg( \sum_\a a_{\a1,2}\overline{a_{\a2,1}} \bigg) \bigg).
\end{aligned}
\]
By Cauchy--Schwarz inequality, we observe that
\begin{equation} \label{CS1}
2\Re \bigg( \sum_\a a_{2\a,2} \overline{a_{1\a,1}} \bigg) \leq 2 \sqrt{\sum_\a |a_{2\a,2}|^2} \sqrt{\sum_\a |a_{1\a,1}|^2}
\leq \sum_\a |a_{2\a,2}|^2+\sum_\a |a_{1\a,1}|^2,
\end{equation}
\begin{equation} \label{CS2}
2 \Re \bigg( \sum_\a a_{\a1,2}\overline{a_{\a2,1}} \bigg) \leq 2 \sqrt{\sum_\a |a_{\a1,2}|^2} \sqrt{\sum_\a |a_{\a2,1}|^2} \leq \sum_\a |a_{\a1,2}|^2+\sum_\a |a_{\a2,1}|^2.
\end{equation}
Using these inequalities, we have
\begin{equation} \label{semipositivity estimate}
\begin{aligned}
&-\i \Tr \bigg(a''F_H (a'')^\ast+F_H a'' (a'')^\ast-\omega_{\FS} \Id_{T' \C\P^2} a'' (a'')^\ast \bigg) \\
& \geq \frac{(\i)^2}{\pi} dz^1 dz^{\bar{1}} dz^2 dz^{\bar{2}} \bigg( |a_{11,2}|^2+|a_{12,2}|^2+|a_{21,1}|^2+|a_{22,1}|^2  \bigg) \\
&\geq 0.
\end{aligned}
\end{equation}
Moreover, the equality holds if and only if the equality holds in all of the three inequalities \eqref{CS1}, \eqref{CS2}, \eqref{semipositivity estimate}, \ie
\[
a_{11,2}=a_{12,2}=a_{21,1}=a_{22,1}=0,
\]
\[
\sum_\a |a_{2\a,2}|^2=\sum_\a |a_{1\a,1}|^2, \quad \sum_\a |a_{\a1,2}|^2=\sum_\a |a_{\a2,1}|^2,
\]
and sets of vectors $\big((a_{2 \a,2})_{\a=1,2}, (a_{1\a,1})_{\a=1,2} \big)$, $\big((a_{\a1,2})_{\a=1,2}, (a_{\a2,1})_{\a=1,2} \big)$ are linearly dependent. This occurs if and only if $a''=0$. Thus $H$ is $J$-positive. Since $T' \C\P^2$ is simple, we can apply Lemma \ref{small deformation} to know that for any K\"ahler form $\omega$ sufficiently close to $\omega_{\FS}$ (in $C^{\ell,\b}$ for some $\ell \in \Z_{\geq 0}$ and $\b \in (0,1)$), we have a solution to \eqref{J-equation} with respect to $\omega$.
\begin{rk}
In the case $n=2$, we can also check \eqref{cohomological invariants for projective spaces} directly by using a well-known formula for Chern classes of $T' \C\P^n$
\[
c_k(T' \C\P^n)=\dbinom{n+1}{k} c_1(\cO(1))^k \quad (k=0,1,\ldots,n).
\]
\end{rk}

\subsection{Sufficiently smooth vector bundles and asymptotic $J$-stability} \label{Sufficiently smooth vector bundles and asymptotic J-stability}
Let $L$ be an ample line bundle over a compact K\"ahler surface $(X,\omega)$ such that
\[
\big( 2([\omega] \cdot c_1(L)) c_1(L)-c_1(L)^2 [\omega] \big) \cdot Y>0
\]
for any curve $Y \subset X$. By the Nakai--Moishezon type criterion \cite{Son20}, this condition is equivalent to the existence of a solution $\chi \in c_1(L)$
\begin{equation} \label{J-equation on surfaces repeat}
\frac{[\omega] \cdot c_1(L)}{c_1(L)^2} \chi^2-\omega \chi=0.
\end{equation}
In particular, $\eta:=2([\omega] \cdot c_1(L)) \chi-c_1(L)^2 \omega$ is a K\"ahler form. In this subsection, we will prove the following:
\begin{thm}[Theorem \ref{existence of a solution on twisted bundles}] \label{existence of a solution on twisted bundles repeat}
Let $(X,\omega)$, $L$, $\chi$ ,$\eta$ as above, $E$ a holomorphic vector bundle of rank $r$ over $X$, $h_L \in \Herm(L)$ a Hermitian metric with curvature $\chi$ and $D_L$ the associated connection of $h_L$. Let us consider the $J$-equation for $\tilde{h}_k \in \Herm(E \otimes L^k)$ with respect to $\omega$
\begin{equation} \label{J-equation on twisted bundles repeat}
\frac{[\omega] \cdot \ch_1(E \otimes L^k)}{2 \ch_2(E \otimes L^k)} F_{\tilde{h}_k}^2-\omega \Id_E F_{\tilde{h}_k}=0.
\end{equation}
Then we have the following:
\begin{enumerate}
\item The bundle $E \otimes L^k$ satisfies the property \eqref{positivity for Chern characters} with respect to $\omega$ for sufficiently large integer $k$.
\item Assume that there exist a Hermitian metric $h_0 \in \Herm(E)$ and a sequence of $h_0$-compatible integrable connections $D_k$ on $E$ satisfying
\begin{enumerate}
\item The pair $(h_0 \otimes h_L^k,D_k \otimes D_L^k)$ solves \eqref{J-equation on twisted bundles}.
\item $D_k$ are uniformly bounded in $C^1$.
\end{enumerate}
Then for any subbundle $S \subset E$ with $0<\rank S<r$ that is holomorphic with respect to all $D_k$, we have $\varphi_k(S) \leq \varphi_k(E)$ for sufficiently large $k$, with the strict inequality holding on each level $k$ if $E$ is indecomposable with respect to $D_k$.
\item Assume that $E$ is an asymptotically $J$-stable, sufficiently smooth vector bundle and the associated graded object $\Gr(E)$ has at most $2$ Mumford stable components. Then for sufficiently large integer $k$, there exists a $J$-positive Hermitian metric $\tilde{h}_k \in \Herm(E \otimes L^k)$ satisfying \eqref{J-equation on twisted bundles repeat}.
\end{enumerate}
\end{thm}
When we discuss the asymptotic $J$-stability, it is convenient to express $c_k$ in terms of $\varphi_k(E)$ as
\begin{equation} \label{c and varphi}
c_k:=\frac{[\omega] \cdot \ch_1(E \otimes L^k)}{2 \ch_2(E \otimes L^k)}=\frac{1}{2 \varphi_k(E)}.
\end{equation}
Then the first clue is to compute the asymptotic expansion of the asymptotic $J$-slope as follows:
\begin{prop} \label{Mumford stability vs asymptotic J-stability}
Let $(X,\omega)$, $L$, $\eta$ be as in Theorem \ref{existence of a solution on twisted bundles repeat} and $\cE$ a coherent sheaf over $X$. Then we have the following:
\begin{enumerate}
\item If $\cE$ is asymptotically $J$-semistable, then it is Mumford semistable with respect to $\eta$.
\item If $\cE$ is Mumford stable with respect to $\eta$, then it is asymptotically $J$-stable.
\end{enumerate}
\end{prop}
\begin{proof}
Let $\cS$ be any coherent subsheaf with $s:=\rank \cS>0$. By using the relation $[\eta]=2([\omega] \cdot c_1(L)) \cdot c_1(L)-c_1(L)^2 \cdot [\omega]$, we observe that
\[
\begin{aligned}
\varphi_k(\cS)&=\frac{\frac{1}{2}k^2s c_1(L)^2+k c_1(\cS) \cdot c_1(L)+\ch_2(\cS)}{ks ([\omega] \cdot c_1(L))+([\omega] \cdot c_1(\cS))} \\
&=\frac{c_1(L)^2}{2([\omega] \cdot c_1(L))}k+\frac{\mu_\eta(\cS)}{2([\omega] \cdot c_1(L))^2}+O(k^{-1}),
\end{aligned}
\]
where $\mu_\eta(\cS)=(c_1(\cS) \cdot [\eta])/s$. Then the desired properties follow immediately from this formula.
\end{proof}
In particular, the above formula shows that
\begin{equation} \label{asymptotic expansion of ck}
c_k=\frac{[\omega] \cdot c_1(L)}{c_1(L)^2}k^{-1}-\frac{\mu_\eta(E)}{\big( c_1(L)^2 \big)^2}k^{-2}+O(k^{-3}).
\end{equation}

\begin{proof}[Proof of (1) and (2) in Theorem \ref{existence of a solution on twisted bundles}]
The property (1) follows from Proposition \ref{Mumford stability vs asymptotic J-stability}. As for (2), applying the Gauss--Codazzi equation to each pair $(h_0,D_k)$, we know that the curvature $F_k$ of $D_k$ satisfies
\[
F_k=
\begin{pmatrix}
F_{S,k}-\frac{\i}{2 \pi} A_k^\ast A_k & \frac{\i}{2 \pi} D_k' A_k^\ast \\
-\frac{\i}{2 \pi} D_k'' A_k & F_{Q,k}-\frac{\i}{2 \pi} A_k A_k^\ast
\end{pmatrix},
\]
where $Q:=E/S$ and the quantities $F_{S,k}$, $F_{Q,k}$, $A_k$ are defined with respect to $(D_k,h_0)$. Thus the curvature $\tilde{F}_k$ of $(D_k \otimes D_L^k,h_0 \otimes h_L^k)$ is given by
\[
\tilde{F}_k=
\begin{pmatrix}
F_{S,k}+k \chi \Id_S-\frac{\i}{2 \pi} A_k^\ast A_k & \frac{\i}{2 \pi} D_k' A_k^\ast \\
-\frac{\i}{2 \pi} D_k'' A_k & F_{Q,k}+k \chi \Id_Q-\frac{\i}{2 \pi} A_k A_k^\ast
\end{pmatrix}.
\]
By the assumption, we know that $A_k$ and $F_k$ are uniformly bounded in $C^1$ and $C^0$ respectively. Since the pair $(h_0 \otimes h_L^k,D_k \otimes D_L^k)$ satisfies \eqref{J-equation on twisted bundles repeat}, we can apply the similar computation as in the proof of Theorem \ref{J-stability for rank-2 bundles over surfaces} to the decomposition $E \otimes L^k=(S \otimes L^k) \oplus (Q \otimes L^k)$ while taking account of the non-commutativity of endomorphisms in higher rank case. As a result, we obtain
\[
\begin{aligned}
&4c_k \ch_2(S \otimes L^k)-2[\omega] \cdot \ch_1(S \otimes L^k)-2c_k \bigg( \frac{\i}{2 \pi} \bigg)^2 \int_X \Tr (A_k A_k^\ast A_k A_k^\ast) \\
&+\frac{\i}{\pi} c_k \int_X \Tr \big(A_k (F_{S,k}+k \chi \Id_S) A_k^\ast+(F_{Q,k}+k \chi \Id_Q) A_kA_k^\ast \big)-\frac{\i}{\pi} \int_X \omega \Tr (A_kA_k^\ast)=0.
\end{aligned}
\]
Now we will show that the sum of the last three terms in the left hand side is positive. We note that
\[
\frac{\i}{\pi} \omega \Tr(A_kA_k^\ast)=|A_k|^2 \omega^2,
\]
where the norm $|A_k|^2$ is taken with respect to $h_0$ and $\omega$. By using these facts and uniform $C^1$ bound of $A_k$, we observe that
\[
2c_k \bigg( \frac{\i}{2 \pi} \bigg)^2 \int_X \Tr (A_k A_k^\ast A_k A_k^\ast) \leq Ck^{-1} \cdot \i \int_X \omega \Tr (A_k A_k^\ast)
\]
for some uniform constant $C>0$ (independent of $k$). Similarly, by \eqref{asymptotic expansion of ck} and uniform $C^0$ bound of $F_k$, there exists a uniform constant $C'>0$ such that
\[
\begin{aligned}
& \frac{\i}{\pi} c_k \int_X \Tr \big(A_k (F_{S,k}+k \chi \Id_S) A_k^\ast+(F_{Q,k}+k \chi \Id_Q) A_kA_k^\ast \big) \\
& \geq \frac{\i}{\pi c_1(L)^2} \int_X 2 ([\omega] \cdot c_1(L)) \chi \Tr (A_k A_k^\ast)-\i C'k^{-1} \int_X \omega \Tr (A_k A_k^\ast).
\end{aligned}
\]
Thus by using $\eta=2([\omega] \cdot c_1(L)) \chi-c_1(L)^2 \omega$, we have
\[
\begin{aligned}
&4c_k \ch_2(S \otimes L^k)-2[\omega] \cdot \ch_1(S \otimes L^k)
+\frac{\i}{\pi c_1(L)^2} \int_X (\eta-C'' k^{-1} \omega) \Tr (A_kA_k^\ast) \leq 0
\end{aligned}
\]
for some uniform constant $C''>0$. Since $\eta$ is K\"ahler, this together with \eqref{c and varphi} implies that
\[
\varphi_k(S) \leq \varphi_k(E)
\]
for sufficiently large $k$. Moreover, if $E$ is indecomposable, then each $A_k$ is non-trivial. So the above inequality must be strict.
\end{proof}
\begin{rk}
In the case $r=2$, we can replace the assumption in (2) with the existence of $J$-Griffiths positive Hermitian metrics $\tilde{h}_k \in \Herm(E \otimes L^k)$ satisfying \eqref{J-equation on twisted bundles repeat} for sufficiently large $k$. Indeed, for any subbundles $S \subset E$ of $\rank S=1$, we can apply Theorem \ref{J-stability for rank-2 bundles over surfaces} directly to the subbundle $S \otimes L^k \subset E \otimes L^k$ in this case.
\end{rk}
\subsubsection{The Mumford stable case}
Now we will prove Theorem \ref{existence of a solution on twisted bundles repeat} (3) when the associated graded object $\Gr(E)$ has exactly $1$ component, or equivalently, $E$ is Mumford stable with respect to $\eta$.
\begin{thm} \label{existence theorem in Mumford stable case}
Let $(X,\omega)$, $L$, $\eta$ as in Theorem \ref{existence of a solution on twisted bundles repeat}. Assume that $E$ is Mumford stable with respect to $\eta$. Then for sufficiently large integer $k$, there exists a $J$-positive Hermitian metric $\tilde{h}_k \in \Herm(E \otimes L^k)$ satisfying \eqref{J-equation on twisted bundles repeat}.
\end{thm}
To show this, for a fixed $h_0 \in \Herm(E)$, we first consider the asymptotic expansion of the $J$-operator
\[
\cP_\e(D):=\frac{1}{2 \varphi_\e(E)} (F_D+\e^{-1} \chi \Id_E)^2-\omega \Id_E \wedge (F_D+\e^{-1} \chi \Id_E), \quad D \in \cA_{\inte}(E,h_0)
\]
where $\e:=1/k$ and $\varphi_\e(E):=\varphi_k(E)$.
\begin{lem} \label{P is a deformation of P0}
We have
\begin{equation} \label{expansion of the linearized operator}
\cP_\e(D)=\cP_0(D)+O(\e),
\end{equation}
where $\cP_0$ denotes the weak Hermitian--Einstein operator
\[
\cP_0(D):=\frac{1}{c_1(L)^2} \bigg( \eta F_D-\frac{\mu_\eta(E)}{c_1(L)^2} \chi^2 \Id_E \bigg).
\]
The linearization of $\cP_0$ coincides with the Laplacian up to scale
\[
\d|_{D,v} \cP_0=\frac{\i}{2 \pi c_1(L)^2} \eta \wedge (D'D''v-D''D'v)=\frac{1}{4 \pi c_1(L)^2} \Delta_{\eta,D}v, \quad v \in \i \Omega^0(\End(E,h_0)).
\]
\end{lem}
\begin{proof}
By using \eqref{asymptotic expansion of ck}, we have the following asymptotic expansion
\[
\cP_\e(D)=\bigg( \frac{[\omega] \cdot c_1(L)}{c_1(L)^2} \chi^2-\omega \chi \bigg) \Id_E \e^{-1}+\cP_0(D)+O(\e).
\]
Then the crucial point is that the coefficient of $\e^{-1}$ is zero since $\chi$ satisfies \eqref{J-equation on surfaces repeat}. Thus $\cP_\e$ is well-defined even when $\e=0$, and we have $\cP_\e=\cP_0+O(\e)$ as desired.
\end{proof}

\begin{proof}[Proof of Theorem \ref{existence theorem in Mumford stable case}]
By the Kobayashi--Hitchin correspondence \cite{Don85,UY86} and the argument \cite[Proposition 4.2.3]{Kob14}, we can solve the following weak Hermitian--Einstein equation for $h_0 \in \Herm(E)$
\begin{equation} \label{weak HE equation}
\eta F_{h_0}=\frac{\mu_\eta(E)}{c_1(L)^2} \chi^2 \Id_E.
\end{equation}
Indeed, one can easily check that the integral of both hand sides of \eqref{weak HE equation} are equal. Let $D_0$ be the associated connection of $h_0$. Now we try to find a connection $D \in \cA_{\inte}(E,h_0)$ satisfying $\cP_\e(D)=0$ with $\e=1/k$, which is equivalent to say that the pair $(h_0 \otimes h_L^k,D \otimes D_L^k)$ solves \eqref{J-equation on twisted bundles repeat}. For any $\ell \in \Z_{\geq 0}$ and $\b \in (0,1)$, let $\cV^{\ell,\b}$ be a Banach manifold consisting of all $C^{\ell,\b}$ Hermitian endomorphisms $v$ of $(E,h_0)$ satisfying $\int_X \Tr (v) \eta^2=0$. We regard $\cP_\e$ as an operator $\cP \colon \cV^{\ell+2,\b} \times \R \to \cV^{\ell,\b}$ by setting $D:=D_0^{\exp(v)}$. Then we have $\cP(0,0)=0$ from our choice of $D_0$. Since the linearization $\d|_{(0,0)} \cP$ coincides with the Laplacian up to scale and $E$ is simple, we know that the kernel of $\d|_{(0,0)} \cP$ is trivial, and $\d|_{(0,0)} \cP$ is surjective by Fredholm alternative. Hence we can apply the standard implicit function theorem to get a solution $D_k \in \cA_{\inte}(E,h_0)$.

Finally, from the convergence $D_k \to D_0$ in $C^{\ell+2,\b}$, we have an asymptotic expansion
\[
-\i \Tr \bigg( c_k \big( a'' F_{\tilde{D}_k} (a'')^\ast+F_{\tilde{D}_k} a'' (a'')^\ast \big)-\omega \Id_E a'' (a'')^\ast \bigg)
=-\frac{\i}{c_1(L)^2} \eta \Tr \big(a'' (a'')^\ast \big)+O(k^{-1})
\]
for any $a'' \in \Omega^{0,1}(\End(E))$, where $F_{\tilde{D}_k}$ denotes the curvature of $\tilde{D}_k:=D_k \otimes D_L^k$. This shows that $\tilde{D}_k$ is $J$-positive if $k$ is sufficiently large. In particular, we find that $\tilde{D}_k$ is actually smooth by the elliptic regularity. This completes the proof.
\end{proof}

\begin{rk} \label{higher dimensional case}
For a compact K\"ahler manifold $(X,\omega)$ of arbitrary dimension $n$, we assume that there is a solution $\chi \in c_1(L)$ to the $J$-equation
\[
\frac{[\omega] \cdot c_1(L)^{n-1}}{c_1(L)^n} \chi^n-\omega \chi^{n-1}=0
\]
with $\eta:=n([\omega] \cdot c_1(L)^{n-1}) \chi^{n-1}-(n-1) c_1(L)^n \omega \chi^{n-2}>0$. Then in the same way as in the $n=2$ case, we compute the asymptotic expansion of the $J$-operator $\cP_\e(D)$ as
\[
\cP_\e(D)=\bigg( \frac{[\omega] \cdot c_1(L)^{n-1}}{c_1(L)^n} \chi^n-\omega \chi^{n-1} \bigg) \Id_E \e^{1-n}+\frac{\e^{2-n}}{c_1(L)^n} \bigg( \eta F_D-\frac{c_1(E) \cdot [\eta]}{r c_1(L)^n} \chi^n \Id_E \bigg)+O(\e^{3-n}).
\]
Thus the $\e^{1-n}$-term is zero, and $\cP_\e(D_0)=O(\e^{3-n})$ holds if $D_0$ satisfies
\begin{equation} \label{equation arising from the leading term}
\eta F_{D_0}=\frac{c_1(E) \cdot [\eta]}{r c_1(L)^n} \chi^n \Id_E.
\end{equation}
As far as the author's knowledge, the solvability of \eqref{equation arising from the leading term} is not  well understood for $n \geq 3$, that is the only reason why we restrict ourselves to the case $n=2$.
\end{rk}

\subsubsection{The sufficiently smooth case; the two components case}
Now we will consider a more general situation. We start with the following potentially useful proposition, so called the ``see-saw'' property which relates $\varphi_k(\cS)$ with $\varphi_k(\cE/\cS)$ for any coherent subsheaf $\cS \subset \cE$, and can be proved similarly as in the Mumford stable case \cite[Proposition 5.7.4, Corollary 5.7.14]{Kob14}:
\begin{prop} \label{see-saw property}
Let $\cE$ be a torsion-free coherent sheaf and
\[
0 \to \cS \to \cE \to \cQ \to 0
\]
a short exact sequence of coherent sheaves over $X$ with $0<\rank \cS<\rank \cE$. Then for sufficiently large $k$, we have
\begin{equation} \label{additive property in short exact sequence}
\big( [\omega] \cdot \ch_1(\cS \otimes L^k) \big) (\varphi_k(\cE)-\varphi_k(\cS))+\big( [\omega] \cdot \ch_1(\cQ \otimes L^k) \big) (\varphi_k(\cE)-\varphi_k(\cQ))=0.
\end{equation}
In particular, we have the following:
\begin{enumerate}
\item If $\cE$ is asymptotically $J$-semistable (resp. $J$-stable), then $\varphi_k(\cE) \leq \varphi_k(\cQ)$ (resp. $\varphi_k(\cE)<\varphi_k(\cQ)$) holds for all quotient coherent sheaves $\cQ$ with $\rank \cQ>0$ (resp. $0<\rank \cQ<\rank \cE$) and sufficiently large $k$.
\item If $\cE$ is asymptotically $J$-stable, then it is simple.
\end{enumerate}
\end{prop}
\begin{proof}
(1) Assume that $\cE$ is asymptotically $J$-semistable. For any quotient coherent sheaf $\cQ$ with $0<\rank \cQ<\rank \cE$ and sufficiently large $k$, the inequality $\varphi_k(\cE) \leq \varphi_k(\cQ)$ follows immediately from \eqref{additive property in short exact sequence}. If $\rank \cQ=\rank \cE$, then we have $\cS=0$ since $\cE$ is torsion-free, and hence $\cE \simeq \cQ$, $\varphi_k(\cE)=\varphi_k(\cQ)$. Moreover, if $\cE$ is asymptotically $J$-stable, then for any $\cQ$ with $0<\rank \cQ<\rank \cE$, we have $0<\rank \cS<\rank \cE$, so the desired inequality $\varphi_k(\cE)<\varphi_k(\cQ)$ follows from the stability condition $\varphi_k(\cS)<\varphi_k(\cE)$ and \eqref{additive property in short exact sequence}.

\noindent
(2) If $\cE$ is asymptotically $J$-stable, then for any non-zero $u \in H^0(X,\End(\cE))$, we have $0<\rank(\Im u)<\rank(\cE)$ and a contradiction
\[
\varphi_k(\Im u)<\varphi_k(\cE)<\varphi_k(\cE/\Ker u)=\varphi_k(\Im u)
\]
unless $u$ is not an isomorphism. Now for any non-zero $f \in H^0(X,\End(\cE))$, we set $u:=f-\lambda \Id_\cE$, where $\lambda \in \C^\ast$ denotes any eigenvalue of the non-zero map $f_x \colon \cE_x \to \cE_x$ for some $x \in X$. Then $u$ can not be isomorphism since $u_x$ has non-trivial kernel. This yields that $f=\lambda \Id_\cE$ on $X$, and hence $\cE$ is simple.
\end{proof}
\begin{rk}
Much like the Mumford stability, one may expect that the condition $\varphi_k(\cE) \leq \varphi_k(\cQ)$ for all quotient coherent sheaves $\cQ$ with $\rank \cQ>0$ implies the asymptotic $J$-semistability of $\cE$. However, a problem occurs when considering coherent subsheaves $\cS \subset \cE$ with $\rank \cS=\rank \cE$. In this case, the additive property shows that
\[
\ch_i(\cE \otimes L^k)=\ch_i(\cS \otimes L^k)+\ch_i(\cQ \otimes L^k) \quad (i=1,2),
\]
where $\cQ:=\cE/\cS$ is a coherent torsion sheaf. For instance, if $\Supp \cQ$ has codimension $2$, then we have $\ch_1(\cQ \otimes L^k)=0$ and $\ch_2(\cQ \otimes L^k) \geq 0$ by \cite[Proposition 3.1]{SW15}, which implies the desired inequality $\varphi_k(\cS) \leq \varphi_k(\cE)$. However, it seems likely to the author that the inequality $\varphi_k(\cS)>\varphi_k(\cE)$ could happen in general.
\end{rk}

We recall that a {\it Jordan--H\"older filtration} of a holomorphic vector bundle $E$ is a filtration by coherent subsheaves
\[
0=:\cE_0 \subset \cE_1 \subset \ldots \subset \cE_\ell:=E
\]
such that the corresponding quotients $\cV_i:=\cE_i/\cE_{i-1}$ ($i=1,\ldots,\ell$) are Mumford stable with respect to $\eta$ and satisfy $\mu_\eta(\cV_i)=\mu_\eta(E)$. In particular, the associated graded object
\[
\Gr(E):=\bigoplus_{i=1}^\ell \cV_i
\]
is Mumford polystable in the sense that it is a direct sum of Mumford stable bundles with the same Mumford slope.
\begin{prop}[\cite{Kob14}, Theorem 5.7.18] \label{existence and uniqueness of the associated graded object}
Any Mumford semistable vector bundle $E$ over a compact K\"ahler manifold $X$ admits a Jordan--H\"older filtration. Such a filtration may not be unique, but the associated graded object $\Gr(E)$ is unique up to isomorphism.
\end{prop}
We are interested in the following smoothness property introduced by Leung \cite{Leu97}:
\begin{dfn}
We say that a Mumford semistable bundle $E$ is {\it sufficiently smooth} if the associated graded object $\Gr(E)$ is locally free.
\end{dfn}
In what follows, we assume that $E$ is asymptotically $J$-stable. Then by Proposition \ref{Mumford stability vs asymptotic J-stability} and Proposition \ref{existence and uniqueness of the associated graded object}, we know that $E$ is Mumford semistable with respect to $\eta$, and $E$ admits a Jordan--H\"older filtration. We further assume that $E$ is sufficiently smooth and the associated graded object $\Gr(E)$ has at most $2$ Mumford stable components. Our strategy is essentially the same as \cite{ST20}. The basic ideas are to construct approximate solutions to arbitrary order, and then apply the quantitative implicit function theorem. Especially, we will focus on the construction of approximate solutions to order $2q$, in which the asymptotic $J$-stability condition is used (Lemma \ref{approximate solutions to order 2q}). From now on, the Jordan--H\"older filtration takes the simple form
\[
0 \subset V_1 \subset E
\]
and graded object is
\[
\Gr(E)=V_1 \oplus V_2.
\]
As in the proof of Theorem \ref{existence theorem in Mumford stable case}, for each component $V_i$ ($i=1,2$), we take a Hermitian metric $h_i \in \Herm(V_i)$ solving a weak Hermitian--Einstein equation
\[
\eta F_{h_i}=\frac{\mu_\eta(V_i)}{c_1(L)^2} \chi^2 \Id_{V_i}
\]
with the associated connection $D_{V_i}$. Then the product metric $h_0:=h_1 \oplus h_2 \in \Herm(E)$ also satisfies the weak Hermitian--Einstein equation
\[
\eta F_{h_0}=\frac{\mu_\eta(E)}{c_1(L)^2} \chi^2 \Id_E
\]
since $\mu_\eta(V_1)=\mu_\eta(V_2)=\mu_\eta(E)$. Then the associated connection of $h_0$ is given by the product connection $D_0:=D_{V_1} \oplus D_{V_2}$. We note that $E$ is isomorphic to $\Gr(E)$ as $C^\infty$ complex vector bundles, but their holomorphic structures are different. The holomorphic structure on $E$ is explicitly given by
\[
D''_E=D_0''+\g
\]
for some $\g \in \Omega^{0,1}(\Hom(V_2,V_1)) \subset \Omega^{0,1}(\End(\Gr(E)))$, where we note that $\g \neq 0$ since $E$ is simple. Then the integrability condition yields that
\[
D_0'' \g+\g \wedge \g=0.
\]
We will further impose the gauge fixing condition
\[
D_0''^\ast \g=0,
\]
and set
\[
a:=\g-\g^\ast.
\]
By a standard computation together with the Bochner--Kodaira--Nakano identity \cite[page 60]{Kob14}
\begin{equation} \label{BKN identity}
[\Lambda_\eta, D_0']=\i D_0''^\ast, \quad [\Lambda_\eta,D_0'']=-\i D_0'^\ast,
\end{equation}
we obtain:

\begin{lem}[a special case of \cite{ST20}, Lemma 5.5] \label{properties for a and g}
We have the following:
\begin{enumerate}
\item The term $D_0 a$ is off-diagonal.
\item The term $a \wedge a$ is diagonal with $\Tr(a \wedge a)=0$.
\item $-\i \int_X \Tr(\Lambda_\eta \g \wedge \g^\ast)>0$.
\item $\Lambda_\eta(D_0 a)=0$.
\end{enumerate}
\end{lem}

We will use the following straightforward consequence from the asymptotic $J$-stability:

\begin{prop}
The Mumford stable components $V_1$, $V_2$ are not isomorphic.
\end{prop}
\begin{proof}
We apply the asymptotic $J$-stability condition for $V_1 \subset E$ and use the see-saw property (Proposition \ref{see-saw property}) to find that
\[
\varphi_\e(V_1)<\varphi_\e(E)<\varphi_\e(E/V_1)=\varphi_\e(V_2),
\]
for sufficiently small $\e>0$, which shows that $V_1$ is not isomorphic to $V_2$.
\end{proof}

Also, the following result only concerns the geometry of the Mumford polystable vector bundle $\Gr(E)$, and hence can be applied to our situation.
\begin{cor}[\cite{ST20}, Lemma 5.8]
We have $\Aut(\Gr(E))=\C^\ast \Id_{V_1} \times \C^\ast \Id_{V_2} \simeq (\C^\ast)^2$.
\end{cor}

We recall that the local deformation space of holomorphic structures near $D_0''$ is given by an analytic space
\[
\cD:=\{\b \in \Omega^{0,1}(\End(\Gr(E)))| D_0'' \b+\b \wedge \b=0, \; D_0''^\ast \b=0 \}.
\]
Meanwhile, it is well-known that every holomorphic endomorphism is parallel with respect to the weak Hermitian--Einstein connection \cite[Theorem 3.1.38]{Kob14}, which together with \eqref{BKN identity} yields that the group $\Gr(E)$ also acts on the deformation space $\cD$ by
\[
\b \mapsto \kappa_1^{-1} \circ \b \circ \kappa_2, \quad (\kappa_1,\kappa_2) \in \Aut(\Gr(E)).
\]
In particular, if we set $\kappa:=(1,t) \in \Aut(\Gr(E))$ for $t \in \C^\ast$, then the induced connection is
\begin{equation} \label{connection Dt}
D_t:=D_E^\kappa=D_0+ta.
\end{equation}
This shows that the operator $D_t''$ can be extended across $t=0$ and gives a complex family of holomorphic structures which is isomorphic to $D_E''$ when $t \neq 0$ and $D_0''$ when $t=0$. Also we fix a real compact form $K \subset \Aut(\Gr(E))$ defined by
\[
K:=\Aut(V_1,h_1) \times \Aut(V_2,h_2)
\]
with the Lie algebra $\fk:=\Lie(K)$, where $\Aut(V_i,h_i)$ denotes the group of unitary automorphisms of $(V_i,h_i)$. Then we introduce the projection $\Pi_{\i \fk} \colon \i \Omega^0(\End(E,h_0)) \to \i \fk$ by
\[
s \mapsto \frac{1}{\rank(V_1)} \bigg( \int_X \Tr_{V_1}(s) \eta^2 \bigg) \Id_{V_1}+\frac{1}{\rank(V_2)} \bigg( \int_X \Tr_{V_2}(s) \eta^2 \bigg) \Id_{V_2}.
\]
Now we are ready to construct approximate solutions. In order to simplify notations in the approximate procedure, we introduce another variable
\[
\epsilon^2:=\e=\frac{1}{k},
\]
which will avoid the introduction of fractional powers of $\e$. Then we write the corresponding quantities as $\cP_\epsilon:=\cP_\e$, $\varphi_\epsilon(E):=\varphi_\e(E)$, etc.
\begin{dfn}
For a strict subbundle $S \subset E$, the {\it order of discrepancy} is the leading order of the expansion of $\varphi_\epsilon(E)-\varphi_\epsilon(S)$.
\end{dfn}

\begin{rk}
We recall the expansion formula for $\varphi_\epsilon(S)$
\[
\varphi_\epsilon(S)=\frac{c_1(L)^2}{2([\omega] \cdot c_1(L))} \epsilon^{-2}+\frac{\mu_\eta(S)}{2([\omega] \cdot c_1(L))^2}+O(\epsilon^2).
\]
Since the $\epsilon^{-2}$-term does not depend on $S$, we know that the order of discrepancy of $S \subset E$ is an even integer and at least zero.
\end{rk}

Let $2q$ be the order of discrepancy of $V_1 \subset E$. We note that $q \geq 1$ since $\mu_\eta(V_1)=\mu_\eta(E)$. Also, the see-saw property (Proposition \ref{see-saw property}) shows that $\varphi_\epsilon(V_2)-\varphi_\epsilon(E)>0$ whose leading terms is of order $\epsilon^{2q}$. First we apply the Mumford stable case to each component $V_i$ to know that there exists a diagonal Hermitian endomorphism $f_\epsilon=\Id_E+O(\epsilon)$ such that
\[
\cP_\epsilon(D_0^{f_\epsilon})=O(\epsilon^{2q}).
\]
Set $t:=\lambda \epsilon^q$, where the constant $\lambda>0$ is determined in the later argument. Let $D_t$ be the connection defined by \eqref{connection Dt}. If we write $f_\epsilon=(f_{1,\epsilon},f_{2,\epsilon})$ for some Hermitian endomorphisms $f_{i,\epsilon}=\Id_{V_i}+O(\epsilon)$ ($i=1,2$) of $(V_i,h_i)$, then a direct computation shows that
\begin{equation} \label{connection Dtfe}
D_t^{f_\epsilon}=D_0^{f_\epsilon}+ta_\epsilon,
\end{equation}
where
\[
a_\epsilon=f_{1,\epsilon}^{-1} \circ \g \circ f_{2,\epsilon}-f_{2,\epsilon} \circ \g^\ast \circ f_{1,\epsilon}^{-1}=a+O(\epsilon)
\]
is off-diagonal. So the curvature of $D_t^{f_\epsilon}$ is given by
\begin{equation} \label{curvature FDtf}
F_{D_t^{f_\epsilon}}=F_{D_0^{f_\epsilon}}+\frac{\i}{2 \pi} \lambda \epsilon^q D_0^{f_\epsilon} a_\epsilon+\frac{\i}{2 \pi} \lambda^2 \epsilon^{2q} a_\epsilon \wedge a_\epsilon,
\end{equation}
where $D_0^{f_\epsilon} a_\epsilon$ is off-diagonal while $F_{D_0^{f_\epsilon}}$ and $a_\epsilon \wedge a_\epsilon$ are diagonal. We consider the effect of replacing $D_0$ with $D_t$ in the above equation.

\begin{lem}[see also \cite{ST20}, Lemma 5.13] \label{approximate solutions to order 2q-1}
There exist off-diagonal Hermitian endomorphisms $s_{q+1}, \ldots s_{2q-1}$ (depending on $\lambda$) such that if we set
\begin{equation} \label{g 2q-1 epsilon}
g_{2q-1,\epsilon}:=\exp(s_{q+1} \epsilon^{q+1}) \circ \cdots \circ \exp(s_{2q-1} \epsilon^{2q-1}),
\end{equation}
then we have
\[
\cP_\epsilon(D_t^{f_\epsilon \circ g_{2q-1,\epsilon}})=O(\epsilon^{2q}).
\]
\end{lem}
\begin{proof}
By \eqref{curvature FDtf}, we observe that
\[
\cP_\epsilon(D_t^{f_\epsilon})=\cP_\epsilon(D_0^{f_\epsilon})+\epsilon^q \sum_{r=1}^{q-1} \s_r \epsilon^r+O(\epsilon^{2q})
\]
for some off-diagonal Hermitian endomorphisms $\s_1,\ldots,\s_{q-1}$, where the crucial point is that the $\epsilon^q$-term does not appear in the above expansion. Indeed, by Lemma \ref{P is a deformation of P0} and \eqref{curvature FDtf}, the potential $O(\epsilon^q)$-contribution only comes from the leading term of the expansion $\Lambda_\eta D_0^{f_\epsilon} a_\epsilon=\Lambda_\eta D_0 a+O(\epsilon)$, which is zero by Lemma \ref{properties for a and g}. From the linearization, we have
\[
\begin{aligned}
\cP_\epsilon(D_t^{f_\epsilon \circ \exp(s_{q+1} \epsilon^{q+1})})&=\cP_\epsilon(D_t^{f_\epsilon})+(\d|_{D_t^{f_\epsilon},s_{q+1}} \cP_\epsilon) \epsilon^{q+1}+O(\epsilon^{q+2}) \\
&=\epsilon^q \sum_{r=1}^{q-1} \epsilon^r \s_r+(\d|_{D_0,s_{q+1}} \cP_0) \epsilon^{q+1}+O(\epsilon^{q+2}).
\end{aligned}
\]
Thus if we take an off-diagonal Hermitian endomorphism $s_{q+1}$ as a solution to
\[
\d|_{D_0,s_{q+1}} \cP_0=-\s_1,
\]
then we have
\[
\cP_\epsilon(D_t^{f_\epsilon \circ \exp(s_{q+1} \epsilon^{q+1})})=\epsilon^q \sum_{r=2}^{q-1} \epsilon^r \tilde{\s}_r+O(\epsilon^{2q})
\]
for some Hermitian endomorphisms $\tilde{\s}_2,\ldots,\tilde{\s}_{q-1}$. From \eqref{change of the curvature}, \eqref{connection Dtfe} and the fact that $s_{q+1}$ is off-diagonal, one can easily check that $\tilde{\s}_2,\ldots,\tilde{\s}_{q-1}$ are off-diagonal. Therefore we can repeat this procedure to kill the all terms up to order $\epsilon^{2q-1}$.
\end{proof}
We set
\[
I_\dag:=\frac{1}{\rank(V_1)} \Id_{V_1}-\frac{1}{\rank(V_2)} \Id_{V_2},
\]
which is in the kernel of the linearized operator $\d|_{D_0} \cP_0$ but orthogonal to $\Id_E$. We recall that the induced connection of $D_t$ on $\End(E)$ is given by
\begin{equation} \label{the connection Dt acting on endomorphisms}
D_t=D_0+t[a,\cdot].
\end{equation}
Thus for any Hermitian endomorphism $s$, we have
\begin{equation} \label{the linearization of curvature of Dt}
\begin{aligned}
D_t' D_t''s-D_t'' D_t's&=D_0' D_0''s-D_0'' D_0's+t\big( D_0'' [\g^\ast,s]-[\g^\ast,D_0''s]+D_0'[\g,s]-[\g,D_0's] \big) \\
&+t^2 \big( [\g, [\g^\ast,s]]-[\g^\ast,[\g,s]] \big).
\end{aligned}
\end{equation}
By using this formula, we can show the following lemma in the same way as \cite[Proposition 5.14]{ST20}:

\begin{lem} \label{Laplacian and Idpm}
For any Hermitian endomorphism $s$, we have
\begin{equation} \label{expansion of Laplacian}
\Delta_{\eta,D_t}s=\Delta_{\eta,D_0}s-2\lambda \epsilon^q D_0^\ast [a,s]+O(\epsilon^{q+1}).
\end{equation}
Also, we have an estimate
\[
\Delta_{\eta,D_t}(I_\dag)=-2 \lambda^2 \epsilon^{2q} \i \bigg( \frac{1}{\rank(V_1)}+\frac{1}{\rank(V_2)} \bigg) \Lambda_\eta(a \wedge a)+O(\epsilon^{2q+1})
\]
and its projection to $\i \fk$
\begin{equation} \label{expansion of pr Laplacian I}
\Pi_{\i \fk} \Delta_{\eta,D_t}(I_\dag)=-\varrho \lambda^2 \epsilon^{2q} I_\dag+O(\epsilon^{2q+1}),
\end{equation}
where $\varrho:=-2 \i \big( \frac{1}{\rank(V_1)}+\frac{1}{\rank(V_2)} \big) \int_X \Tr_{V_1} (\Lambda_\eta \g \wedge \g^\ast)>0$. The expansion \eqref{expansion of pr Laplacian I} also holds for a perturbed connection $D_t^{g_\epsilon}$ for any $g_\epsilon \in \cG=\Omega^0(\Uni(E,h_0))$ such that $g_\epsilon=\Id_E+O(\epsilon)$ and $g_\epsilon$ is diagonal up to an error of order $\epsilon^{q+1}$. Moreover, the expansion \eqref{expansion of Laplacian} still holds for such a connection $D_t^{g_\epsilon}$, except that one has to replace $\Delta_{\eta,D_0}$ with $\Delta_{\eta,D_0^{g_\epsilon}}$.
\end{lem}

We briefly explain the last assertion in the above lemma. For diagonal $g_\epsilon$, we compute $D_t^{g_\epsilon}=D_0^{g_\epsilon}+t[\tilde{a}_\epsilon,\cdot]$ where $\tilde{a}_\epsilon=a+O(\epsilon)$ is off-diagonal. By using this, one can prove \eqref{expansion of Laplacian} just as in the case $g_\epsilon=\Id_E$, except that one has to replace $\Delta_{\eta,D_0}$ with $\Delta_{\eta,D_0^{g_\epsilon}}$. For general $g_\epsilon$, the contribution of the off-diagonal component of $g_\epsilon$ occurs at least order $\epsilon^{q+1}$ from the assumption, so the same formula still holds. As for $I_\dag$, we note that $D I_\dag=0$ for any product connection $D$. Also the connection $D_t$ is diagonal up to order $\epsilon^q$ by \eqref{the connection Dt acting on endomorphisms}. Keeping these facts in mind and using the formula \eqref{action of the gauge group}, one can observe that the diagonal contribution of $g_\epsilon$ to $\Delta_{\eta,D_t^{g_\epsilon}}(I_\dag)-\Delta_{\eta,D_t}(I_\dag)$ occurs at least order $\epsilon^{2q+1}$. Thus the formula \eqref{expansion of pr Laplacian I} holds for $D_t^{g_\epsilon}$.

Now we will eliminate the $\epsilon^{2q}$-term.

\begin{lem} \label{approximate solutions to order 2q}
There exist a constant $\lambda>0$ and a Hermitian endomorphism $s_{2q}$ such that if we set $t=\lambda \epsilon^q$ and $g_{2q,\epsilon}:=g_{2q-1,\epsilon} \circ \exp(s_{2q} \epsilon^{2q})$, then we have
\[
\cP_\epsilon(D_t^{f_\epsilon \circ g_{2q,\epsilon}})=O(\epsilon^{2q+1}),
\]
where the endomorphism $g_{2q-1,\epsilon}$ is constructed in Lemma \ref{approximate solutions to order 2q-1}.
\end{lem}
\begin{proof}
We decompose the $\epsilon^{2q}$-term of $\cP_\epsilon(D_t^{f_\epsilon \circ g_{2q-1,\epsilon}})$ as
\[
\cP_\epsilon(D_t^{f_\epsilon \circ g_{2q-1,\epsilon}})=(\s_{\i \fk}+\s_\perp) \epsilon^{2q}+O(\epsilon^{2q+1}),
\]
where $\s_{\i \fk} \in \i \fk$ and $\s_\perp$ is orthogonal to $\i \fk$. Let $[\cdot]_{2q}$ denotes the coefficient of $\epsilon^{2q}$ in the expansion. First, by \eqref{g 2q-1 epsilon} the off-diagonal (resp. diagonal) component of $g_{2q-1,\epsilon}-\Id_E$ is $O(\epsilon^{q+1})$ (resp. $O(\epsilon^{2q+2})$). Thus substituting $g=g_{2q-1,\epsilon}$, $D=D_t^{f_\epsilon \circ g_{2q-1,\epsilon}}$ and \eqref{connection Dtfe} in the formula \eqref{change of the curvature}, we find that the off-diagonal (resp. diagonal) component of $F_{D_t^{f_\epsilon \circ g_{2q-1}}}-F_{D_t^{f_\epsilon}}$ is $O(\epsilon^{q+1})$ (resp. $O(\epsilon^{2q+1})$). By using this fact, \eqref{curvature FDtf} and setting $D=D_t^{f_\epsilon \circ g_{2q-1,\epsilon}}$ in \eqref{expansion of the linearized operator}, we observe that the diagonal contribution of $g_{2q-1,\epsilon}$ to $\cP_\epsilon(D_t^{f_\epsilon \circ g_{2q-1,\epsilon}})-\cP_\epsilon(D_t^{f_\epsilon})$ occurs at least order $\epsilon^{2q+1}$. While the off-diagonal contribution appears in lower order terms, which does not affect the projection onto $\i \fk$. Thus we have
\[
\s_{\i \fk}=\Pi_{\i \fk} \big[ \cP_\epsilon(D_t^{f_\epsilon \circ g_{2q-1,\epsilon}}) \big]_{2q}=\Pi_{\i \fk} \big[ \cP_\epsilon(D_t^{f_\epsilon}) \big]_{2q}.
\]
Next, by substituting $D=D_t^{f_\epsilon}$ and \eqref{curvature FDtf} in \eqref{expansion of the linearized operator}, we know that the diagonal contribution of $t=\lambda \epsilon^q$ to $\cP_\epsilon(D_t^{f_\epsilon})-\cP_\epsilon(D_0^{f_\epsilon})$ occurs at least order $\epsilon^{2q}$, which comes from the trace of the third term in the right hand side of \eqref{curvature FDtf}, \ie
\[
\Pi_{\i \fk} \big[ \cP_\epsilon(D_t^{f_\epsilon}) \big]_{2q}=\Pi_{\i \fk} \big[ \cP_\epsilon(D_0^{f_\epsilon}) \big]_{2q}+\frac{\i}{4 \pi c_1(L)^2} \lambda^2 \Pi_{\i \fk} \Lambda_\eta(a \wedge a).
\]
Set
\[
\Pi_{\i \fk} \big[ \cP_\epsilon(D_0^{f_\epsilon}) \big]_{2q}=\varrho_1 \frac{\Id_{V_1}}{\rank(V_1)}+\varrho_2 \frac{\Id_{V_2}}{\rank(V_2)}, \quad \varrho_1, \varrho_2 \in \R.
\]
Then by using the fact that $f_\epsilon$ is diagonal, the coefficient $\varrho_i$ $(i=1,2$) is given by
\[
\begin{aligned}
\varrho_i&=\int_X \Tr_{V_i} \big[ c_\epsilon (F_{D_{V_i}}+\epsilon^{-2} \chi \Id_{V_i})^2-\omega \Id_{V_i} \wedge (F_{D_{V_i}}+\epsilon^{-2} \chi \Id_{V_i}) \big]_{2q} \\
&=\big[ 2c_\epsilon \ch_2(V_i \otimes L^k)-[\omega] \cdot \ch_1(V_i \otimes L^k) \big]_{2q},
\end{aligned}
\]
where we regard $[\omega] \cdot \ch_1(V_i \otimes L^k)$ and $\ch_2(V_i \otimes L^k)$ as functions of $\epsilon$ via the expansion. Also, we compute
\[
\frac{\i}{4 \pi c_1(L)^2} \lambda^2 \Pi_{\i \fk} \Lambda_\eta(a \wedge a)=C \lambda^2 I_\dag
\]
with $C:=-\frac{\i}{4 \pi c_1(L)^2} \int_X \Tr_{V_1} \Lambda_\eta(\g \wedge \g^\ast)=\frac{\i}{4 \pi c_1(L)^2} \int_X \Tr_{V_2} \Lambda_\eta(\g^\ast \wedge \g)>0$. Thus we have $\s_{\i \fk}=0$ if and only if
\[
\begin{cases}
\varrho_1+C\lambda^2=0 \\
\varrho_2-C\lambda^2=0.
\end{cases}
\]
We can find such a $\lambda>0$ if and only if
\[
\varrho_1+\varrho_2=0, \quad \varrho_2>0.
\]
Indeed, the first assertion follows from the additive property of Chern characters and the definition of $c_\epsilon$. As for the second assertion, we observe that
\[
2 c_\epsilon \ch_2 (V_2 \otimes L^k)-[\omega] \cdot \ch_1(V_2 \otimes L^k)=\frac{[\omega] \cdot \ch_1(V_2 \otimes L^k)}{\varphi_\epsilon(E)}(\varphi_\epsilon(V_2)-\varphi_\epsilon(E)).
\]
On the other hand, since the order of discrepancy of $V_1 \subset E$ is $2q$, the above formula together with the see-saw property implies that
\[
\varrho_2=\frac{2 \rank(V_2)([\omega] \cdot c_1(L))^2}{c_1(L)^2}[\varphi_\epsilon(V_2)-\varphi_\epsilon(E)]_{2q}>0.
\]
Hence there exists $\lambda>0$ such that
\[
\cP_\epsilon(D_t^{f_\epsilon \circ g_{2q-1,\epsilon}})=\s_\perp \epsilon^{2q}+O(\epsilon^{2q+1}).
\]
If we define a Hermitian endomorphism $s_{2q}$ as a solution to
\[
\d|_{D_0, s_{2q}} \cP_0=-\s_\perp,
\]
then we can see that
\[
\cP_\epsilon(D_t^{f_\epsilon \circ g_{2q,\epsilon}})=O(\epsilon^{2q+1})
\]
from the linearization.
\end{proof}

In what follows, we fix the constant $\lambda$ determined in Lemma \ref{approximate solutions to order 2q}. As for approximate solutions to higher order, we can deal with any error orthogonal to $\Id_E$ via linearization by using Lemma \ref{Laplacian and Idpm}.

\begin{lem} \label{approximate solutions to higher order}
Let $\lambda>0$, $g_{2q,\epsilon}$ be as in Lemma \ref{approximate solutions to order 2q}. For any integer $N \geq 2q$, there exist real constants $\lambda_{2q+1},\ldots,\lambda_N$ and Hermitian endomorphisms $s_{m,m-q},\ldots,s_{m,m}$ ($m=2q+1,\ldots,N$) such that if we set
\[
\phi_{m,\epsilon}:=\exp(\lambda_m I_\dag \epsilon^{m-2q}) \circ \exp(s_{m,m-q} \epsilon^{m-q}) \circ \cdots \circ \exp(s_{m,m} \epsilon^m), \quad m=2q+1,\ldots,N
\]
and
\[
g_{N,\epsilon}:=g_{2q,\epsilon} \circ \phi_{2q+1,\epsilon} \cdots \circ \phi_{N,\epsilon},
\]
then we have
\[
\cP_\epsilon(D_t^{f_\epsilon \circ g_{N,\epsilon}})=O(\epsilon^{N+1}).
\]
\end{lem}
\begin{proof}
We can prove the lemma almost in the same way as \cite[Proposition 5.16]{ST20}, but we show a sketch of the proof for reader's convenience. By Lemma \ref{approximate solutions to order 2q}, we decompose
\[
\cP_\epsilon(D_t^{f_\epsilon \circ g_{2q,\epsilon}})=(\kappa I_\dag+\s_\perp) \epsilon^{2q+1}+O(\epsilon^{2q+2}),
\]
where $\kappa \in \R$ and $\s_\perp$ is orthogonal to $\i \fk$. Since $D_t^{f_\epsilon \circ g_{2q,\epsilon}}$ is diagonal up to an error of order $\epsilon^q$, the formula \eqref{action of the gauge group} shows that for $\lambda_{2q+1} \in \R$ the off-diagonal contribution of $\exp(\lambda_{2q+1} I_\dag \epsilon)$ to $F_{D_t^{f_\epsilon \circ g_{2q,\epsilon} \circ \exp(\lambda_{2q+1} I_\dag \epsilon)}}-F_{D_t^{f_\epsilon \circ g_{2q,\epsilon}}}$ occurs at least order $\epsilon \cdot \epsilon^q=\epsilon^{q+1}$, whereas the diagonal contribution occurs at least order $\epsilon \cdot \epsilon^q \cdot \epsilon^q=\epsilon^{2q+1}$, which comes from the linearization $\d|_{D_t^{f_\epsilon \circ g_{2q,\epsilon}}, \lambda_{2q+1} I_\dag \epsilon} \cP_0$. So by Lemma \ref{Laplacian and Idpm}, we can find $\lambda_{2q+1} \in \R$ (depending on $\kappa$) such that
\[
\cP_\epsilon(D_t^{f_\epsilon \circ g_{2q,\epsilon} \circ \exp(\lambda_{2q+1} I_\dag \epsilon)})=\epsilon^{q+1} \sum_{r=0}^{q-1} \s_r \epsilon^r+\tilde{\s}_\perp \epsilon^{2q+1}+O(\epsilon^{2q+2}),
\]
where $\s_r$'s are off-diagonal and $\tilde{\s}_\perp$ is orthogonal to $\i \fk$. In order to simplify the notation, for $m=0,1,\ldots,q+1$ we set endomorphisms $\psi_{m,\epsilon}$ by $\psi_{0,\epsilon}:=g_{2q,\epsilon} \circ \exp(\lambda_{2q+1} I_\dag \epsilon)$ and
\[
\psi_{m,\epsilon}:=\psi_{0,\epsilon} \circ \exp(s_{2q+1,q+1} \epsilon^{q+1}) \circ \cdots \circ \exp(s_{2q+1,q+m} \epsilon^{q+m}), \quad m=1,\ldots,q+1,
\]
where Hermitian endomorphisms $s_{2q+1,q+1},\ldots,s_{2q+1,2q+1}$ will be chosen in later arguments. If we perturb $D_t^{f_\epsilon \circ \psi_{0,\epsilon}}$ by $\exp(s_{2q+1,q+1} \epsilon^{q+1})$ for some off-diagonal $s_{2q+1,q+1}$, this changes $\s_r$ (resp. $\tilde{\s}_\perp$) while keeping it off-diagonal (resp. orthogonal to $\i \fk$). On the other hand, the first potentially diagonal change happens at the order $\epsilon^{2q+1}$. By Lemma \ref{Laplacian and Idpm}, we can compute this change as $-2 \lambda D_0^\ast [a,s_{2q+1,q+1}] \epsilon^{2q+1}$, which is contained in the image of $\Delta_{\eta,D_0}$. Thus, for a suitable choice of off-diagonal $s_{2q+1,q+1}$ we have
\[
\cP_\epsilon(D_t^{f_\epsilon \circ \psi_{1,\epsilon}})=\epsilon^{q+1} \sum_{r=1}^{q-1} \s_r \epsilon^r+\tilde{\s}_\perp \epsilon^{2q+1}+O(\epsilon^{2q+2})
\]
after replacing $\s_r$ and $\tilde{\s}_\perp$. Inductively, we can find off-diagonal Hermitian endomorphisms $s_{2q+1,q+1},\ldots,s_{2q+1,2q}$ such that
\[
\cP_\epsilon(D_t^{f_\epsilon \circ \psi_{q,\epsilon}})=\tilde{\s}_\perp \epsilon^{2q+1}+O(\epsilon^{2q+2}).
\]
Then we can remove $O(\epsilon^{2q+1})$-term by using the fact that $\tilde{\s}_\perp$ is orthogonal to $\i \fk$, \ie there exists a Hermitian endomorphism $s_{2q+1,2q+1}$ such that
\[
\cP_\epsilon(D_t^{f_\epsilon \circ \psi_{q+1,\epsilon}})=O(\epsilon^{2q+2}).
\]
Repeating this process, we can remove all the error terms up to order $\epsilon^N$.
\end{proof}

\begin{lem} \label{estimate for the linearization}
There exists a constant $C>0$ such that for all $0<\epsilon \ll 1$ and Hermitian endomorphisms $s$ orthogonal to $\Id_E$, we have
\[
\|\d|_{D_t,s}\cP_\epsilon\|_{L^2} \geq C \epsilon^{2q}\|s\|_{L^2},
\]
where $\| \cdot \|_{L^2}$ denotes the $L^2$-norm with respect to $\eta$ and $h_0$.
\end{lem}
\begin{proof}
It suffices to show that
\[
-(\d|_{D_t,s}\cP_\epsilon,s)_{L^2} \geq C \epsilon^{2q} \|s\|_{L^2}^2
\]
for all Hermitian endomorphisms $s$ orthogonal to $\Id_E$, where $(\cdot,\cdot)_{L^2}$ denotes the $L^2$-inner product with respect to $\eta$ and $h_0$. In what follows, we use the same notation $C$ to express the uniform positive constant, but it changes from line to line. We decompose $s$ as $s=\nu I_\dag+s_\perp$, where $\nu \in \R$ and $s_\perp$ is orthogonal to $\i \fk$. Since $\d|_{D_t}\cP_\epsilon$ is self-adjoint, we compute
\[
-(\d|_{D_t,s}\cP_\epsilon,s)_{L^2}=-\nu^2 (\d|_{D_t,I_\dag}\cP_\epsilon,I_\dag)_{L^2}-2\nu(\d|_{D_t,I_\dag}\cP_\epsilon,s_\perp)_{L^2}-(\d|_{D_t,s_\perp}\cP_\epsilon,s_\perp)_{L^2}.
\]
By the definition, the linearized operator $\d|_{D_t} \cP_\epsilon$ is wedge products of four types of terms:
\begin{enumerate}
\renewcommand{\labelenumi}{(\roman{enumi})}
\item $D_t' D_t''-D_t'' D_t'$
\item $F_{D_t}$
\item $\omega \Id_E$
\item $\chi \Id_E$
\end{enumerate}
and each term contains exactly one $D_t' D_t''-D_t'' D_t'$. The curvature $F_{D_t}$ satisfies $F_{D_t}=F_{D_0}+\frac{\i}{2 \pi} t D_0 a+\frac{\i}{2 \pi} t^2 a \wedge a$. First we consider the term $-\nu^2(\d|_{D_t,I_\dag}\cP_\epsilon,I_\dag)_{L^2}$. Applying the formula \eqref{the linearization of curvature of Dt} to $I_\dag$, we have
\begin{equation} \label{D'D''-D''D' for Idpm}
D_t' D_t''I_\dag-D_t'' D_t'I_\dag=t\big( D_0'' [\g^\ast,I_\dag]+D_0'[\g,I_\dag] \big)+t^2 \big( [\g, [\g^\ast,I_\dag]]-[\g^\ast,[\g,I_\dag]] \big),
\end{equation}
where the first term in the right hand side is off-diagonal and $O(\epsilon^q)$, while the second term is diagonal and $O(\epsilon^{2q})$. Thus combining with Lemma \ref{P is a deformation of P0} and Lemma \ref{Laplacian and Idpm}, we obtain
\[
\Pi_{\i \fk} (\d|_{D_t,I_\dag} \cP_\epsilon)=-\rho \epsilon^{2q} I_\dag+O(\epsilon^{2q+1})
\]
for some constant $\rho>0$. Thus we have
\[
-\nu^2(\d|_{D_t,I_\dag}\cP_\epsilon,I_\dag)_{L^2} \geq C \epsilon^{2q} \|\nu I_\dag\|^2.
\]
Next, to deal with the term $-(\d|_{D_t,s_\perp}\cP_\epsilon,s_\perp)_{L^2}$, we first claim that
\begin{equation} \label{Poincare inequality}
-(\d|_{D_0,s_\perp}\cP_0,s_\perp)_{L^2} \geq C \|s_\perp\|_{L^2}^2
\end{equation}
for any Hermitian endomorphism $s_\perp$ orthogonal to $\i \fk$. Indeed, the operator $\d|_{D_0}\cP_0$ coincides with the Laplacian $\Delta_{\eta,D_0}$ up to scale, whose kernel just coincides with $\i \fk$. Thus the estimate \eqref{Poincare inequality} follows from the Poincar\'e inequality. The similar estimate
\[
-(\d|_{D_t,s_\perp}\cP_\epsilon,s_\perp)_{L^2} \geq C \|s_\perp\|_{L^2}^2
\]
also holds for all Hermitian endomorphisms $s_\perp$ orthogonal to $\i \fk$ since $\d|_{D_t}\cP_\epsilon$ is an $O(\epsilon)$-perturbation of $\d|_{D_0}\cP_0$. Finally, let us consider the second term $-2\nu(\d|_{D_t,I_\dag}\cP_\epsilon,s_\perp)_{L^2}$. By using Lemma \ref{P is a deformation of P0}, Lemma \ref{Laplacian and Idpm} and \eqref{D'D''-D''D' for Idpm} again, we observe that $(\d|_{D_t,I_\dag}\cP_\epsilon,s_\perp)_{L^2}=O(\epsilon^{q+2})$. However, by the elementary inequality $xy \leq \frac{1}{2} x^2+\frac{1}{2} y^2$, we know that
\[
-2\nu(\d|_{D_t,I_\dag}\cP_\epsilon,s_\perp)_{L^2} \geq -C \epsilon^{2q+2} \|\nu I_\dag\|_{L^2}^2-C \epsilon^2 \|s_\perp\|_{L^2}^2.
\]
So the term $-2\nu(\d|_{D_t,I_\dag}\cP_\epsilon,s_\perp)_{L^2}$ can be absorbed in other good terms for sufficiently small $\epsilon$. Hence we obtain the desired estimate.
\end{proof}
Since the off-diagonal component of $g_{N,\epsilon}$ is $O(\epsilon^{q+1})$ from the construction, one can easily check that Lemma \ref{estimate for the linearization} still holds for the connection $D_t^{f_\epsilon \circ g_{N,\epsilon}}$. Also we can show that the operator $\d|_{D_t^{f_\epsilon \circ g_{N,\epsilon}}} \cP_\epsilon$ is elliptic with trivial kernel by using Lemma \ref{ellipticity} and the argument in the last part of the proof of Theorem \ref{existence theorem in Mumford stable case}. In particular, for a sufficiently large integer $r$, we have an estimate for the inverse operator $\cQ_{N,\epsilon} \colon L_{r,0}^2 \to L_{r+2,0}^2$ of $\d|_{D_t^{f_\epsilon \circ g_{N,\epsilon}}} \cP_\epsilon$
\[
\|\cQ_{N,\epsilon}(s)\|_{L^2} \leq C \epsilon^{-2q} \|s\|_{L^2}
\]
for all Hermitian endomorphisms $s$ orthogonal to $\Id_E$, where $L_{r,0}^2$ denotes the Sobolev space consisting of all $L_r^2$-sections of $E$ having zero mean value. Combining this with the standard elliptic estimate as in \cite[Proposition 4.12]{ST20}, we obtain an estimate
\begin{equation} \label{estimate for the inverse operator}
\|\cQ_{N,\epsilon}\|_{\op} \leq C \epsilon^{-2q}
\end{equation}
as an operator from $L_{r,0}^2$ to $L_{r+2,0}^2$. On the other hand, the Lipschitz constant of the non-linear term $\cM_{N,\epsilon}(s):=\cP_\epsilon(D_t^{f_\epsilon \circ g_{N,\epsilon} \circ \exp(s)})-\d|_{D_t^{f_\epsilon \circ g_{N,\epsilon}},s} \cP_\epsilon$ can be controlled by the following consequence from the mean value theorem that can be achieved exactly by the same argument as in \cite[Lemma 2.10]{Fin04}:
\begin{lem} \label{mean value theorem}
There exist constants $\d, C>0$ such that for sufficiently small $\epsilon>0$, if $\|s\|_{L_{r+2}^2}, \|s'\|_{L_{r+2}^2} \leq \d$, then we have
\[
\|\cM_{N,\epsilon}(s)-\cM_{N,\epsilon}(s')\|_{L_r^2} \leq C \big( \|s\|_{L_{r+2}^2}+\|s'\|_{L_{r+2}^2} \big) \|s-s'\|_{L_{r+2}^2}.
\]
\end{lem}
We also use the quantitative version of the implicit function theorem as follows:
\begin{thm}[\cite{Fin04}, Theorem 4.1] \label{quantitative IFT}
Let $P \colon V \to W$ be a differentiable map on Banach spaces $V$, $W$ whose linearization at the origin $\d|_0 P$ is invertible with inverse $Q$. Assume that $P-\d|_0 P$ is Lipschitz of constant $\frac{1}{2 \|Q\|_{\op}}$ on the closed ball of radius $\rho>0$ centered at $0$ in $V$. Then for all $w \in W$ with $\|w-P(0)\|<\frac{\rho}{2 \|Q\|_{\op}}$, there exists $v \in V$ such that $P(v)=w$.
\end{thm}

Now we are ready to prove Theorem \ref{existence of a solution on twisted bundles repeat} (3) in the two components case.

\begin{proof}[Completion of the proof of Theorem \ref{existence of a solution on twisted bundles repeat} (3)]
By Lemma \ref{mean value theorem}, the operator $\cM_{N,\epsilon}$ is Lipschitz with constant $2 \rho C$ on the closed ball of radius $\rho$ centered at $0$. Thus the condition $2 \rho C=\frac{1}{2 \|\cQ_{N,\epsilon} \|_{\op}}$ together with \eqref{estimate for the inverse operator} imposes that
\[
\frac{\rho}{2 \|\cQ_{N,\epsilon}\|_{\op}}=\frac{1}{8C \|\cQ_{N,\epsilon}\|_{\op}^2} \geq C' \epsilon^{4q}
\]
for some constant $C'>0$. So if we take $N=4q$, then we have $\|\cP_\epsilon(D_t^{f_\epsilon \circ g_{4q,\epsilon}}) \|_{L_r^2} \leq \frac{\rho}{2 \|\cQ_{4q,\epsilon}\|_{\op}}$, and hence we can apply Theorem \ref{quantitative IFT} to conclude that there exists a connection $D_\epsilon$ satisfying
\[
\cP_\epsilon(D_\epsilon)=0
\]
for sufficiently small $\epsilon>0$. Just as in the Mumford stable case, one can show that $D_\epsilon \otimes D_L^k$ with $\epsilon^2=1/k$ is a $J$-positive smooth solution. This completes the proof.
\end{proof}

\begin{exam}
Let us consider a simple rank-$3$ vector bundle over $\C\P^2$ studied by Maruyama in the context of the Gieseker stability and Dervan--Maccarthy--Sektnan in the context of the asymptotic $Z$-stability \cite[Example 2.18]{DMS20}. We follow its presentation in \cite[page 177]{OSS11}. Set $L:=\cO_{\C\P^2}(1)$ and take any K\"ahler form $\omega \in c_1(\cO_{\C\P^2}(1))$. Then we have $\eta=\chi=\omega$ in \eqref{J-equation on surfaces repeat}. Let $S$ be a rank-$2$ Mumford stable bundle (with respect to $\omega$) satisfying $c_1(S)=0$ and $H^1(\C\P^2,S) \neq 0$. In particular, we have $c_2(S) \geq 2$ by a refined version of the Bogomolov--Gieseker inequality for rank-$2$ Mumford stable bundles over $\C\P^2$ (see the proof of \cite[Lemma 1.2.7]{OSS11}). Then any non-zero element in $H^1(\C\P^2,S)$ defines a non-trivial extension
\[
0 \to S \to E \to \cO_{\C\P^2} \to 0,
\]
and Chern classes of $E$ are given by
\[
c_1(E)=0, \quad c_2(E)=c_2(S).
\]
So the bundle $E$ is not Mumford stable since $\mu_{\omega}(S)=\mu_{\omega}(E)$, but $E$ is in fact Gieseker stable. It follows that $E$ is simple and Mumford semistable with
\[
0 \subset S \subset E
\]
giving a Jordan--H\"older filtration, and the associated graded object $\Gr(E)$ is given by
\[
\Gr(E)=S \oplus \cO_{\C\P^2}.
\]
We compute
\[
[\omega] \cdot \ch_1(S \otimes L^k)=2k, \quad \ch_2(S \otimes L^k)=k^2-c_2(S),
\]
\[
[\omega] \cdot \ch_1(E \otimes L^k)=3k, \quad \ch_2(E \otimes L^k)=\frac{3}{2} k^2-c_2(E),
\]
which shows that
\[
\varphi_k(E)-\varphi_k(S)=\frac{1}{6k} c_2(S)>0.
\]
Hence from the proof of Theorem \ref{existence of a solution on twisted bundles repeat}, we find that there exists a solution to the $J$-equation with respect to $\omega$ on $E \otimes L^k$ for sufficiently large $k$.
\end{exam}

\subsection{Vortex bundles} \label{Vortex bundles}
Let $\Sigma$ be a Riemann surface, $L$ a holomorphic line bundle of degree $1$ over $\Sigma$ and $h_\Sigma$ a Hermitian metric on $L$ with curvature $\omega_\Sigma>0$. Set $X:=\Sigma \times \C\P^1$ and $\omega:=s \pi_1^\ast \omega_\Sigma+\pi_2^\ast \omega_{\FS}$ ($s \in \R_{>0}$), where $\pi_1 \colon X \to \Sigma$, $\pi_2 \colon X \to \C\P^1$ are projections. The vortex bundle $E$ is a rank-$2$ vector bundle over $X$ defined by
\[
E:=\pi_1^\ast((r_1+1)L) \otimes \pi_2^\ast (r_2 \cO(2)) \oplus \pi_1^\ast(r_1 L) \otimes \pi_2^\ast ((r_2+1) \cO(2)),
\]                                        
where $r_1, r_2$ are positive integers and the holomorphic structure on $E$ is chosen so that the first component $\pi_1^\ast((r_1+1)L) \otimes \pi_2^\ast (r_2 \cO(2))$ is a holomorphic subbundle of $E$.

First of all, we will check that the bundle $E$ satisfies the condition \eqref{positivity for Chern characters} for all $r_1, r_2, s$. Indeed, we have
\begin{equation} \label{ch1 E}
\begin{aligned}
\ch_1(E)&=\ch_1 \big( \pi_1^\ast((r_1+1)L) \otimes \pi_2^\ast(r_2 \cO(2)) \big)+\ch_1\big( \pi_1^\ast(r_1 L) \otimes \pi_2^\ast((r_2+1) \cO(2)) \big) \\
&=(2r_1+1) c_1(L)+2(2r_2+1) c_1(\cO(1)),
\end{aligned}
\end{equation}
\begin{equation} \label{omega ch1 E}
\begin{aligned}[]
[\omega] \cdot \ch_1(E)&=\big( s c_1(L)+c_1(\cO(1)) \big) \cdot \big((2r_1+1) c_1(L)+2(2r_2+1) c_1(\cO(1)) \big) \\
&=2r_1+1+2s(2r_2+1)>0,
\end{aligned}
\end{equation}
\begin{equation} \label{ch2 E}
\begin{aligned}
\ch_2(E)&=\ch_2 \big( \pi_1^\ast((r_1+1)L) \otimes \pi_2^\ast (r_2 \cO(2)) \big)+\ch_2 \big( \pi_1^\ast(r_1 L) \otimes \pi_2^\ast ((r_2+1) \cO(2)) \big) \\
&=\ch_1 \big( \pi_1^\ast((r_1+1)L) \big) \cdot \ch_1 \big( \pi_2^\ast(r_2 \cO(2)) \big)+\ch_1 \big(\pi_1^\ast(r_1L) \big) \cdot \ch_1 \big( \pi_2^\ast((r_2+1) \cO(2)) \big) \\
&=2(r_1+1)r_2+2r_1(r_2+1)>0.
\end{aligned}
\end{equation}
In particular, we have
\begin{equation} \label{formula for c}
c=\frac{[\omega] \cdot \ch_1(E)}{2 \ch_2(E)}=\frac{2r_1+1+2s(2r_2+1)}{4(r_1+1)r_2+4r_1(r_2+1)}.
\end{equation}

\subsubsection{A necessary condition}
Now we will study how to choose parameters $r_1,r_2,s$ to obtain a solution of \eqref{J-equation} with respect to $\omega$.
\begin{lem}
If the bundle $E$ admits a $J$-Griffiths positive Hermitian metric satsfying \eqref{J-equation} with a non-trivial second fundamental form of $\pi_1^\ast((r_1+1)L) \otimes \pi_2^\ast(r_2(\cO(2)))$, then we have
\[
s<\frac{r_1(r_1+1)}{2r_2(r_2+1)}.
\]
\end{lem}
\begin{proof}
We will invoke Theorem \ref{J-stability for rank-2 bundles over surfaces} for the subbundle $\pi_1^\ast((r_1+1)L) \otimes \pi_2^\ast(r_2 \cO(2))$ to know that
\[
\ch_2 \big( \pi_1^\ast((r_1+1)L) \otimes \pi_2^\ast(r_2 \cO(2)) \big) \cdot \big( [\omega] \cdot \ch_1(E) \big)<\bigg( [\omega] \cdot \ch_1 \big( \pi_1^\ast((r_1+1)L) \otimes \pi_2^\ast (r_2 \cO(2)) \big) \bigg) \cdot \ch_2(E).
\]
We compute
\[
\begin{aligned}[]
[\omega] \cdot \ch_1 \big( \pi_1^\ast(r_1+1)L \otimes \pi_2^\ast(r_2 \cO(2) \big)&=\big( s c_1(L)+c_1(\cO(1)) \big) \cdot \big( (r_1+1)c_1(L)+2r_2 c_1(\cO(1)) \big) \\
&=r_1+1+2r_2 s,
\end{aligned}
\]
\[
\begin{aligned}
\ch_2 \big( \pi_1^\ast((r_1+1)L) \otimes \pi_2^\ast(r_2 \cO(2) \big)&=\ch_1 \big(\pi_1^\ast((r_1+1)L) \big) \cdot \ch_1 \big( \pi_2^\ast (r_2\cO(2)) \big)  \\
&=2(r_1+1)r_2.
\end{aligned}
\]
Combining with \eqref{omega ch1 E} and \eqref{ch2 E}, we have
\[
\begin{aligned}
& (r_1+1+2r_2s)\big( 2(r_1+1)r_2+2r_1(r_2+1) \big)-2(r_1+1)r_2 \big( 2r_1+1+2s(2r_2+1) \big) \\
&=2(r_1^2+r_1-2r_2^2s-2r_2s)>0,
\end{aligned}
\]
which is equivalent to
\[
s<\frac{r_1(r_1+1)}{2r_2(r_2+1)}.
\]
\end{proof}
\begin{rk} \label{the lower bound may not be sharp}
As another criterion, we recall that the existence of a $J$-Griffiths positive Hermitian metric on $E$ implies the positivity of the cohomology class
\begin{equation} \label{c c1 E omega}
c \cdot c_1(E)-[\omega]>0.
\end{equation}
By using \eqref{ch1 E} and \eqref{formula for c}, we can compute
\[
\begin{aligned}
& \big( 4(r_1+1)r_2+4r_1(r_2+1) \big) \big( c \cdot c_1(E)-[\omega] \big) \\
&=(4r_1^2+4r_1+2s+1)c_1(L)+2(8r_2^2s+8r_2s+2s+1)c_1(\cO(1))>0.
\end{aligned}
\]
Thus the condition \eqref{c c1 E omega} is satisfied automatically for all $r_1,r_2,s$.
\end{rk}

\subsubsection{Dimensional reduction; the $J$-vortex equation}
Since the group $\SU(2)$ acts on the bundle $\cO(2) (\simeq T' \C\P^1)$, it is natural to consider $\SU(2)$-invariant solutions to \eqref{J-equation}. We observe that
\[
\begin{aligned}
&H^1\big( \Sigma \times \C\P^1, \Hom \big( \pi_1^\ast(r_1L) \otimes \pi_2^\ast((r_2+1)\cO(2)), \pi_1^\ast((r_1+1)L) \otimes \pi_2^\ast(r_2\cO(2)) \big) \\
&=H^1(\Sigma \times \C\P^1, \pi_1^\ast L \otimes \pi_2^\ast \cO(-2)) \\
&\simeq H^0(\Sigma,L) \otimes H^1(\C\P^1,\cO(-2)) \\
& \simeq H^0(\Sigma,L),
\end{aligned}
\]
where we used the facts that $H^0(\C\P^1, \cO(-2))=0$ and $H^1(\C\P^1,\cO(-2)) \simeq H^0(\C\P^1,\cO)^\ast \simeq \C$ by Serre duality. Indeed, any section $\phi \in H^0(\Sigma,L)$ determines an extension whose $\SU(2)$-invariant representative is given by $\pi_1^\ast \phi \wedge \pi_2^\ast \zeta$ with
\[
\zeta=\frac{\s}{(1+|z|^2)^2} \otimes d\bar{z}, \quad \s:=dz,
\]
where we express the section of $\cO(-2)$ as $\s$ to distinguish from the $1$-form $dz$\footnote{Indeed, one can easily check that $\zeta$ is smooth at $z=\infty$ by considering the change of coordinates $w=1/z$.}. Let $h$ be a Hermitian metric on $L$ and $f$ a smooth positive function on $\Sigma$ (which can be regarded as a Hermitian metric on the trivial $\C$-bundle on $\Sigma$). Set
\[
H:=h_ 1 \oplus h_2, \quad h_1:=\pi_1^\ast(h f h_{\Sigma}^{r_1}) \otimes \pi_2^\ast(h_{\FS}^{2r_2}), \quad h_2:=\pi_1^\ast(f h_{\Sigma}^{r_1}) \otimes \pi_2^\ast(h_{\FS}^{2r_2+2}).
\]
So $H, h_1, h_2$ are Hermitian metrics on $E$, $\pi_1^\ast((r_1+1)L) \otimes \pi_2^\ast (r_2 \cO(2))$, $\pi_1^\ast(r_1 L) \otimes \pi_2^\ast ((r_2+1) \cO(2))$ respectively. We choose a holomorphic structure on $E$ so that the second fundamental form $A$ satisfies $A^\ast:=-\pi_1^\ast \phi \otimes \pi_2^\ast \zeta$, where $\ast$ means the adjoint with respect to $H$. Then the associated connection $D_H$ of $H$ has the following expression
\[
D_H=
\begin{pmatrix}
D_{h_1} & -A^\ast \\
A & D_{h_2}
\end{pmatrix},
\]
where $D_{h_1}$, $D_{h_2}$ denotes the associated connection of $h_1$ on $\pi_1^\ast((r_1+1)L) \otimes \pi_2^\ast (r_2 \cO(2))$ and $h_2$ on $\pi_1^\ast(r_1 L) \otimes \pi_2^\ast ((r_2+1) \cO(2))$ respectively. To compute $A$, one can use the induced Hermitian metric $h_1 \otimes h_2^{-1}=\pi_1^\ast h \otimes \pi_2^\ast (h_{\FS}^{-2})$ on $\pi_1^\ast L \otimes \pi_2^\ast \cO(-2)$. By using $h_{\FS}^{-2}(dz,dz)=(1+|z|^2)^2$, we have
\[
\frac{\i}{2 \pi} A^\ast A=\frac{\i}{2 \pi} \pi_1^\ast( \phi \phi^\ast) \otimes \pi_2^\ast( \zeta \wedge \zeta^\ast)=-|\phi|_h^2 \omega_{\FS},
\]
where we regard $\phi$ as a homomorphism from the trivial $\C$-bundle to $L$ and the adjoint $\phi^\ast$ is taken with respect to $h$, \ie $\phi^\ast=h(\cdot,\phi)$. Similarly,
\[
D' \zeta=D' \bigg( \frac{\s}{(1+|z|^2)^2} \bigg) \otimes d\bar{z}.
\]
Since
\[
D'(\s)=\p \log (1+|z|^2)^2 \cdot \s,
\]
we observe that
\[
\begin{aligned}
D' \bigg( \frac{\s}{(1+|z|^2)^2} \bigg)&=\frac{1}{(1+|z|^2)^2} \p \log (1+|z|^2)^2 \cdot \s+\p \bigg(\frac{1}{(1+|z|^2)^2} \bigg) \cdot \s \\
&=\frac{1}{(1+|z|^2)^4} \p (1+|z|^2)^2 \cdot \s-\frac{1}{(1+|z|^2)^4} \p (1+|z|^2)^2 \cdot \s \\
&=0.
\end{aligned}
\]
This shows that
\[
D' A^\ast=-D_h' \phi \wedge \zeta.
\]
Thus we have
\begin{equation} \label{computation for FH}
\begin{aligned}
F_H&=
\begin{pmatrix}
F_{h_1}-\frac{\i}{2 \pi} A^\ast A & -\frac{\i}{2 \pi} D' A^\ast \\
\frac{\i}{2 \pi} D'' A & F_{h_2}-\frac{\i}{2 \pi} A A^\ast
\end{pmatrix} \\
&=
\begin{pmatrix}
F_h+F_f+r_1 \omega_\Sigma+2r_2 \omega_{\FS}+|\phi|_h^2 \omega_{\FS} & \frac{\i}{2\pi} D_h' \phi \wedge \zeta \\
-\frac{\i}{2\pi} D_h'' \phi^\ast \wedge \zeta^\ast & F_f+r_1 \omega_\Sigma+(2r_2+2) \omega_{\FS}-|\phi|_h^2 \omega_{\FS}
\end{pmatrix}.
\end{aligned}
\end{equation}
So in this case, the $J$-equation is given by a system of equations
\[
c \bigg( F_{h_1}-\frac{\i}{2 \pi} A^\ast A \bigg)^2-c \bigg(\frac{\i}{2 \pi} \bigg)^2 D' A^\ast D'' A -\omega \bigg( F_{h_1}-\frac{\i}{2 \pi} A^\ast A \bigg)=0,
\]
\[
c \bigg( F_{h_2}-\frac{\i}{2 \pi} A A^\ast \bigg)^2-c \bigg(\frac{\i}{2 \pi} \bigg)^2 D'' A D' A^\ast-\omega \bigg( F_{h_2}-\frac{\i}{2 \pi} A A^\ast \bigg)=0.
\]
We note that
\[
\bigg( \frac{\i}{2\pi} \bigg)^2 D' A^\ast D'' A=-\bigg( \frac{\i}{2\pi} \bigg)^2 D_h' \phi D_h'' \phi^\ast \pi_2^\ast \zeta \pi_2^\ast \zeta^\ast=\frac{\i}{2 \pi} D_h' \phi D_h'' \phi^\ast \omega_{\FS}.
\]
By using this, the above system further can be reduced to the following system of equations
\begin{equation} \label{hfS}
\begin{aligned}
& 2c(F_h+F_f+r_1 \omega_\Sigma)(2 r_2+|\phi|_h^2)-c \frac{\i}{2 \pi} D_h' \phi D_h'' \phi^\ast \\
& -2r_2 s \omega_\Sigma-(F_h+F_f+r_1 \omega_\Sigma)-s|\phi|_h^2 \omega_\Sigma=0,
\end{aligned}
\end{equation}
\begin{equation}
\begin{aligned} \label{fS}
& 2c (F_f+r_1 \omega_\Sigma) (2r_2+2-|\phi|_h^2)-c \frac{\i}{2 \pi} D_h' \phi D_h'' \phi^\ast \\
& -(2r_2+2)s \omega_{\Sigma}-(F_f+r_1 \omega_\Sigma)+s|\phi|_h^2 \omega_\Sigma=0.
\end{aligned}
\end{equation}

Now we assume that $\phi$ is a non-zero section satisfying $|\phi|_h^2 \leq 1$, and
\begin{equation} \label{condition for r1 r2 s}
r_2 \geq 2, \quad \frac{(2r_1+1)\kappa_0+4r_1}{2(4r_2-\kappa_0)(2r_2+1)}<s<\frac{r_1(r_1+1)}{2r_2(r_2+1)},
\end{equation}
where
\[
\kappa_0:=\frac{1}{3} \bigg(2+\sqrt[3]{1232-528 \sqrt{3}}+2 \sqrt[3]{22(7+3 \sqrt{3})} \bigg) \approx 7.2405
\]
denotes the (unique) positive root of the polynomial $\kappa^3-2\kappa^2-28\kappa-72$. Indeed, we have already showed that the upper bound for $s$ is necessary. From \eqref{formula for c}, one can easily check that the lower bound for $s$ means
\begin{equation} \label{lower bound for 4cr2 minus 1}
4cr_2-1>\kappa_0 c.
\end{equation}
In particular, we have
\[
4cr_2+2c|\phi|_h^2-1>\kappa_0 c, \quad 4cr_2-2c|\phi|_h^2-1+4c>(\kappa_0+2)c.
\]
Thus combining \eqref{hfS} with \eqref{fS}, we obtain
\begin{equation} \label{equation for f}
F_f+r_1 \omega_\Sigma=\frac{4c r_2+2c |\phi|_h^2-1}{4c(1-|\phi|_h^2)} F_h+\frac{s}{2c} \omega_\Sigma.
\end{equation}
Substituting \eqref{equation for f} to \eqref{fS}, we have
\begin{equation}
\begin{aligned} \label{J-vortex equation}
F_h=2(1-|\phi|_h^2) \frac{\frac{\i c^2}{\pi} D_h' \phi D_h'' \phi^\ast+s \omega_\Sigma}{(4cr_2-2c|\phi|_h^2-1+4c)(4cr_2+2c|\phi|_h^2-1)}.
\end{aligned}
\end{equation}
We will refer to \eqref{J-vortex equation} as the $J$-vortex equation. We note that the derivative $D_h' \phi D_h'' \phi^\ast$ can be computed by using the Weitzenb\"ock type formula
\begin{equation} \label{Weitzenbock type formula}
\frac{\i}{2 \pi} \p \bp |\phi|_h^2=-F_h |\phi|_h^2+\frac{\i}{2 \pi} D_h' \phi D_h'' \phi^\ast.
\end{equation}

\begin{lem}
If the equation \eqref{J-vortex equation} has a smooth solution $h$ satisfying $|\phi|_h^2 \leq 1$, then the equation \eqref{equation for f} has a smooth solution $f$ unique up to additive constants.
\end{lem}
\begin{proof}
The equation \eqref{J-vortex equation} can be written as\footnote{We remark that the form $\frac{F_h}{1-|\phi|_h^2}$ is smooth on $X$ since the form $F_h$ can be divided by $(1-|\phi|_h^2)$.}
\[
\frac{\i}{2 \pi} D_h' \phi D_h'' \phi^\ast=\frac{(4cr_2-2c|\phi|_h^2-1+4c)(4cr_2+2c|\phi|_h^2-1)}{4c^2(1-|\phi|_h^2)} F_h-\frac{s}{2c^2} \omega_\Sigma.
\]
Combining with \eqref{Weitzenbock type formula}, we have
\[
\begin{aligned}
\int_\Sigma F_h |\phi|_h^2&=\frac{\i}{2 \pi} \int_\Sigma D_h' \phi D_h'' \phi^\ast \\
&=\int_\Sigma \frac{(4cr_2-2c|\phi|_h^2-1+4c)(4cr_2+2c|\phi|_h^2-1)}{4c^2(1-|\phi|_h^2)} F_h-\frac{s}{2c^2} \int_\Sigma \omega_\Sigma,
\end{aligned}
\]
and hence
\[
2s+4c^2=(4cr_2+2c-1)^2 \int_\Sigma \frac{F_h}{1-|\phi|_h^2}.
\]
In order to solve \eqref{equation for f}, it suffices to check that the integrals of both hand sides are equal. Indeed, we have
\[
\int_\Sigma(F_f+r_1 \omega_\Sigma)=r_1,
\]
\[
\begin{aligned}
\int_\Sigma \frac{4c r_2+2c |\phi|_h^2-1}{4c(1-|\phi|_h^2)} F_h+\frac{s}{2c} \int_\Sigma \omega_\Sigma&=\int_\Sigma \bigg( -\frac{1}{2}+\frac{4cr_2+2c-1}{4c(1-|\phi|_h^2)} \bigg) F_h+\frac{s}{2c} \\
&=-\frac{1}{2}+\frac{4cr_2+2c-1}{4c} \int_\Sigma \frac{F_h}{1-|\phi|_h^2}+\frac{s}{2c} \\
&=-\frac{1}{2}+\frac{s+2c^2}{2c(4cr_2+2c-1)}+\frac{s}{2c} \\
&=\frac{4r_2s-4cr_2+2s+1}{2(4cr_2+2c-1)} \\
&=r_1.
\end{aligned}
\]
This completes the proof.
\end{proof}

\begin{lem}
If $h$ and $f$ are the solutions to \eqref{J-vortex equation} and \eqref{equation for f} respectively, then the corresponding Hermitian metric $H$ is $J$-Griffiths positive.
\end{lem}
\begin{proof}
At each point $x \in X$, we identify $v \in E_x$ with a $2$-vector $\begin{pmatrix} v_1 \\ v_2 \end{pmatrix} \in \C^2$ by using an $H$-orthonormal frame of $E$. Then we may show that $v^\ast (2cF_H-\omega \Id_E) v$ is a positive $(1,1)$-form at any points $x \in X$ for all non-zero $v \in E_x$. We compute
\[
\begin{aligned}
& v^\ast (2c F_H-\omega \Id_E) v \\
&=2c \begin{pmatrix} \overline{v_1} & \overline{v_2} \end{pmatrix}
\begin{pmatrix}
F_h+F_f+r_1 \omega_\Sigma+2r_2 \omega_{\FS}+|\phi|_h^2 \omega_{\FS} & \frac{\i}{2\pi} D'_h \phi \wedge \zeta \\
-\frac{\i}{2\pi} D''_h \phi^\ast \wedge \zeta^\ast & F_f+r_1 \omega_\Sigma+(2r_2+2) \omega_{\FS}-|\phi|_h^2 \omega_{\FS}
\end{pmatrix}
\begin{pmatrix} v_1 \\ v_2 \end{pmatrix} \\
& -(s \omega_\Sigma+\omega_{\FS})(|v_1|^2+|v_2|^2) \\
&=2c |v_1|^2 (F_h+F_f+r_1 \omega_\Sigma+2r_2 \omega_{\FS}+|\phi|_h^2 \omega_{\FS})+2c |v_2|^2 (F_f+r_1 \omega_\Sigma+(2r_2+2) \omega_{\FS}-|\phi|_h^2 \omega_{\FS}) \\
&+\overline{v_1} v_2 \frac{\i c}{\pi} D'_h \phi \wedge \zeta-\overline{v_2} v_1 \frac{\i c}{\pi} D''_h \phi^\ast \wedge \zeta^\ast-(s \omega_\Sigma+\omega_{\FS}) (|v_1|^2+|v_2|^2) \\
&=|v_1|^2 \big(2cF_h+2cF_f+(2c r_1-s) \omega_\Sigma \big)+|v_2|^2 \big( 2cF_f+(2cr_1-s) \omega_\Sigma \big)+\omega_{\FS} \big( |v_1|^2 (4c r_2+2c|\phi|_h^2-1)\\ 
& +|v_2|^2(4cr_2-2c|\phi|_h^2-1+4c) \big) \\
&+\overline{v_1} v_2 \frac{\i c}{\pi} D'_h \phi \wedge \zeta-\overline{v_2} v_1 \frac{\i c}{\pi} D''_h \phi^\ast \wedge \zeta^\ast.
\end{aligned}
\]
Thus the form $v^\ast (2c F_H-\omega \Id_E) v$ is positive for all $v \neq 0$ if and only if
\begin{align}
2cF_h+2cF_f+(2c r_1-s) \omega_\Sigma>0, \label{p I} \\
2cF_f+(2cr_1-s) \omega_\Sigma>0, \label{p II} \\
4cr_2+2c|\phi|_h^2-1>0, \label{p III} \\
4cr_2-2c|\phi|_h^2-1+4c>0, \label{p IV}
\end{align}
and
\begin{equation} \label{p V}
\begin{aligned}
0&<\frac{\big( v^\ast (2c F_H-\omega \Id_E) v \big)^2}{2} \\
&=\omega_{\FS} \bigg[ |v_1|^4 (4cr_2+2c |\phi|_h^2-1)(2cF_h+2cF_f+(2cr_1-s)\omega_\Sigma) \\
&+|v_2|^4(4cr_2-2c|\phi|_h^2-1+4c)(2cF_f+(2cr_1-s)\omega_\Sigma) \\
&+|v_1 v_2|^2 \bigg(2c(4cr_2-2c|\phi|_h^2-1+4c)F_h+2(4cr_2+2c-1) \big(2cF_f+(2cr_1-s)\omega_\Sigma \big) \\
&-\frac{2c^2 \i}{\pi}D_h' \phi D_h'' \phi^\ast \bigg) \bigg],
\end{aligned}
\end{equation}
where \eqref{p I}, \eqref{p II} are inequalities as $(1,1)$-forms on $\Sigma$. We have already observed that the conditions \eqref{p III} and \eqref{p IV} are satisfied from our choice of $r_1,r_2,s$. Moreover, from \eqref{hfS} and \eqref{fS}, we have
\[
2cF_h+2cF_f+(2c r_1-s) \omega_\Sigma=\frac{\frac{\i c^2}{\pi} D_h' \phi D_h'' \phi^\ast+s \omega_\Sigma}{4cr_2+2c|\phi|_h^2-1}>0,
\]
\[
2cF_f+(2cr_1-s)\omega_\Sigma=\frac{\frac{\i c^2}{\pi} D_h' \phi D_h'' \phi^\ast+s \omega_\Sigma}{4cr_2-2c|\phi|_h^2-1+4c}>0
\]
on $\Sigma$, so the conditions \eqref{p I} and \eqref{p II} are satisfied. Finally, we will check \eqref{p V}. From the above observations, we know that the coefficients of $|v_1|^4$ and $|v_2|^4$ are positive. So we will prove that the coefficient of $|v_1v_2|^2$ is non-negative. Indeed, we have
\[
\begin{aligned}
& 2c(4cr_2-2c|\phi|_h^2-1+4c)F_h+2(4cr_2+2c-1) \big(2cF_f+(2cr_1-s)\omega_\Sigma \big)-\frac{2c^2 \i}{\pi}D_h' \phi D_h'' \phi^\ast \\
&=4c(1-|\phi|_h^2) \frac{\frac{\i c^2}{\pi} D_h' \phi D_h'' \phi^\ast+s \omega_\Sigma}{4cr_2+2c|\phi|_h^2-1}+2(4cr_2+2c-1) \cdot \frac{\frac{\i c^2}{\pi} D_h' \phi D_h'' \phi^\ast+s \omega_\Sigma}{4cr_2-2c|\phi|_h^2-1+4c} \\
&-\frac{2c^2 \i}{\pi}D_h' \phi D_h'' \phi^\ast \\
&\geq 4c(1-|\phi|_h^2) \frac{\frac{\i c^2}{\pi} D_h' \phi D_h'' \phi^\ast}{4cr_2+2c|\phi|_h^2-1}+2(4cr_2+2c-1) \cdot \frac{\frac{\i c^2}{\pi} D_h' \phi D_h'' \phi^\ast}{4cr_2-2c|\phi|_h^2-1+4c}-\frac{2c^2 \i}{\pi}D_h' \phi D_h'' \phi^\ast \\
&=\frac{16c^4(1-|\phi|_h^2)^2}{(4cr_2+2c|\phi|_h^2-1)(4cr_2-2c|\phi|_h^2-1+4c)} \cdot \frac{\i}{\pi} D_h' \phi D_h'' \phi^\ast \\
& \geq 0.
\end{aligned}
\]
Hence the metric $H$ is $J$-Griffiths positive with respect to $\omega$.
\end{proof}

\subsubsection{Continuity method}
In order to obtain a solution to \eqref{J-vortex equation}, we use the continuity method as follows: multiplying a scaling factor, we may always assume that $|\phi|_{h_\Sigma}^2<\frac{1}{2}$. Let $\cT \subset [0,1]$ be the set of all $t$ such that the following equation has a smooth solution $h_t:=h_\Sigma e^{-\psi_t}$
\begin{equation} \label{continuity path}
F_{h_t}=\omega_\Sigma+\frac{\i}{2 \pi} \p \bp \psi_t=2(1-|\phi|_{h_t}^2) \frac{\frac{\i c^2 t}{\pi} D'_t \phi D''_t \phi^{\ast_t}+su^{1-t} \omega_\Sigma}{(4cr_2-2ct|\phi|_{h_t}^2-1+4c)(4cr_2+2ct|\phi|_{h_t}^2-1)},
\end{equation}
where $\ast_t$, $D'_t$ and $D''_t$ are taken with respect to $h_t$ and
\[
u:=\frac{1}{\a(1-|\phi|_{h_\Sigma}^2)}, \quad \a:=\frac{2s}{(4cr_2-1+4c)(4cr_2-1)}.
\]
One can check that there is a unique solution $\psi_0=0$ when $t=0$ so that $0 \in \cT$. To simplify computations, we often write
\[
I=I(\psi,t):=4cr_2+2ct|\phi|_h^2-1, \quad K=K(\psi,t):=4cr_2-2ct|\phi|_h^2-1+4c,
\]
\[
J=J(\psi,t):=\frac{\frac{\i c^2 t}{\pi} D'_h \phi D''_h \phi^{\ast_h}+su^{1-t} \omega_\Sigma}{IK},
\]
where $h=h_\Sigma e^{-\psi}$. Then we have
\[
F=2(1-|\phi|_h^2)J.
\]
To compute the linearization of \eqref{continuity path}, we need to prove some lemmas.
\begin{lem} \label{linearization of D phi D phi ast}
The linearization of $D'_h \phi D''_h \phi^\ast$ in the direction $v \in C^\infty(\Sigma;\R)$ is given by
\[
\frac{\i}{2 \pi} \d_v ( D'_h \phi D''_h \phi^\ast)=-v \frac{\i}{2 \pi} D'_h \phi D''_h \phi^\ast-\frac{\i}{2 \pi} \p v \wedge \bp |\phi|_h^2-\frac{\i}{2 \pi} \p |\phi|_h^2 \wedge \bp v.
\]
\begin{proof}
By using \eqref{Weitzenbock type formula}, we obtain
\[
\begin{aligned}
\frac{\i}{2 \pi} \d_v ( D'_h \phi D''_h \phi^\ast)&=\d_v (F_h |\phi|_h^2)+\frac{\i}{2 \pi}\p \bp \d_v |\phi|_h^2 \\
&=|\phi|_h^2 \frac{\i}{2 \pi}\p \bp v-F_h |\phi|_h^2 v-\frac{\i}{2 \pi} \p \bp (|\phi|_h^2 v) \\
&=-F_h |\phi|_h^2 v-v \frac{\i}{2 \pi}\p \bp |\phi|_h^2-\frac{\i}{2 \pi} \p v \wedge \bp |\phi|_h^2-\frac{\i}{2 \pi} \p |\phi|_h^2 \wedge \bp v \\
&=-v \frac{\i}{2 \pi} D'_h \phi D''_h \phi^\ast-\frac{\i}{2 \pi} \p v \wedge \bp |\phi|_h^2-\frac{\i}{2 \pi} \p |\phi|_h^2 \wedge \bp v.
\end{aligned}
\]
\end{proof}
\end{lem}
In particular, the above lemma shows that \eqref{J-vortex equation} as well as \eqref{continuity path} are elliptic since the second derivative of $v$ only appears from $\d_v F_h=\frac{\i}{2 \pi} \p \bp v$ in their linearized operators.

\begin{lem}
We have
\[
\p |\phi|_h^2 \wedge \bp |\phi|_h^2=|\phi|_h^2 D'_h \phi \wedge D''_h \phi^\ast.
\]
\end{lem}
\begin{proof}
Let $\eta$ be the section of $L$ satisfying $h(\eta,\phi)=1$ on a Zariski open subset $\{\phi \neq 0\}$. Then we have
\[
(D'_h \phi^\ast)(\eta)=-\phi^\ast (D'_h \eta)=-h(D'_h \eta, \phi)=h(\eta, D''_h \phi)=0.
\]
This together with $D_h'' \phi=0$ implies that
\[
\p |\phi|_h^2 \wedge \bp |\phi|_h^2=D_h' \phi \cdot \phi^\ast \wedge D''_h \phi^\ast \cdot \phi=|\phi|_h^2 D'_h \phi \wedge D''_h \phi^\ast.
\]
\end{proof}
Finally, we will show that the condition $|\phi|_{h_t}^2 \leq 1$ is preserved along the path \eqref{continuity path}.
\begin{lem}
We have $|\phi|_{h_t}^2 \leq 1$ for all $t \in \cT$.
\end{lem}
\begin{proof}
We apply the formula \eqref{Weitzenbock type formula} to each $h_t$
\[
\frac{\i}{2 \pi} \p \bp |\phi|_{h_t}^2=-F_{h_t} |\phi|_{h_t}^2+\frac{\i}{2 \pi} D'_t \phi D''_t \phi^{\ast_t}.
\]
This shows that at the maximum point $\hat{x} \in \Sigma$ of $|\phi|_{h_t}^2$, we have
\[
0 \geq -F_{h_t} |\phi|_{h_t}^2
\]
at $\hat{x}$. Since $\phi$ is not identically zero, we have $F_{h_t}(\hat{x}) \geq 0$. Thus from \eqref{continuity path}, we have
\[
|\phi|_{h_t}^2 \leq |\phi|_{h_t}^2(\hat{x}) \leq 1
\]
as desired.
\end{proof}

\subsubsection{Openness}
We can prove the openness along the path \eqref{continuity path} in the similar way as \cite[Section 4]{Pin20}. However, we will give a detailed proof here since it is deeply committed to our choice of $\kappa_0$ (see Lemma \ref{kernel of the adjoint operator is trivial} for more details).

Assume $t_0 \in \cT$. For any $\ell \in \Z_{ \geq 0}$ and $\b \in (0,1)$, let $\cC^{\ell,\b}$ be an open subset of $C^{\ell,\b}(\Sigma;\R)$ consisting of all $C^{\ell,\b}$ functions $\psi$ satisfying $4cr_2-2c|\phi|_{h_\Sigma e^{-\psi}}^2-1+4c>0$. Set $h:=h_\Sigma e^{-\psi}$ and define an operator $\cP \colon \cC^{\ell+2,\b} \times [0,1] \to C^{\ell,\b}(\Sigma;\R)$ by
\[
\cP(\psi,t):=F_h-2(1-|\phi|_h^2) J.
\]
Then $\cP$ is a smooth map between Banach manifolds with $\cP(\psi_{t_0},t_0)=0$. If we prove the linearization $\d|_{(\psi_{t_0},t_0),v} \cP \colon C^{\ell+2,\b}(X;\R) \to C^{\ell,\b}(X;\R)$ is an isomorphism, then by the implicit function theorem of Banach manifolds, we know that $\cT$ is open near $t_0$. First we compute the derivative of $J$ in direction $v$ as
\[
\begin{aligned}
\d_v J&=\frac{\frac{\i c^2}{\pi} t \d_v (D_h' \phi D_h'' \phi^{\ast h}) IK-(\frac{\i c^2 t}{\pi} D_h' \phi D_h'' \phi^{\ast h}+su^{1-t} \omega_\Sigma) \d_v(IK)}{I^2K^2} \\
&=\frac{\i c^2t}{\pi} \frac{\d_v(D_h' \phi D_h'' \phi^{\ast h})}{IK}+\frac{8c^2tvJ |\phi|_h^2 (1-t |\phi|_h^2)}{IK},
\end{aligned}
\]
where we used
\[
\begin{aligned}
\d_v (IK)&=\d_v I \cdot K+I \cdot \d_v K \\
&=-2ct v |\phi|_h^2 \cdot K+I \cdot 2ct v |\phi|_h^2 K \\
&=-2ct v |\phi|_h^2(K-I) \\
&=-8c^2t v |\phi|_h^2(1-t|\phi|_h^2).
\end{aligned}
\]
Combining with Lemma \ref{linearization of D phi D phi ast}, we compute the linearization of $\cP$ as
\begin{equation} \label{linearization of P}
\begin{aligned}
\d_v \cP=\frac{\i}{2 \pi} \p \bp v-2|\phi|_h^2 v J-16c^2(1-|\phi|_h^2)(1-t|\phi|_h^2) \frac{t v |\phi|_h^2 J}{IK} \\
+\frac{2 \i c^2 t}{\pi} (1-|\phi|_h^2) \frac{v D'_h \phi D''_h \phi^{\ast_h}+\p v \wedge \bp |\phi|_h^2+\p |\phi|_h^2 \wedge \bp v}{IK}.
\end{aligned}
\end{equation}

\begin{lem}
We have $[0,\d) \subset \cT$ for some $\d>0$.
\end{lem}
\begin{proof}
We compute the linearization of $\cP$ at $(\psi_0,0)=(0,0)$ as
\[
\d|_{(\psi_0,0),v} \cP=\frac{\i}{2 \pi} \p \bp v-2|\phi|_{h_\Sigma}^2 v J(\psi_0,0)=\frac{\i}{2 \pi} \p \bp v-\frac{|\phi|_{h_\Sigma}^2 v}{1-|\phi|_{h_\Sigma}^2} \omega_\Sigma.
\]
If $v \in \Ker \d|_{(\psi_0,0)} \cP$, then multiplying $v$ and integrating by parts, we get
\[
\int_\Sigma \frac{(|\phi|_{h_\Sigma} v)^2}{1-|\phi|_{h_\Sigma}^2} \omega_\Sigma=0.
\]
This shows that $v=0$ on a Zariski open subset $\{\phi \neq 0\}$. Since $v$ is continuous, we have $v=0$ on $\Sigma$. Thus we know that $\Ker \d|_{(\psi_0,0)} \cP=0$, and hence $\d|_{(\psi_0,0)} \cP$ is an isomorphism by Fredholm alternative.
\end{proof}

In what follows, we may assume that $t_0>0$, and all computations are performed at $(\psi_0,t_0)$ although we suppress the dependence on $(\psi_0,t_0)$. In this case, the linearization \eqref{linearization of P} is not self-adjoint due to the terms
\begin{equation} \label{exceptional terms}
\frac{2 \i c^2 t}{\pi} (1-|\phi|_h^2) \frac{\p v \wedge \bp |\phi|_h^2+\p |\phi|_h^2 \wedge \bp v}{IK}.
\end{equation}
\begin{lem} \label{formula for exceptional terms}
The formal adjoint operator of \eqref{exceptional terms} is given by
\[
\begin{aligned}
& \frac{4 \i c^2 tv (2|\phi|_h^2-1)}{\pi IK} D'_h \phi D''_h \phi^{\ast_h}+\frac{32 \i c^4 t^2 v (1-|\phi|_h^2)(1-t |\phi|_h^2) |\phi|_h^2}{\pi I^2 K^2} D'_h \phi D''_h \phi^{\ast_h} \\
&-\frac{2 \i c^2 t (1-|\phi|_h^2)}{\pi IK} (\p v \wedge \bp |\phi|_h^2+\p |\phi|_h^2 \wedge \bp v)+\frac{16c^2tv(1-|\phi|_h^2)^2 |\phi|_h^2}{IK} J.
\end{aligned}
\]
\end{lem}
\begin{proof}
The formal adjoint operator is given by
\[
\frac{2 \i c^2 t}{\pi} \bigg[ -\p \bigg( \frac{(1-|\phi|_h^2) v \bp |\phi|_h^2}{IK} \bigg)+\bp \bigg( \frac{(1-|\phi|_h^2) v \p |\phi|_h^2}{IK} \bigg) \bigg].
\]
We compute
\[
\p (IK)=\p I \cdot K+I \cdot \p K=2ct(K-I) \p |\phi|_h^2=8c^2 t(1-t |\phi|_h^2) \p |\phi|_h^2,
\]
\[
\p \bigg( \frac{1-|\phi|_h^2}{IK} \bigg)=-\frac{\p |\phi|_h^2 \cdot IK+(1-|\phi|_h^2) \p (IK)}{I^2 K^2}=-\frac{\p |\phi|_h^2}{IK}-\frac{8c^2 t(1-|\phi|_h^2)(1-t |\phi|_h^2)}{I^2 K^2} \p |\phi|_h^2.
\]
This together with \eqref{Weitzenbock type formula}, Lemma \ref{linearization of D phi D phi ast} and $F_h=2(1-|\phi|_h^2)J$, we obtain
\[
\begin{aligned}
&\frac{2 \i c^2 t}{\pi} \bigg[ -\p \bigg( \frac{(1-|\phi|_h^2) v \bp |\phi|_h^2}{IK} \bigg)+\bp \bigg( \frac{(1-|\phi|_h^2) v \p |\phi|_h^2}{IK} \bigg) \bigg] \\
&=\frac{2 \i c^2 t}{\pi} \bigg[ \frac{2v \p |\phi|_h^2 \wedge \bp |\phi|_h^2}{IK}+\frac{16c^2 tv(1-|\phi|_h^2)(1-t |\phi|_h^2)}{I^2 K^2} \p |\phi|_h^2 \wedge \bp |\phi|_h^2 \\
&-\frac{1-|\phi|_h^2}{IK} (\p v \wedge \bp |\phi|_h^2+\p |\phi|_h^2 \wedge \bp v)-\frac{2v(1-|\phi|_h^2)}{IK} \p \bp |\phi|_h^2 \bigg] \\
&=\frac{4 \i c^2 t v |\phi|_h^2 }{\pi IK} D'_h \phi D''_h \phi^{\ast_h}+\frac{32 \i c^4 t^2 v(1-|\phi|_h^2)(1-t |\phi|_h^2) |\phi|_h^2}{\pi I^2 K^2} D'_h \phi D''_h \phi^{\ast_h} \\
&-\frac{2 \i c^2 t (1-|\phi|_h^2)}{\pi IK} (\p v \wedge \bp |\phi|_h^2+\p |\phi|_h^2 \wedge \bp v) \\
&-\frac{8c^2tv(1-|\phi|_h^2)}{IK} \bigg(-2(1-|\phi|_h^2)|\phi|_h^2 J+\frac{\i}{2 \pi} D'_h \phi D''_h \phi^{\ast_h} \bigg) \bigg] \\
&=\frac{4 \i c^2 tv |\phi|_h^2}{\pi IK} D'_h \phi D''_h \phi^{\ast_h}+\frac{32 \i c^4 t^2 v (1-|\phi|_h^2)(1-t |\phi|_h^2) |\phi|_h^2}{\pi I^2 K^2} D'_h \phi D''_h \phi^{\ast_h} \\
&-\frac{2 \i c^2 t (1-|\phi|_h^2)}{\pi IK} (\p v \wedge \bp |\phi|_h^2+\p |\phi|_h^2 \wedge \bp v) \\
&+\frac{16c^2tv(1-|\phi|_h^2)^2 |\phi|_h^2}{IK} J-\frac{4 \i c^2tv}{\pi IK} (1-|\phi|_h^2) D'_h \phi D''_h \phi^{\ast_h} \\
&=\frac{4 \i c^2 tv (2|\phi|_h^2-1)}{\pi IK} D'_h \phi D''_h \phi^{\ast_h}+\frac{32 \i c^4 t^2 v (1-|\phi|_h^2)(1-t |\phi|_h^2) |\phi|_h^2}{\pi I^2 K^2} D'_h \phi D''_h \phi^{\ast_h} \\
&-\frac{2 \i c^2 t (1-|\phi|_h^2)}{\pi IK} (\p v \wedge \bp |\phi|_h^2+\p |\phi|_h^2 \wedge \bp v)+\frac{16c^2tv(1-|\phi|_h^2)^2 |\phi|_h^2}{IK} J.
\end{aligned}
\]
\end{proof}
By Lemma \ref{formula for exceptional terms}, we compute the formal adjoint operator of $\d|_{(\psi_{t_0},t_0)} \cP$ as
\begin{equation} \label{formula for the adjoint of the linearization of P}
\begin{aligned}
&(\d|_{(\psi_{t_0},t_0)} \cP)^\ast(v) \\
&=\frac{\i}{2 \pi} \p \bp v-2|\phi|_h^2 v J-\frac{16c^2tvJ(1-t)}{IK} (1-|\phi|_h^2)|\phi|_h^4+\frac{2 \i c^2tv}{\pi IK} (3 |\phi|_h^2-1) D'_h \phi D''_h \phi^{\ast_h} \\
&+\frac{32 \i c^4 t^2 v (1-|\phi|_h^2)(1-t |\phi|_h^2) |\phi|_h^2}{\pi I^2 K^2} D'_h \phi D''_h \phi^{\ast_h}-\frac{2 \i c^2 t (1-|\phi|_h^2)}{\pi IK} \big( \p v \wedge \bp |\phi|_h^2+\p |\phi|_h^2 \wedge \bp v \big).
\end{aligned}
\end{equation}
By the Fredholm alternative, if we prove that
\[
\Ker \d|_{(\psi_{t_0},t_0)} \cP=0, \quad \Ker (\d|_{(\psi_{t_0},t_0)} \cP)^\ast=0,
\]
then we are done. To show this, we need the following lemma for $\tilde{v}:=vI$.
\begin{lem} \label{max min tilde v}
At the maximum (resp. minimum) point of $\tilde{v}$, we have
\[
\bp v=-\frac{2ctv \bp |\phi|_h^2}{I}
\]
and
\[
\begin{aligned}
0 &\geq \; \text{(resp. $\leq$)} \; -\frac{4c^2t^2 v \i |\phi|_h^2}{\pi I^2} D'_h \phi D''_h \phi^{\ast_h}-\frac{4ctvJ}{I} (1-|\phi|_h^2) |\phi|_h^2 \\
&+\frac{\i}{\pi I}ctv D'_h \phi D''_h \phi^{\ast_h}+\frac{\i}{2 \pi} \p \bp v.
\end{aligned}
\]
\end{lem}
\begin{proof}
We compute the derivatives of $\tilde{v}$ as
\[
\bp \tilde{v}=2ctv \bp |\phi|_h^2+I \bp v,
\]
\[
\p \bp \tilde{v}=2ct \p v \wedge \bp |\phi|_h^2+2ctv \p \bp |\phi|_h^2+2ct \p |\phi|_h^2 \wedge \bp v+I \p \bp v.
\]
At the maximum (resp. minimum) point of $\tilde{v}$, we have $\i \p \bp \tilde{v} \leq \; \text{(resp. $\geq$)} \; 0$ and $\p \tilde{v}=\bp \tilde{v}=0$. In particular, we have
\[
\bp v=-\frac{2ctv \bp |\phi|_h^2}{I}.
\]
Substituting this to the above, we observe that
\[
\begin{aligned}
0 &\geq \; \text{(resp. $\leq$)} \; \frac{\i}{2 \pi} \p \bp \tilde{v} \\
&=-\frac{4c^2t^2 v \i}{\pi I} \p |\phi|_h^2 \wedge \bp |\phi|_h^2-2ctv F_h |\phi|_h^2+\frac{\i}{\pi}ctv D'_h \phi D''_h \phi^{\ast_h}+I \frac{\i}{2 \pi} \p \bp v \\
&=-\frac{4c^2t^2 v \i |\phi|_h^2}{\pi I} D'_h \phi D''_h \phi^{\ast_h}-4ctvJ (1-|\phi|_h^2) |\phi|_h^2+\frac{\i}{\pi}ctv D'_h \phi D''_h \phi^{\ast_h}+I \frac{\i}{2 \pi} \p \bp v.
\end{aligned}
\]
Dividing both sides by $I$, we obtain the desired result.
\end{proof}

\begin{lem}
We have $\Ker \d|_{(\psi_{t_0},t_0)} \cP=0$.
\end{lem}
\begin{proof}
If $v \in \Ker \d|_{(\psi_{t_0},t_0)} \cP$, then we have
\[
\begin{aligned}
0=\frac{\i}{2 \pi} \p \bp v-2|\phi|_h^2 v J-16c^2(1-|\phi|_h^2)(1-t|\phi|_h^2) \frac{t v |\phi|_h^2 J}{IK} \\
+\frac{2 \i c^2 t}{\pi} (1-|\phi|_h^2) \frac{v D'_h \phi D''_h \phi^{\ast_h}+\p v \wedge \bp |\phi|_h^2+\p |\phi|_h^2 \wedge \bp v}{IK}.
\end{aligned}
\]
In what follows, we perform all computations at the maximum point $\hat{x}$ of $\tilde{v}$. Then at this point, we get
\[
\p v \wedge \bp |\phi|_h^2+\p |\phi|_h^2 \wedge \bp v=-\frac{4ctv \p |\phi|_h^2 \wedge \bp |\phi|_h^2}{I}=-\frac{4ctv |\phi|_h^2 D'_h \phi D''_h \phi^{\ast_h}}{I}.
\]
Substituting this to the above, we obtain
\[
\begin{aligned}
0 &\geq -\frac{4c^2t^2 v \i |\phi|_h^2}{\pi I^2} D'_h \phi D''_h \phi^{\ast_h}-\frac{4ctvJ}{I} (1-|\phi|_h^2) |\phi|_h^2+\frac{\i}{\pi I}ctv D'_h \phi D''_h \phi^{\ast_h}+2|\phi|_h^2 v J \\
&+16c^2(1-|\phi|_h^2)(1-t|\phi|_h^2) \frac{t v |\phi|_h^2 J}{IK}-\frac{2 \i c^2 tv}{\pi IK} (1-|\phi|_h^2) D'_h \phi D''_h \phi^{\ast_h} \\
&+\frac{8 \i c^3 t^2 v}{\pi I^2 K} (1-|\phi|_h^2) |\phi|_h^2 D'_h \phi D''_h \phi^{\ast_h} \\
&=P \frac{\i ctv}{\pi I} D'_h \phi D''_h \phi^{\ast_h}+Q \frac{2su^{1-t} \omega_\Sigma}{IK} |\phi|_h^2 v,
\end{aligned}
\]
where
\[
\begin{aligned}
P&:=-\frac{4ct |\phi|_h^2}{I}-\frac{4c^2t(1-|\phi|_h^2) |\phi|_h^2}{IK}+1+\frac{2c |\phi|_h^2}{K}+\frac{16c^3t}{IK^2}(1-|\phi|_h^2)(1-t |\phi|_h^2) |\phi|_h^2-\frac{2c}{K}(1-|\phi|_h^2) \\
&+\frac{8c^2 t}{IK} (1-|\phi|_h^2) |\phi|_h^2 \\
&\geq -\frac{4c}{I}-\frac{4c^2}{IK}+1-\frac{2c}{K} \\
&>-\frac{4}{\kappa_0}-\frac{4}{\kappa_0(\kappa_0+2)}+1-\frac{2}{\kappa_0+2} \\
&=1-\frac{6}{\kappa_0} \\
&>0,
\end{aligned}
\]
and
\[
\begin{aligned}
Q&:=-\frac{2ct}{I}(1-|\phi|_h^2)+1+\frac{8c^2 t}{IK} (1-|\phi|_h^2)(1-t |\phi|_h^2) \\
& \geq -\frac{2c}{I}+1 \\
&> -\frac{2}{\kappa_0}+1 \\
&>0.
\end{aligned}
\]
Since $L$ is degree $1$, we have $\phi (\hat{x}) \neq 0$ or $D \phi (\hat{x}) \neq 0$. Also we have $t>0$ by the assumption. Taking this into account, we have $\max \tilde{v} \leq 0$. A similar computation shows that $\min \tilde{v} \geq 0$, so we have $\tilde{v}=v=0$ as desired.
\end{proof}

\begin{lem} \label{kernel of the adjoint operator is trivial}
We have $\Ker (\d|_{(\psi_{t_0},t_0)} \cP)^\ast=0$.
\end{lem}
\begin{proof}
We can prove in the same way as the previous lemma. If $v \in \Ker (\d|_{(\psi_{t_0},t_0)} \cP)^\ast$, we have
\[
\begin{aligned}
& 0=\frac{\i}{2 \pi} \p \bp v-2|\phi|_h^2 v J-\frac{16c^2tvJ(1-t)}{IK} (1-|\phi|_h^2)|\phi|_h^4+\frac{2 \i c^2tv}{\pi IK} (3 |\phi|_h^2-1) D'_h \phi D''_h \phi^{\ast_h} \\
&+\frac{32 \i c^4 t^2 v (1-|\phi|_h^2)(1-t |\phi|_h^2) |\phi|_h^2}{\pi I^2 K^2} D'_h \phi D''_h \phi^{\ast_h}-\frac{2 \i c^2 t (1-|\phi|_h^2)}{\pi IK} \big( \p v \wedge \bp |\phi|_h^2+\p |\phi|_h^2 \wedge \bp v \big).
\end{aligned}
\]
In what follows, we perform all computations at the maximum point $\hat{x}$ of $\tilde{v}$. We invoke Lemma \ref{max min tilde v} to get 
\[
\begin{aligned}
0 & \geq -\frac{4c^2t^2v \i}{\pi I^2} |\phi|_h^2 D'_h \phi \wedge D''_h \phi^{\ast_h}-\frac{4ctvJ}{I}(1-|\phi|_h^2) |\phi|_h^2+\frac{\i ctv}{\pi I} D'_h \phi D''_h \phi^{\ast_h} \\
& +2|\phi|_h^2 v J+\frac{16c^2tvJ(1-t)}{IK} (1-|\phi|_h^2)|\phi|_h^4-\frac{2 \i c^2tv}{\pi IK} (3 |\phi|_h^2-1) D'_h \phi D''_h \phi^{\ast_h} \\
&-\frac{32 \i c^4 t^2 v (1-|\phi|_h^2)(1-t |\phi|_h^2) |\phi|_h^2}{\pi I^2 K^2} D'_h \phi D''_h \phi^{\ast_h}-\frac{8 \i c^3 t^2 v}{\pi I^2 K} (1-|\phi|_h^2) |\phi|_h^2 D'_h \phi D''_h \phi^{\ast_h} \\
&=P \frac{\i ctv}{\pi I} D'_h \phi D''_h \phi^{\ast_h}+Q \frac{2su^{1-t} \omega_\Sigma}{IK} |\phi|_h^2 v,
\end{aligned}
\]
where
\[
\begin{aligned}
P&:=-\frac{4ct}{I} |\phi|_h^2-\frac{4c^2t}{IK}(1-|\phi|_h^2) |\phi|_h^2+1+\frac{2c}{K} |\phi|_h^2+\frac{16c^3t(1-t)}{IK^2} (1-|\phi|_h^2) |\phi|_h^4-\frac{2c}{K}(3 |\phi|_h^2-1) \\
&-\frac{32c^3t}{IK^2}(1-|\phi|_h^2)(1-t |\phi|_h^2) |\phi|_h^2-\frac{8c^2t}{IK} (1-|\phi|_h^2) |\phi|_h^2 \\
&=-\frac{4ct}{I} |\phi|_h^2-\frac{4c^2t}{IK}(1-|\phi|_h^2) |\phi|_h^2+1-\frac{4c}{K} |\phi|_h^2+\frac{16c^3t(1-t)}{IK^2} (1-|\phi|_h^2) |\phi|_h^4+\frac{2c}{K} \\
&-\frac{32c^3t}{IK^2}(1-|\phi|_h^2)(1-t |\phi|_h^2) |\phi|_h^2-\frac{8c^2t}{IK} (1-|\phi|_h^2) |\phi|_h^2 \\
& \geq -\frac{4c}{I}-\frac{4c^2}{IK}+1-\frac{4c}{K}+\frac{2c}{K}-\frac{32c^3}{IK^2}-\frac{8c^2}{IK} \\
&=-\frac{4c}{I}-\frac{12c^2}{IK}+1-\frac{2c}{K}-\frac{32c^3}{IK^2} \\
&> -\frac{4}{\kappa_0}-\frac{12}{\kappa_0(\kappa_0+2)}+1-\frac{2}{\kappa_0+2}-\frac{32}{\kappa_0(\kappa_0+2)^2} \\
&=\frac{\kappa_0^3-2\kappa_0^2-28 \kappa_0-72}{\kappa_0(\kappa_0+2)^2} \\
&=0,
\end{aligned}
\]
and
\[
\begin{aligned}
Q&:=-\frac{2ct}{I}(1-|\phi|_h^2)+1+\frac{8c^2 t(1-t)}{IK}(1-|\phi|_h^2) |\phi|_h^2 \\
& \geq -\frac{2c}{I}+1 \\
&> -\frac{2}{\kappa_0}+1 \\
&>0.
\end{aligned}
\]
Thus we obtain $\max \tilde{v} \leq 0$. Similarly, we have $\min \tilde{v} \geq 0$, and hence $\tilde{v}=v=0$, which proves the result.
\end{proof}

\subsubsection{Closedness}

Define the maximal existence time $T$ as
\[
T:=\sup \cT>0.
\]
We prove the following crucial observation in which the upper bound for $s$ is used.
\begin{lem} \label{alpha is greater than one}
We have $\a>1$.
\end{lem}
\begin{proof}
Let us consider the function $G$ by
\[
\begin{aligned}
G(s)&:=\big( (r_1+1)r_2+r_1(r_2+1) \big)^2 \big(2s-(4cr_2-1+4c)(4cr_2-1) \big) \\
&=-4r_2(r_2+1)(2r_2+1)^2 s^2+2\big(r_1(r_1+1)(2r_2+1)^2-r_2(r_2+1) \big)s+r_1(r_1+1).
\end{aligned}
\]
By the definition, we have $\a>1$ if and only if $G(s)>0$. We compute the discriminant of $G$ as
\[
\begin{aligned}
\Delta(G)/4&=\big(r_1(r_1+1)(2r_2+1)^2-r_2(r_2+1) \big)^2+4r_2(r_2+1)(2r_2+1)^2 r_1(r_1+1) \\
&=\big( r_1(r_1+1)(2r_2+1)^2+r_2(r_2+1) \big)^2.
\end{aligned}
\]
Thus the roots of $G$ are
\[
-\frac{1}{2(2r_2+1)^2}, \quad \frac{r_1(r_1+1)}{2r_2(r_2+1)}.
\]
This shows that $G(s)>0$ from the assumption \eqref{condition for r1 r2 s}.
\end{proof}
We need lemma \ref{alpha is greater than one} to show the uniform upper bound for $\psi_t$ along the path \eqref{continuity path}. We can prove the following lemma in the same line as in \cite[Lemma 4.8]{Pin20}.
\begin{lem}
The function $\psi_t$ satisfies $\psi_t \leq C$, where $C$ is independent of $t$.
\end{lem}
\begin{proof}
Suppose the maximum of $\psi_t$ is attained at $\hat{x} \in \Sigma$. Then we have $\i \p \bp \psi_t \leq 0$ and
\[
\begin{aligned}
\omega_\Sigma &\geq \omega_\Sigma+\frac{\i}{2 \pi} \p \bp \psi_t \\
&=2(1-|\phi|_{h_t}^2) \frac{\frac{\i c^2 t}{\pi} D'_t \phi D''_t \phi^{\ast_t}+su^{1-t} \omega_\Sigma}{(4cr_2-2ct|\phi|_{h_t}^2-1+4c)(4cr_2+2ct|\phi|_{h_t}^2-1)} \\
& \geq 2(1-|\phi|_{h_t}^2) \frac{su^{1-t} \omega_\Sigma}{(4cr_2-2ct|\phi|_{h_t}^2-1+4c)(4cr_2+2ct|\phi|_{h_t}^2-1)}
\end{aligned}
\]
at $\hat{x} \in \Sigma$. Now we assume that the upper bound does not hold, so there exists a sequence $t_k \nearrow T$ and $\hat{x}_k \in \Sigma$ such that
\[
\psi_{t_k}(\hat{x}_k)=\max \psi_{t_k} \to \infty, \quad \hat{x}_k \to \hat{x}_\infty
\]
for some $\hat{x}_\infty \in \Sigma$ as $k \to \infty$. Then we have $|\phi|_{h_{t_k}}^2(\hat{x}_k) \to 0$ as $k \to \infty$ and hence
\[
1 \geq \frac{2su^{1-T}}{(4cr_2-1+4c)(4cr_2-1)}=\frac{\a^T}{(1-|\phi|_{h_\Sigma}^2)^{1-T}} \geq \a^T>1
\]
at $\hat{x}_\infty$ since $\a>1$. This yields a contradiction.
\end{proof}
Assuming \eqref{lower bound for 4cr2 minus 1}, the remaining part of the proof is exactly the same as \cite[Section 4.5]{Pin20}. So we omit the proof.

\section{The deformed Hermitian--Yang--Mills equation} \label{The deformed Hermitian--Yang--Mills equation}
Let $E$ be a holomorphic vector bundle or rank $r$ over an $n$-dimensional compact K\"ahler manifold $(X,\omega)$. Recall that the deformed Hermitian--Yang--Mills (dHYM) equation for $h \in \Herm(E)$ is given by
\[
\begin{cases}
\Im \big( e^{-\i \Theta} \big(\omega \Id_E+\i F_h \big)^n \big)=0 \\
\int_X \Tr \bigg( \Re \big( e^{-\i \Theta} \big(\omega \Id_E+\i F_h \big)^n \big) \bigg)>0,
\end{cases}
\]
where $\Theta$ is a real constant. If the solution exists, then taking the trace and integrating the above equation over $X$, we know that the constant $\Theta$ must satisfies
\[
Z \in \R_{>0} \cdot e^{\i \Theta},
\]
where the complex constant $Z$ is defined by
\begin{equation} \label{complex number Z}
\begin{aligned}
Z&:=\int_X \Tr \big( (\omega \Id_E+\i F_h)^n \big) \\
&=\sum_{k=0}^n \int_X \binom{n}{k} \omega^k (\i)^{n-k} \Tr(F_h^{n-k}) \\
&=\sum_{k=0}^n \frac{n!}{k!} (\i)^{n-k} [\omega]^k \cdot \ch_{n-k}(E).
\end{aligned}
\end{equation}
In other words, the constant $\Theta$ is determined as the angle of $Z$ modulo $2 \pi$ whenever $Z \neq 0$ holds. Thus we would like to call $\Theta$ the {\it average angle}.

We recall the positivity notion for the dHYM equation introduced in \cite[Definition 2.33]{DMS20}. For any $h \in \Herm(E)$ and $a'' \in \Omega^{0,1}(\End(E))$, we define the $(n-1,n)$-form $\Phi_\omega (h,a'')$ by replacing any term of the form
\[
\omega^k \Id_E \wedge F_h^{n-k} \quad (k=0,1,\ldots,n)
\]
in $\Im \big( e^{-\i \Theta}(\omega \Id_E+\i F_h)^n \big)$ with
\[
\omega^k \Id_E \wedge \sum_{i=0}^{n-k-1} F_h^i a'' F_h^{n-k-1-i} \quad (k=0,1,\ldots,n).
\]
\begin{dfn}
We say that a Hermitian metric $h \in \Herm(E)$ is {\it dHYM-positive} if for all $a'' \in \Omega^{0,1}(\End(E))$, the $(n,n)$-form
\[
-\i \Tr \big( \Phi_\omega (h,a'') \wedge (a'')^\ast \big)
\]
is positive at all points on $X$ where $a'' \neq 0$. Also we say that for any fixed Hermitian metric $h \in \Herm(E)$, a connection $D \in \cA_{\inte}(E,h)$ is dHYM-positive if $h$ is dHYM-positive with respect to the holomorphic structure induced by $D$.
\end{dfn}
As shown in \cite[Lemma 2.34]{DMS20}, we know that the dHYM equation is elliptic for any dHYM-positive Hermitian metrics. Now we will show the following:

\begin{thm}[Theorem \ref{existence result for the dHYM equation}]
Let $E$ be a simple holomorphic vector bundle over an $n$-dimensional compact K\"ahler manifold $(X,\omega)$ satisfying the property \eqref{positivity for Chern characters}. Assume that there exists a $J$-positive Hermitian metric $h_0 \in \Herm(E)$ satisfying the $J$-equation
\[
\frac{[\omega] \cdot \ch_{n-1}(E)}{n \ch_n(E)} F_{h_0}^n-\omega \Id_E F_{h_0}^{n-1}=0.
\]
Then for sufficiently small $\e>0$, the bundle $E$ admits a dHYM-positive Hermitian metric $h_\e$ satisfying the deformed Hermitian--Yang--Mills equation with respect to $\e \omega$
\[
\Im \big(e^{-\i \Theta_\e}(\e \omega \Id_E+\i F_{h_\e})^n \big)=0
\]
for some constant $\Theta_\e$ (depending on $\e$).
\end{thm}

\begin{proof}
For each $\e>0$, let $Z_\e$ be a complex number defined by \eqref{complex number Z} with respect to $\e \omega$. First we check that $\Re(Z_\e)$ and $\Im(Z_\e)$ do not vanish for sufficiently small $\e>0$. Indeed,
\[
\begin{aligned}
Z_\e&=\sum_{k=0}^n \frac{n!}{k!} (\i)^{n-k} [\e \omega]^k \cdot \ch_{n-k}(E) \\
&=n! (\i)^n \ch_n(E)+n!(\i)^{n-1} [\omega] \cdot \ch_{n-1}(E) \e+O(\e^2).
\end{aligned}
\]
Thus both $\Re(Z_\e)$ and $\Im(Z_\e)$ are not zero from the assumption \eqref{positivity for Chern characters}. In particular, the constant $\Theta_\e$ is well-defined modulo $2 \pi$ and the dHYM equation is equivalent to
\begin{equation}
\tan (\Theta_\e) \cdot \Re (\e \omega \Id_E+\i F_{h_\e})^n=\Im (\e \omega \Id_E+\i F_{h_\e})^n.
\end{equation}
In what follows, we only consider the case $n=4m$ for some positive integer $m$ (one can prove other cases in the similar way). Then the constant $\Theta_\e$ satisfies $\cos (\Theta_\e)>0$, $\sin (\Theta_\e)<0$ and
\[
\tan \Theta_\e=\frac{\Im(Z_\e)}{\Re(Z_\e)}=\frac{-n! [\omega] \cdot \ch_{n-1}(E) \e+O(\e^3)}{n! \ch_n(E)+O(\e^2)}=-\frac{[\omega] \cdot \ch_{n-1}(E)}{\ch_n(E)}\e+O(\e^2).
\]
For any $\ell \in \Z_{\geq 0}$ and $\b \in (0,1)$, let $\cV^{\ell,\b}$ be a Banach manifold consisting of all $C^{\ell,\b}$ Hermitian endomorphisms on $(E,h_0)$ satisfying $\int_X \Tr (v) \omega^n=0$. Define an operator $\cP \colon \cV^{\ell+2,\b} \times \R \to \cV^{\ell,\b}$ by
\[
\cP(v,\e):=\frac{1}{n \e} \bigg( \Im (\e \omega \Id_E+\i F_D)^n-\tan (\Theta_\e) \cdot \Re (\e \omega \Id_E+\i F_D)^n \bigg),
\]
where $D_0$ denotes the induced connection of $h_0$ and $D:=D_0^{\exp(v)}$. A direct computation shows that
\begin{equation} \label{expansion of Re and Im}
\Re (\e \omega \Id_E+\i F_D)^n=F_D^n+O(\e^2), \quad \Im (\e \omega \Id_E+\i F_D)^n=-n \omega \Id_E F_D^{n-1} \e+O(\e^3).
\end{equation}
Thus we obtain
\[
\cP(v,\e)=\frac{[\omega] \cdot \ch_{n-1}(E)}{n \ch_n(E)} F_D^n-\omega \Id_E F_D^{n-1}+O(\e).
\]
This shows that $\cP$ is well-defined and $\cP(0,0)=0$. The linearization of $\cP$ at $(0,0)$ in the direction $v$ is given by
\[
\d|_{(0,0),v} \cP=\frac{\i}{2 \pi} \bigg( c \sum_{k=0}^{n-1} F_{D_0}^k (D_0' D_0''v-D_0'' D_0' v) F_{D_0}^{n-1-k}-\omega \Id_E \sum_{k=0}^{n-2} F_{D_0}^k (D_0' D_0''v-D_0'' D_0' v) F_{D_0}^{n-2-k} \bigg).
\]
So by the simpleness of $E$, we can apply the same argument as in Lemma \ref{small deformation} to find that $\d|_{(0,0)} \cP$ is an isomorphism.

Finally, we will check the dHYM-positivity condition. Let $D_\e$ be the connection constructed as above. Then by using \eqref{expansion of Re and Im} and the convergence $D_\e \to D_0$ in $C^{\ell+2,\b}$, for any $a'' \in \Omega^{0,1}(\End(E))$ we compute
\[
\begin{aligned}
&-\i \Tr \big( \Phi_{\e \omega} (D_\e,a'') \wedge (a'')^\ast \big) \\
&=\i \cos (\Theta_\e) \bigg( \tan(\Theta_\e) \sum_{k=0}^{n-1} F_{D_\e}^k a'' F_{D_\e}^{n-1-k} (a'')^\ast+n \omega \Id_E \sum_{k=0}^{n-2} F_{D_\e}^k a'' F_{D_\e}^{n-2-k} (a'')^\ast \e+O(\e^2) \bigg) \\
&=-\i n \e \cos (\Theta_\e) \bigg( c \sum_{k=0}^{n-1} F_{D_0}^k a'' F_{D_0}^{n-1-k} (a'')^\ast-\omega \Id_E \sum_{k=0}^{n-2} F_{D_0}^k a'' F_{D_0}^{n-2-k} (a'')^\ast+O(\e) \bigg).
\end{aligned}
\]
Since $\cos (\Theta_\e)>0$ and $D_0$ is $J$-positive, this $(n,n)$-form is positive at all points on $X$ where $a'' \neq 0$. In particular, the elliptic regularity shows that $D_\e$ is actually smooth. This completes the proof.
\end{proof}
                                                                                                                                                                                                                                                                                                                                                                                                                                                                                                                                                                                                                                                                                                                                                                                                                                                                                                                                                                                                                                                                                                                                                                                                                                                                                                                                                                                                                                                                                                                                                                                                                                                                                                                                                                                                                                                                                                                                                                                                                                                                                                                                                                                                                                                                                                                                                                                                                                                                                                                                                                                                                                                                                                                                                                                                                                                                                                                                                                                                                                                                                                                                                                                                                                                                                                                                                                                                                                                                                                                                                                                                                                                                                                                                                                                                                                                                                                                                                                                                                                                                 

\end{document}